\documentclass[leqno,11pt]{article}
\usepackage{microtype}

\usepackage{amsmath,amssymb,amsthm}
\usepackage{fancybox,fancyhdr,graphics,epsfig}
\usepackage[usenames,dvipsnames]{color}

\usepackage{ae,aecompl}

\usepackage{amsmath,amsthm,amssymb,mathrsfs} 
\usepackage[T1]{fontenc}
\usepackage{ae,aecompl}
\usepackage{times}
\usepackage[english]{babel}
\usepackage[colorlinks,pagebackref,hypertexnames=false]{hyperref}
\usepackage{verbatim}
\usepackage{bbm,subfigure}
\usepackage[numeric,backrefs]{amsrefs}
\usepackage{bookmark}

\numberwithin{equation}{section}

\newcommand{\rG}{{\rm G}}

\newcommand{\Ad}{\mathrm{Ad}}

\renewcommand{\det}{\mathop\mathrm{det}\nolimits}

\renewcommand{\epsilon}{\varepsilon}

\newcommand{\Hol}{\mathrm{Hol}}

\newcommand{\Ric}{{\rm Ric}}

\newcommand{\diag}{\mathrm{diag}}

\newcommand{\dvol}{\mathop\mathrm{dvol}\nolimits}
\newcommand{\id}{\mathrm{id}}

\renewcommand{\Im}{\mathop{\mathrm{Im}}}

\newcommand{\tr}{\mathop{\mathrm{tr}}\nolimits}

\newcommand{\sign}{\mathrm{sign}}

\def\<{\mathopen{}\left<}
\def\>{\right>\mathclose{}}
\def\({\mathopen{}\left(}
\def\){\right)\mathclose{}}

\newtheorem{theorem}{Theorem}

\newtheorem{corollary}{Corollary}

\newtheorem{definition}{Definition}
\newtheorem{example}{Example}

\newtheorem{lemma}{Lemma}

\newtheorem{proposition}{Proposition}
\newtheorem{remark}{Remark}

\numberwithin{equation}{section}

\newcommand{\Comment}[2][\empty]{\ifthenelse{\equal{#1}{\empty}}{\todo[color=gray!10]{#2}}{\todo[color=gray!10,#1]{#2}}}

\title{Gauge theory on Aloff-Wallach spaces}

\author{ Gavin Ball\footnote{ball@math.duke.edu} \ \ \ \ \ \   Goncalo Oliveira\footnote{oliveira@math.duke.edu} \\ 
Duke University}

\date{October 2016}

\begin{document}
\maketitle

\begin{abstract}
For gauge groups $U(1)$ and $SO(3)$ we classify invariant $G_2$-instantons for homogeneous coclosed $G_2$-structures on Aloff-Wallach spaces $X_{k,l}$. As a consequence, we give examples where $G_2$-instantons can be used to distinguish between different strictly nearly parallel $G_2$-structures on the same Aloff-Wallach space. In addition to this, we find that while certain $G_2$-instantons exist for the strictly nearly parallel $G_2$-structure on $X_{1,1}$, no such $G_2$-instantons exist for the tri-Sasakian one. As a further consequence of the classification, we produce examples of some other interesting phenomena, such as: irreducible $G_2$-instantons that, as the structure varies, merge into the same reducible and obstructed one; and $G_2$-instantons on nearly parallel $G_2$-manifolds that are not locally energy minimizing.
\end{abstract}

\setcounter{tocdepth}{2}
\tableofcontents


\section{Introduction}

A $3$-form $\varphi$ on an oriented 7-dimensional manifold $X^7$ is called a $G_2$-structure, if it takes values in a certain open subbundle $\Lambda^3_+ \subset \Lambda^3$. Such $3$-forms $\varphi$ determine (in a nonlinear way) a Riemannian metric $g_{\varphi}$. In the case when the holonomy of $g_{\varphi}$ lies inside the exceptional Lie group $G_2$, the pair $(X^7 , \varphi)$ is called a $G_2$-manifold, or equivalently $\varphi$ is said to be torsion free. A $G_2$-instanton is a solution to a gauge theoretical equation that can be written in an oriented $7$-dimensional manifold $X^7$ equipped with a $G_2$-structure $\varphi$. Even though $G_2$-instantons have entered the mathematical literature for over 30 years now \cite{Corrigan1983}, it was only in the past few years that the first nontrivial examples appeared, namely in \cite{Walpuski2011}, \cite{SaEarp2015}, \cite{Walpuski2015}, \cite{Clarke14}, \cite{Oliveira2014} and \cite{Lotay2016}. This recent interest in $G_2$-instantons is mostly due to the suggestion by Donaldson-Thomas \cite{Donaldson1998} and Donaldson-Segal \cite{Donaldson2009} that it may be possible to use $G_2$-instantons to construct an enumerative invariant of $G_2$-manifolds. However, adding to the scarcity of examples there are substantial difficulties in constructing such an invariant. In fact, it is conceivable that in order to overcome some of these difficulties one may need to consider $G_2$-structures that are not torsion free. Indeed, there is a larger class of $G_2$-structures, other than just the torsion free class, with respect to which the $G_2$-instanton equation still lies in an elliptic complex. All of this leads us to investigate $G_2$-instantons for these more general $G_2$-structures. For example, one may ask to what extent $G_2$-instantons are persistent under deformations of the $G_2$-structure. In this paper we classify homogeneous (invariant) $G_2$-instantons on an infinite family of $7$-manifolds admitting many such $G_2$-structures. As a consequence we find many examples of new phenomena and are able to investigate what happens to the $G_2$-instantons when the $G_2$-structure varies.

\subsection{Preliminaries}

Let $(X^7, \varphi)$ be a compact, oriented, $7$-manifold equipped with a $G_2$-structure $\varphi$. Let $g_{\varphi}$ be the induced Riemannian metric, $\ast_{\varphi}$ the associated Hodge star, and $\psi$ the $4$-form $\ast_{\varphi} \varphi$. If $G$ is a compact, semisimple Lie group and $P \rightarrow X$ is a principal $G$ bundle, a connection $A$ on $P$ is called a $G_2$-instanton if
\begin{equation}\label{eq:G2_Instanton}
F_A \wedge \psi = 0,
\end{equation}
where $F_A$ denotes the curvature of $A$. When the $G_2$-structure is coclosed, i.e. $d \psi = 0$, the $G_2$-instanton equation lies in an elliptic complex and we shall restrict to this case. The torsion free $G_2$-structures correspond to the special case when $\varphi$ is harmonic. One other special class of coclosed $G_2$-structures are the so called nearly parallel ones, for which $d \varphi=\lambda \psi$ for some $\lambda \neq 0$. If $\varphi$ is nearly parallel, then $g_{\varphi}$ is Einstein with positive scalar curvature. Another perspective on nearly parallel $G_2$-structures is that they are exactly those $G_2$-structures for which the metric cone $(\mathbb{R}^{+} \times X^7, g_C = dr^2 + r^2 g_{\varphi})$ has holonomy contained in $Spin(7).$\\
One other interesting class of connections on a principal bundle over an oriented Riemannian manifold are the Yang-Mills connections. These are defined as the critical points of the Yang-Mills energy
$$E(A)= \frac{1}{2} \int_X \vert F_A \vert^2 ,$$
where we use an $\Ad$-invariant inner product to compute the norm $\vert F_A \vert$. If the $G_2$-structure is either torsion free or nearly parallel, then $G_2$-instantons are also Yang-Mills connections. Moreover, in the torsion free case a simple computation (equation \ref{eq:action}) shows that any $G_2$-instanton actually minimizes the Yang-Mills energy.

\subsection{Summary of the main results}

The Aloff-Wallach space $X_{k,l}$ is defined as the quotient of $SU(3)$ by a $U(1)$ subgroup, whose embedding in $SU(3)$ is determined by two integers $k,l$. On each $X_{k,l}$ we consider a real 4-dimensional family $\mathcal{C}$ of $G_2$-structures, which contains exactly two nearly parallel $G_2$-structures. As proved in \cite{Cabrera1996}, for most $k,l$\footnote{$k \neq \pm l$, $k \neq 0$, $l \neq 0$, $k \neq 2l$, $l \neq -2k$} this family completely exhausts all homogeneous, coclosed $G_2$-structures. In fact, for $k \neq  l$, $k \neq 2l$, $l \neq -2k$, the two nearly parallel $G_2$-structures are in fact strict, meaning that the holonomy of the cone metric $g_C = dr^2 + r^2 g_{\varphi}$ on $\mathbb{R}^{+} \times X_{k,l}$ is exactly $Spin(7)$. These and other facts regarding the geometry of Aloff-Wallach spaces are recalled, with more detail, in section \ref{sec:Aloff-Wallach}. In \ref{ss:InvariantConnections}, we classify invariant connections on each $X_{k,l}$. These results are then used in section \ref{sec:AWG2Instantons} to investigate $G_2$-instantons on the Aloff-Wallach spaces $X_{k,l}$, for $k \neq  l$, $k \neq 2l$, $l \neq -2k$. The remaining case are analysed seperately in section \ref{sec:X11}. We now summarize the main results of those sections starting with the more general situation. In \ref{ss:Abelian_X_(k,l)} we classify invariant Abelian $G_2$-instantons with respect to all $\varphi \in \mathcal{C}$, see theorem \ref{thm:AbelianInstantons}. Here we only state a corollary, which is proved in the third item of remark \ref{rem:Abelian_G2_Instantons_(k,l)}

\begin{theorem}\label{thm:Intro_Reducible_(k,l)}
Let $k \neq  l$, $k \neq 2l$, $l \neq -2k$. For the generic $\varphi \in \mathcal{C}$ there is a unique invariant $G_2$-instanton on any homogeneous complex line bundle over $X_{k,l}$. However, for any such $k,l$, there do exist $\varphi \in \mathcal{C}$ so that any such bundle has a 1-parameter family of invariant $G_2$-instantons.
\end{theorem}

Then, in section \ref{ss:Non_Abelian_X_(k,l)} we focus on invariant $G_2$-instantons with gauge group $SO(3)$. Any homogeneous $SO(3)$-bundle on $X_{k,l}$ can be constructed as $P_{\lambda_n} = SU(3) \times_{U(1)_{k,l}, \lambda_n} SO(3)$, where $\lambda_n : U(1)_{k,l} \rightarrow SO(3)$ is a group homomorphism and the integer $n \in \mathbb{Z}$ denotes the degree of the induced map between maximal tori. We construct explicit maps $\sigma_i : \mathcal{C} \rightarrow \mathbb{R}$, for $i=1,2,3$ whose significance is given in theorem \ref{thm:IrreducibleInstantons}. Below we give a summarized version of that result, when combined with theorem \ref{thm:G2_Instanton_Obstructed}.

\begin{theorem}\label{thm:Intro_Ireducible_(k,l)}
Let $k \neq  l$, $k \neq 2l$, $l \neq -2k$ and $\varphi$ be a homogeneous coclosed $G_2$-structure on $X_{k,l}$. Then, invariant and irreducible $G_2$-instantons on $P_{\lambda_n}$ with respect to $\varphi$ exist if and only if one of the following holds
\begin{enumerate}
\item $n=k-l$ and $\sigma_1(\varphi)>0$,
\item $n=2l+k$ and $\sigma_2(\varphi)>0$, 
\item $n=-l-2k$ and $\sigma_3(\varphi)>0$.
\end{enumerate} 
Moreover, if $\lbrace \varphi(s) \rbrace_{s \in \mathbb{R}} \subset \mathcal{C}$ is a continuous family of $G_2$-structures with $\lbrace  \sigma_1(\varphi(s)) \rbrace_{s \in \mathbb{R}}$ crossing zero once from above, then as $\sigma_1(\varphi(s)) \searrow 0$, two irreducible $G_2$-instantons on $P_{k-l}$ merge and become the same reducible and obstructed $G_2$-instanton for $\sigma_1(\varphi(s)) \leq 0$. Similar statements hold for $\sigma_2$ and $\sigma_3$.
\end{theorem}

To better visualize the content of the last part of this theorem we refer the reader to examples \ref{ex:(k,l)=(1,-1)} and \ref{ex:(k,l)=(1,-5)}, together with their accompanying figures \ref{fig:(k,l)=(1,-1)} and \ref{fig:(k,l)=(1,-5)} respectively. Recall that for $k \neq  l$, $k \neq 2l$, $l \neq -2k$, the Aloff-Wallach space $X_{k,l}$ admits two strictly nearly parallel $G_2$-structures. As an application of theorem \ref{thm:Intro_Ireducible_(k,l)}, in \ref{ss:distinguish} we use $G_2$-instantons to distinguish these for many values of $k,l$. Here we will simply state

\begin{corollary}
The are many examples of $k,l$ as in theorem \ref{thm:Intro_Ireducible_(k,l)}, such that the two inequivalent strictly nearly parallel $G_2$-structures on $X_{k,l}$ always admit invariant and irreducible $G_2$-instantons, but on topologically different $SO(3)$-bundles.
\end{corollary}

In subsection \ref{sss:X1-1} we consider a particular example, namely $X_{1,-1}$. As one other application of theorem \ref{thm:Intro_Ireducible_(k,l)}, we show in \ref{sss:X1-1_G2_Instantons} that $X_{1,-1}$ admits non-Abelian, irreducible $G_2$-instantons for a strictly nearly parallel $G_2$-structure. These $G_2$-instantons are also Yang-Mills, as the $G_2$-structure is nearly parallel, but contrary to the torsion free case we show in \ref{sss:YMX1-1} that they are not energy minimizing (not even locally). We refer the reader to figure \ref{fig:YMfunctional} for a contour plot of the invariant Yang-Mills functional. The results quoted above can be combined into the following

\begin{theorem}
There is a strictly nearly parallel $G_2$-structure $\varphi$ on $X_{1,-1}$ such that: 
\begin{itemize}
\item For gauge group $SO(3)$, there is an irreducible $G_2$-instanton $A$ with respect to $\varphi$.
\item As a Yang-Mills connection, $A$ is not locally energy minimizing.
\end{itemize}
\end{theorem}

We now turn to the case when either $k=l$, or $k=2l$, or $l =-2k$ which was excluded from the previous results. Using the action of the Weyl group of $SU(3)$, and up to coverings, we may assume that $k=l=1$ so we are working on $X_{1,1}$. This case is analyzed in section \ref{sec:X11}. As already remarked before, on $X_{1,1}$ the $G_2$-structures we consider, i.e. those in $\mathcal{C}$, are not all the homogeneous, coclosed ones. Nevertheless, $\mathcal{C}$ does contain nearly parallel $G_2$-structures, inducing $2$ different metrics, one of which is tri-Saskian and the other strictly nearly parallel. There is however, one other homogeneous nearly parallel $G_2$-structure not in $\mathcal{C}$, which a Sasaki-Einstein metric. Our first result for $X_{1,1}$ is theorem \ref{thm:abelianinstantonsX11}, which classifies invariant Abelian $G_2$-instantons with respect to the $\varphi \in \mathcal{C}$. The statement is similar to the case $k \neq l$ in theorem \ref{thm:Intro_Reducible_(k,l)}. As in that case, the generic $\varphi$ admits a unique invariant $G_2$-instanton on any line bundle, however there do exist $\varphi \in \mathcal{C}$ so that the space of invariant $G_2$-instantons on any complex line bundle is $3$-dimensional. In fact, this can be interpreted in light of a more general phenomenon explained in proposition \ref{prop:3_Parameter_Family}. Then, in theorem \ref{thm:irreducibleinstantonsX11} we consider $SO(3)$-bundles over $X_{1,1},$ and for all $\varphi \in \mathcal{C}$ classify irreducible invariant $G_2$-instantons on them. The statement is however very similar to that of theorem \ref{thm:Intro_Ireducible_(k,l)} and we shall omit it in this introduction. Instead, we state here corollary \ref{cor:NoInstantons} which a direct application of that result. Its content being that the existence of invariant $G_2$-instantons, with gauge group $SO(3)$, distinguishes between the $G_2$-structures inducing the tri-Saskian and the strictly nearly parallel metrics.

\begin{theorem}
Let $\varphi^{ts}$ and $\varphi^{np}$ be respectively the $G_2$-structures inducing the tri-Sasakian and the strictly nearly parallel metrics on $X_{1,1}$. Then, there are no irreducible invariant $G_2$-instantons with gauge group $SO(3)$ for $\varphi^{ts},$ but such $G_2$-instantons do exist for $\varphi^{np}$.
\end{theorem}

\subsection*{Acknowledgements}

We would like to thank Robert Bryant, Mark Haskins, Jason Lotay, Henrique S\'a Earp, Mark Stern, and Thomas Walpuski for conversations. In particular, we thank Thomas Walpuski for having kindly suggested the interpretation given in theorem \ref{thm:G2_Instanton_Obstructed} and its visualization through the figures in examples \ref{ex:(k,l)=(1,-1)} and \ref{ex:(k,l)=(1,-5)}.

\section{Gauge theory and coclosed $G_2$-structures}\label{sec:Gauge_Theory_And_Colosed_G2_Structures}

\subsection{Background}\label{ss:Background}

This section starts off in \ref{ss:Coclosed_G2_Strs} with some basic facts about $G_2$-structures\footnote{see \cite{Bryant2006} for more on this and other aspects of $G_2$-structures} {and their torsion}. In \ref{ss:Gauge_Theory} we recall some background on $G_2$-gauge theory. In particular, we identify the coclosed $G_2$-structures, i.e. those for which $d\psi =0$, as the ones for which the $G_2$-instanton equation lies in an elliptic complex. Then, in \ref{ss:Deformation_Theory} we derive some general results on the deformation theory of $G_2$-instantons. These will be used to give an abstract result, proposition \ref{prop:MonLineBundle}, yielding a criteria for when a $G_2$-structure has the property that any circle bundle processes a $G_2$-instanton. As a consequence, in corollary \ref{cor:G2InstLineBundle} this result is applied in the strictly nearly parallel setting.

\subsubsection{Coclosed $G_2$-structures}\label{ss:Coclosed_G2_Strs}

\subsubsection*{Torsion of a $G_2$-structure}

Fernandez-Gray first classified the torsion of $G_2$-structures in \cite{Fernandez1982} by decomposing $\nabla \varphi$ into irreducible $G_2$-representations. The components of $d \varphi$ and $d\psi= d \ast \varphi$ can then be written in terms of those of $\nabla \varphi$. What is nontrivial, but easily checked using the representation theory of $G_2$, is that the converse is also true. Recall that the $2$-forms and $3$-forms decompose into irreducible $G_2$-representations as $\Lambda^2 \cong \Lambda^2_7 \oplus \Lambda^2_{14}$ and $\Lambda^3 = \Lambda^3_1 \oplus \Lambda^3_7 \oplus \Lambda^3_{27}$, where the subscript denotes the dimension of the representation. The Hodge-$\ast$ is an isomorphism of representations and so induces isomorphic decompositions in $\Lambda^4$ and $\Lambda^5$. Using these decompositions the Fernandez-Gray classification can be recast as follows. Given a $G_2$-structure $\varphi$, we have
$$d \varphi = \tau_0 \psi + 3 \tau_1 \wedge \varphi + \ast \tau_3, \ \ d \psi = 4 \tau_1 \wedge \psi + \tau_2 \wedge \varphi $$
for some uniquely determined $\tau_0 \in \Omega^0(X)$, $\tau_1 \in \Omega^1(X)$, $\tau_2 \in \Omega^2_{14}(X)$ and $\tau_3 \in \Omega^3_{27}(X)$. Of special interest to us will be the case when the $G_2$-structure is coclosed, i.e. when $d \psi = d \ast \varphi =0$. Then, $\tau_1 = \tau_2 =0$ and $d \varphi = \tau_0 \psi + \ast \tau_3$.\\
For future reference we shall use $\pi_i$ for $i=1,7,14,27$ to denote the projection onto an $i$-dimensional irreducible representation. For example, if $\omega$ is a two form we shall denote by $\pi_7 (\omega)$ the component of $\omega \in \Lambda^2_7$.

\subsubsection*{Nearly parallel $G_2$-structures}

We now turn to the definition of nearly parallel $G_2$-structures. Given a closed, oriented, $7$-manifold $(X^7, \varphi)$ equipped with a $G_2$-structure, its metric cone $(\mathbb{R}^+ \times X^7 , g_C=dr^2 + r^2 g_{\varphi})$ comes equipped with a $Spin(7)$-structure determined by $\Omega= r^3 dr \wedge \varphi + r^4 \psi$. From the Riemannian holonomy point of view, if $g_C$ is nonsymmetric its holonomy is one of the groups in the following ascending chain
$$\lbrace 1 \rbrace \subset Sp(2) \subset SU(4) \subset Spin(7) \subset SO(8).$$
Equivalently, thinking of $G_2$ has the group stabilizing a nonvanishing spinor in $7$-dimensions, the groups above are possible stabilizers of spinors in $8$-dimensions and each is determined by the number of linearly independent spinors fixed. In the language of spinors, the condition that the holonomy reduces to one of the groups above is then that the respective spinors are parallel.

\begin{definition}\label{def:NearlyParallel}
As a Riemannian manifold $(X^7, g_{\varphi})$ is said to be nearly parallel if $\Hol(g_{C}) \subseteq Spin(7)$, and if $\Hol(g_C)$ is $Sp(2)$, $SU(4)$, or $ Spin(7)$, $g_{\varphi}$ is said to be tri-Sasakian, Sasaki-Einstein, or strictly nearly parallel, respectively. Similarly, we shall say that the $G_2$-structure $\varphi$ is strictly nearly parallel if $\Hol(g_{\varphi}) = Spin(7)$.
\end{definition}

Strictly nearly-parallel $G_2$-structures, can also be equivalently characterized in terms of differential forms. Notice that, given a metric $g$ on $X^7$, the cone metric $g_C = dr^2 + r^2 g$ has holonomy contained in $Spin(7)$ if and only if there is a compatible $G_2$-structure $\varphi$, such that the $4$-form $\Omega= r^3 dr \wedge \varphi + r^4 \psi$ is closed. That is the case if and only if $d \varphi = 4 \psi$, which up to scaling and changing the orientation can be written as
\begin{equation}\label{eq:StrictlyNearlyParallel}
d \varphi = \lambda \psi,
\end{equation}
for some $\lambda \in \mathbb{R} \backslash \lbrace 0 \rbrace$. Notice that, as $\psi$ is exact, this implies $d \psi=0$ and from the point of view of torsion of $G_2$-structures $\varphi$ is coclosed, meaning that all $\tau_1, \tau_2, \tau_3$ vanish and $\tau_0= \lambda$ is the only nonzero component. As $\tau_0$ is the torsion component living in the smallest irreducible representation, we may think of strictly nearly parallel $G_2$-structures as the closer to become parallel.

\begin{remark}\label{rem:ParallelAndNearlyParallel}
In fact, notice that if we require that $d \psi=0$ separately and allow $\lambda$ to vanish, then equation \ref{eq:StrictlyNearlyParallel} also include the torsion free case. This shall be useful as some arguments used for strictly nearly parallel $G_2$-structures also work in the torsion free case. 
\end{remark}

In \cite{Friedrich1997}, the authors classify homogeneous nearly parallel $G_2$-manifolds, and give a construction of strictly nearly parallel $G_2$-structures starting from tri-Sasakian manifolds. We shall recall and use this construction in section \ref{sec:NearlyParallel}.

\subsubsection{Gauge Theory}\label{ss:Gauge_Theory}

Let $G$ be a compact semisimple Lie group and $P$ a principal $G$-bundle over a manifold $X$, equipped with a $G_2$-structure $\varphi$. Recall that a connection $A$ on $P$ is called a $G_2$-instanton if $F_A \wedge \psi=0$, equivalently if $\pi_7 (F_A)=0$, or if the following analogue of anti-self-duality holds:

\begin{equation}
\ast F_A = - F_A \wedge \varphi.
\end{equation} 

On the other hand, a connection $A$ is said to be Yang-Mills if it is a critical point of the Yang-Mills energy
\begin{equation}
E(A) = \frac{1}{2} \int_X \vert F_A \vert^2 \dvol_g,
\end{equation}
and so satisfies the Yang-Mills equation $d_A^* F_A=0$, which together with the Bianchi identity $d_A F_A=0$ forms a second order elliptic system for the connection (up to gauge). $G_2$-instantons satisfy a first order equation which in this generality need not imply they are Yang-Mills connections. Nevertheless we have the following folklore result, which in the nearly parallel case is due to Harland and N\"olle in \cite{Harland2011}.

\begin{proposition}\label{prop:NearlyG2YM}(\cite{Harland2011})
If the $G_2$-structure is either parallel or nearly parallel, then any $G_2$-instanton is a Yang-Mills connection. 
\end{proposition}
\begin{proof}
If the $G_2$-structure is either parallel or strictly nearly parallel, $d \psi=0$ and $d \varphi = \lambda \psi$ for some $\lambda \in \mathbb{R}$, as in remark \ref{rem:ParallelAndNearlyParallel}. Then, if $A$ is a $G_2$-instanton, $\ast F_A = - F_A \wedge \varphi$ and so
$$d_A \ast F_A = - d_A (F_A \wedge \varphi)= \lambda F_A \wedge \psi =0.$$
where in the last equality we use the Bianchi identity and $d\varphi = \lambda \psi$.
\end{proof}

The Yang-Mills energy can be equivalently written as 
\begin{equation}\label{eq:action}
E(A)= -\frac{1}{2} \int_X \langle F_A \wedge F_A \rangle \wedge \varphi + \frac{1}{2} \Vert F_A \wedge \psi  \Vert^2_{L^2}  .
\end{equation}
In particular, if $\varphi$ is torsion free, then the first term is topological and $G_2$-instantons minimize the Yang-Mills energy. It is then a natural question to ask if the same must hold for nearly parallel $G_2$-structures. We shall show in example \ref{ex:UnstableS7} that is not the case, by providing an example of a nearly parallel $G_2$-structure, together with a $G_2$-instanton which is unstable as a Yang-Mills connection.

\begin{remark}
The variation of the Yang-Mills functional at a connection $A$ is
\begin{equation}
\delta^2 E_A (a) = \frac{d^2}{ds^2} |_{s=0} E(A+sa) = \int_X \vert d_A a \vert^2 - \langle [a \wedge a] , F_A \rangle,
\end{equation}
and so we may instead think of the second order operator $H=d_A^* d_A a - \ast [a \wedge \ast F_A]$.
\end{remark}

When the $G_2$-structure $\varphi$ is coclosed the $G_2$-instaton equation lies on the elliptic complex
\begin{equation}\label{eq:ComplexInstanton}
\Omega^0(X , \mathfrak{g}_P) \xrightarrow{-d_A \cdot } \Omega^1(X , \mathfrak{g}_P) \xrightarrow{d_A \cdot \wedge \psi} \Omega^6(X \mathfrak{g}_P) \xrightarrow{d_A} \Omega^7(X, \mathfrak{g}_P).
\end{equation}
Hence, in the coclosed case the $G_2$-instanton equation is elliptic modulo gauge (rather than overdetermined). From now on we shall suppose this is the case.

\begin{remark}\label{rem:NoMonopoles}
\begin{enumerate}
\item The reason why the $G_2$-instaton equation is consistent in the torsion free case can be interpreted as follows. The $G_2$-monopole equation
$$\ast \nabla_A \Phi = F_{A} \wedge \psi,$$
is always elliptic modulo gauge. Moreover, if $\varphi$ is coclosed, then the monopole equation, $d \psi=0$ and the Bianchi identity $d_A F_A=0$, give $\Delta_A \Phi=0$. We can then compute $\Delta \vert \Phi \vert^2 = - 2 \vert \nabla_A \Phi \vert^2 \leq 0$,
and the maximum principle implies that $\vert \Phi \vert^2$ is constant. Then $\vert \nabla_A \Phi \vert^2$ must vanish, and the monopole equation reduces to the $G_2$-instaton equation. Furthermore, the fact that $\nabla_A \Phi=0$, implies that if $\Phi \neq 0$, and $G$ is semisimple, then $A$ must be reducible.

\item To conclude this remark we point out that if the $G_2$-structure $\varphi$ is not coclosed one may ask a similar questions to those answered in this paper, but for $G_2$-monopoles rather than $G_2$-instantons.
\end{enumerate}
\end{remark}

In particular, if $(X, \varphi)$ is a compact irreducible $G_2$ manifold, i.e. the holonomy of the metric $g_{\varphi}$ induced by $\varphi$ is equal to $G_2$, any harmonic $2$-form can be shown to be of type $\Lambda^2_{14}$ and so if $F \in \Omega^2(X)$ is harmonic and has integer periods, it defines the curvature of a connection on a line bundle whose first Chern class is $[F]/ 2\pi i$. Still in the torsion free case, Thomas Walpuski, in \cite{Walpuski2011} and \cite{Walpuski2015} using the results of \cite{SaEarp2015}, constructed the only known examples of non-Abelian $G_2$-instantons on compact, irreducible, $G_2$-manifolds. There are also examples in the noncompact case, see \cite{Clarke14}, \cite{Oliveira2014} and \cite{Lotay2016}.\\

\subsubsection{Deformation theory and Abelian $G_2$-instantons}\label{ss:Deformation_Theory}
 
The main idea for this approach to the deformation theory comes from remark \ref{rem:NoMonopoles}. This suggests that given a coclosed $G_2$-structure, instead of studying the deformation theory of an irreducible $G_2$-instanton $A$ we may instead study that of a $G_2$-monopole $(A, \Phi)$ with $\Phi=0$. Before restricting to that case suppose for now that $\Phi \neq 0$, then the relevant elliptic complex is
\begin{equation}\label{eq:ComplexMonopole}
\Omega^0(X , \mathfrak{g}_P) \xrightarrow{d_1} \Omega^1(X , \mathfrak{g}_P) \oplus \Omega^0(X , \mathfrak{g}_P) \xrightarrow{d_2} \Omega^1(X , \mathfrak{g}_P),
\end{equation}

with $d_1(\phi)=(-d_A \phi , [ \phi , \Phi ] )$ and $d_2(a, \phi)=\ast (d_A a \wedge \psi) - [ a , \Phi ] -d_A \phi$. Equivalently, we can consider the elliptic operator $d_1^* \oplus d_2 :  \Omega^1(X , \mathfrak{g}_P) \oplus \Omega^0(X , \mathfrak{g}_P) \rightarrow  \Omega^1(X , \mathfrak{g}_P) \oplus \Omega^0(X , \mathfrak{g}_P)$
$$(d_1^* \oplus d_2)(a, \phi) = (\ast (d_A a \wedge \psi)-d_A \phi , -d_A^* a) + ( [\Phi , a] , [\Phi, \phi]),$$
which is self adjoint when $\varphi$ is coclosed. The following result, which is a consequence of remark \ref{rem:G2fromNK}, shows that in the coclosed case any infinitesimal monopole deformation of a $G_2$-instanton is actually an infinitesimal instanton deformation. This fully justifies studying the deformation theory of the complex \ref{eq:ComplexMonopole}. 

\begin{proposition}
Let $A$ be an irreducible $G_2$-instanton with respect to a coclosed $G_2$-structure on a closed manifold. Then, if $(a, \phi) \in \ker(d_2)$, where $d_2$ is the operator associated with $(A, 0)$ we have $\phi=0$.
\end{proposition}
\begin{proof}
Let $(a, \phi) \in \ker(d_2)$. Then, $d_A \phi=\ast (d_A a \wedge \psi)$, and $d_A^* a=0$. Combining these and using that $\psi$ is closed, we compute 
$$d_A^* d_A \phi = - \ast d_A (d_A a \wedge \psi) = - \ast [F_A \wedge a] \wedge \psi.$$
This vanishes as $A$ is a $G_2$-instanton and so $F_A \wedge \psi=0$. Then taking the inner product with $\phi$ gives $d_A \phi=0$ and so $\phi$ must vanish as $A$ is assumed to be irreducible.
\end{proof}

Next we shall study the operator $d_1^* \oplus d_2$ for the trivial connection $A=d$. It will be used later to give an existence result for $G_2$-instantons in the Abelian case

\begin{lemma}\label{lem:DiracSurjective}
Let $L$ be the operator 
$$L: L^{2,1} (\Lambda^0 \oplus \Lambda^1) \rightarrow L^{2} (\Lambda^0 \oplus \Lambda^1),$$
given by $L(f , a )= (-d^* a , -df + \ast (da \wedge \psi))$. Its cokernel can be identified with those $(g,b)\in \Omega^0(X) \oplus \Omega^1(X)$ such that $g$ is constant and $b$ is a coclosed $1$-form satisfying $d(b \wedge \psi)=0$.\\
In particular, if $(X, \varphi)$ has the property that there are no coclosed $1$-forms $b$, such that $d(b \wedge \psi)=0$, then $L$ is surjective onto $\Omega_0^0 (X) \oplus \Omega^1(X)$, where $\Omega_0^0 (X)$ denotes the functions with zero average on $X$. 
\end{lemma}
\begin{proof}
We shall identify the cokernel of $L$ with the kernel of its formal adjoint $L^*$, using the $L^2$-inner product. Then, one computes that $L^*(g , b )= (-d^* b , -dg + \ast d(b \wedge \psi))$, and so
\begin{eqnarray}\nonumber
LL^* (g , b) & = & \left( \Delta g , dd^* b \right) + \left( 0 , \ast (d \ast d(b \wedge \psi) \wedge \psi) \right).
\end{eqnarray}
By taking the $L^2$ inner product with $(g,b)$ and using Stokes' theorem we obtain
\begin{eqnarray}\nonumber
 \langle (g,b) , LL^* (g , b) \rangle_{L^2} & = & \Vert dg \Vert^2_{L^2} + \Vert d^*b \Vert^2_{L^2} + \langle b , \ast (d \ast d(b \wedge \psi) \wedge \psi) \rangle_{L^2} \\ \nonumber
& = & \Vert dg \Vert^2_{L^2} + \Vert d^*b \Vert^2_{L^2} + \int_X  b \wedge d \ast d(b \wedge \psi) \wedge \psi \\ \nonumber
& = & \Vert dg \Vert^2_{L^2} + \Vert d^*b \Vert^2_{L^2} + \int_X  d (b \wedge \psi ) \wedge \ast d(b \wedge \psi)  \\ \nonumber
& = & \Vert dg \Vert^2_{L^2} + \Vert d^*b \Vert^2_{L^2} +  \Vert d (b \wedge \psi ) \Vert^2_{L^2}  .
\end{eqnarray}
Hence if $(g,b)$ is in the kernel of $L^*$, then also $L L^*(g,b)=0$ and the computation above shows that $dg=d^*b = d(b \wedge \psi)=0$.
\end{proof}

The following result gives a criteria for an abstract construction of Abelian $G_2$-instantons.

\begin{proposition}\label{prop:MonLineBundle}
Suppose $(X, \varphi)$ has no nonzero coclosed $1$-forms $b$ such that $d(b \wedge \psi)=0$ and $L$ is a complex line bundle over $X$, then there is a monopole $(\phi, A)$ on $L$.\\
Moreover, if $\varphi$ is coclosed, then any such monopole is actually a $G_2$-instanton and it is unique.
\end{proposition}
\begin{proof}
To prove this we start with any connection $A_0$ on $L$ and look for $(\phi, a) \in \Omega^0 (X) \oplus \Omega^1(X)$ such that $(\phi, A_0 + a)$ solves the monopole equation $d \phi = \ast( F_{A_0+a} \wedge \psi)$. This can be rewritten in the form
$$-d \phi + \ast( d_{A_0} a \wedge \psi) = - \ast (F_{A_0} \wedge \psi),$$
and so, together with the gauge fixing condition $-d_{A_0}^* a=0$ it suffices to solve the equation $L(\phi, a)= (0, -\ast (F_{A_0} \wedge \psi))$. Since $0$ certainly has vanishing average, by lemma \ref{lem:DiracSurjective} this right hand side lies in the image of the operator $L$ and we can find $(\phi , a)$ such that $(\phi, A_0 + a)$ is a monopole on $L$.\\
The fact that in the coclosed case the monopoles are actually instantons follows from the discussion in remark \ref{rem:NoMonopoles}. The uniqueness follows from the fact that in this case the operator $L$ is formally self-adjoint. Then, the fact that when restricted to $\Omega_0^0 (X) \oplus \Omega^1(X)$ it has no kernel shows it is actually an isomorphism from $L^{2,1}$ to $L^2$.
\end{proof}

As a particular example of how to apply the previous result we shall now consider the strictly nearly parallel case.

\begin{corollary}\label{cor:G2InstLineBundle}
Let $(X,\varphi)$ be a strictly nearly parallel $G_2$-manifold. Then, for any $\alpha \in H^2(X, \mathbb{Z})$, there is a unique $G_2$-instanton on the complex line bundle $L$ with $c_1(L)= \alpha$.
\end{corollary}
\begin{proof}
We start by showing that in the strictly nearly parallel case we are in the setup of proposition \ref{prop:MonLineBundle}. Suppose $b \in \Omega^1(X)$ is such that $d^* b=0$ and $d(b \wedge \psi)=0$. First notice that in this case $\psi$ is exact and so closed, so the second equation can be written $d b \wedge \psi=0$. This shows that $3 d^7 b = \ast ( \ast (db \wedge \psi) \wedge \psi )=0$, which we can rewrite as $0 = 3d^7 b = d b + \ast (db \wedge \varphi)$. Hence, taking $d^*$ of this equation, we find
$$0=3 d^* d^7 b = d^* db + \ast (db \wedge d \varphi) = d^* db + \lambda \ast (db \wedge \psi) = d^* db,$$
where we have used that $d \varphi = \lambda \psi$ and $db \wedge \psi=0$ by hypothesis. Putting this together with $d^* b=0$, we conclude that $\Delta b=0$ and so $b$ is a harmonic $1$-form. However, strictly nearly parallel $G_2$-structures are Einstein with positive constant, and so have positive Ricci curvature. It then follows from the B\"ochner formula and Myers theorem that $b=0$. We are then in position to apply proposition \ref{prop:MonLineBundle} and conclude that there is a $G_2$-instanton on any line bundle over $X$. 
\end{proof}

\begin{remark}
\begin{enumerate}
\item One may wonder if the previous corollary extends from nearly parallel to a more general class of $G_2$-structures. We will see in the second bullet of the first item of theorem \ref{thm:X11} examples of coclosed $G_2$-structures where we do not have uniqueness of Abelian $G_2$-instantons. See also the second item in remark \ref{rem:X11}.

\item The previous proof works equally well for torsion free, irreducible $G_2$-manifolds, i.e. those with holonomy equal to $G_2$. In that case $\lambda=0$ and $\Ric=0$, but the irreducibility shows that there are no harmonic $1$-forms. 

\item In fact, the previous corollary has the following consequence. Any harmonic $2$-form on a strictly nearly parallel $G_2$-manifold, must lie on $\Lambda^2_{14}$. As proved by Lorenzo Foscolo, a similar result holds for nearly K\"ahler manifolds, see theorem 3.23 in \cite{Foscolo2016}. 
\end{enumerate}
\end{remark}

\subsubsection{$S^1$-invariant $G_2$-instantons}\label{ss:S1_Invariant_G2_Instantons}

In section \ref{sec:Aloff-Wallach} we will be interested in studying $G_2$-instantons that are invariant under the action of a group which acts transitively. Here we make a detour into $U(1)$-invariant $G_2$-instantons, on $U(1)$-invariant $G_2$-structures. We include this section so we can refer to its main computation in the proof of theorem \ref{thm:AbG2instYM}. Let $V$ be the infinitesimal generator of a $U(1)$-action preserving a coclosed $G_2$-structure, i.e. $L_V \varphi=0$ and so $L_V \psi=0$ as well. Now let $\eta \in \Omega^1(X^7)$ be the unique connection form on the circle bundle $X^7 \rightarrow M^6=X^7 / S^1$, such that $\eta(V)=1$ and $\eta \vert_{V^{\perp}} =0$. Then, the equation $L_V \psi=0$, together with $d \psi=0$ shows that both, $\iota_V \psi$ and $\psi- \eta \wedge \iota_V \psi$ are $V$-basic, and so are pulled back from $M^6$. We may then write
$$\psi= - \eta \wedge \Omega_1+ \tau,$$
where $\Omega_1$ and $\tau$ are $-\iota_V \psi$ and $\psi - \eta \wedge \iota_V \psi$ respectively. Moreover, the equations $L_V \psi=0$, and $d \psi=0$ further imply
$$d \Omega_1=0, \ \ d\eta \wedge \Omega_1= d \tau .$$
In fact, since $\psi= \ast \varphi$ is the $4$-form associated with the $G_2$-structure $\varphi$, there must further exist $V$-semibasic forms $\omega \in \Omega^2(X)$ and $\Omega_2 \in \Omega^3(X)$, such that $\varphi= \eta \wedge \omega + \Omega_2$ and $\tau = \frac{\omega^2}{2}$. In the setting we will be interested in, all the relevant principal bundles $P$ over $X$ can actually be regarded as bundles pulled back from $M$. Hence, if $A$ is a connection on $P$ over $X$ and $a'$ a connection pulled back from $M$ to $X$, we have that $A-a' \in \Omega^0(X, \Lambda^1 \otimes \mathfrak{g}_P)$. Then, splitting $\Lambda^1 = \langle \eta \rangle \oplus \langle \eta \rangle^{\perp}$ we can write $A-a'=a'' + \phi \otimes \eta$, where $a'' \in \Omega^0(X, \langle \eta \rangle^{\perp} \otimes \mathfrak{g}_P)$ and $\phi \in \Omega^0(X, \mathfrak{g}_P)$. Defining now $a=a'+a''$, the connection $A$ may written as $A=a+\phi \otimes \eta$. Its curvature may then be computed to be $F_A = F_a + d_a \phi \wedge \eta + \phi \otimes d \eta$, and $F_a = F_a^{\perp} -L_V a \wedge \eta$ with $F_a^{\perp}$ semibasic. However, as the connection is assumed to be invariant under the action generated by $V$, $L_V a=0$ and $F_a=F_a^{\perp}$ is actually $V$-basic. We then compute 
\begin{eqnarray}\nonumber
F_A \wedge \psi & = & ( F_a + d_a \phi \wedge \eta + \phi \otimes d \eta) \wedge (- \eta \wedge \Omega_1+ \tau) \\ \nonumber
& = & - \eta \wedge ( F_a \wedge \Omega_1 + \phi \otimes d \eta \wedge \Omega_1 + d_a \phi \wedge \tau  ) + (F_a + \phi \otimes d \eta ) \wedge \tau ,
\end{eqnarray} 
and so the $G_2$-instanton equation amounts to
\begin{equation}\label{eq:U(1)_Invariant_G2_Instanton}
 (F_a + \phi \otimes d \eta ) \wedge \Omega_1  + d_a \phi \wedge \tau  =0 \ \ (F_a + \phi \otimes d \eta ) \wedge \tau = 0 .
\end{equation}

\subsection{Examples from tri-Sasakian geometry}\label{sec:NearlyParallel}

We start this subsection with a brief discussion of tri-Sasakian geometry, following the nice review paper \cite{Boyer1999}. Then, starting from a tri-Sasakian manifold we construct a family of coclosed $G_2$-structures containing a strictly nearly parallel structure, and give some existence results for $G_2$-instantons, see propositions \ref{prop:3_Parameter_Family}, \ref{prop:G2InstFromQK}, and \ref{prop:G2InstFromNK}.\\
A tri-Sasakian $7$-manifold may be equivalently defined as a Riemannian $7$-manifold $(X^7, g_7)$ equipped with a $3$-orthonormal vector fields $\lbrace \xi_i \rbrace_{i=1}^3$ satisfying $[\xi_i, \xi_j] = \epsilon_{ijk} \xi_k$. Any tri-Sasakian $X$ is quasi-regular in the sense that the vector fields $\lbrace \xi_i \rbrace_{i=1}^3$ generate a locally free $SU(2)$ action. The space of leafs $Z^4$, equipped with the Riemannian metric $g_Z$ such that $\pi:X^7 \rightarrow Z^4$ is an orbifold Riemannian submersion, has the structure of an anti-self-dual, Einstein orbifold with scalar curvature $s>0$. Let $g_7$ be the tri-Sasakian metric on $X^7$ and regard $\pi: X^7 \rightarrow Z^4$ as an $SU(2)$, or $SO(3)$ (orbi) bundle of frames of $\Lambda^2_+ Z^4$. The Levi-Civita connection of $Z^4$ equips it with a connection $\eta=\eta_i \otimes T_i \in \Omega^1(X^7, \mathfrak{so}(3))$, where the $T_i$ are a standard basis of $\mathfrak{so}(3)$ satisfying $[T_i,T_j]=2\epsilon_{ijk}T_k$. This has the property that the $\eta$-horizontal forms $\omega_i$ defined by
$$F_{\eta}= d\eta + \frac{1}{2}[\eta \wedge \eta] =  \frac{s}{24} \  \omega_i \otimes T_i,$$
form an orthogonal basis of $(\Lambda^2_+ \ker (\eta), g_7 \vert_{\ker (\eta)})$ with $\vert \omega_i \vert = \sqrt{2}$ and $s \in \mathbb{R}^+$. We further remark that the metric $g_7$ can be written as
$$g_7 = \eta^i \otimes \eta^i + \pi^* g_Z .$$

\begin{remark}\label{rem:TriSasakian_Sp(2)}
To make a connection with the holonomy point of view used in definition \ref{def:NearlyParallel} we remark that the 2-forms $\overline{\omega}_i = r dr \wedge \eta_i + \frac{r^2}{2} d \eta_i$ equip the cone $(\mathbb{R}^+_r \times X, g_C = dr^2 + r^2 g_7)$ with a compatible, torsion free $Sp(2)$ structure.
\end{remark}

The strictly nearly parallel $G_2$-structure $\varphi$ constructed in \cite{Friedrich1997}, determines a Riemannian metric $g_{\varphi}$ which is a squash of the tri-Sasakian metric $g_7$. We shall consider the $1$-parameter family of $G_2$-structures $\lbrace \varphi_t \rbrace_{t \in \mathbb{R} \backslash 0}$ such that
\begin{equation}\label{eq:G2_structure_t}
\varphi_t = t^3 \eta_1 \wedge \eta_2 \wedge \eta_3 + t \frac{s}{48} \left( \eta_1 \wedge \omega_1 +  \eta_2 \wedge \omega_2 +  \eta_3 \wedge \omega_3 \right), 
\end{equation}
which determines $g_{\varphi_t}= t^2 (\eta_1^2 + \eta_2^2 + \eta_3^2)+ \pi^* g_Z$ and
$$\psi_t =  \frac{1}{6} \left(\frac{s}{48}\right)^2 \omega_i \wedge \omega_i + t^2 \frac{s}{48} \left( \eta_1 \wedge \eta_2 \wedge \omega_3 + \eta_2 \wedge \eta_3 \wedge \omega_1 + \eta_3 \wedge \eta_1 \wedge \omega_2 \right).$$
Recall that, up to scaling, the condition that $\varphi_t$ be nearly parallel can be written as $d \varphi_t = \lambda \psi_t$ for some constant $\lambda>0$. In our case we can easily compute from $\frac{s}{24} \omega_i = d\eta^i + \epsilon_{ijk} \eta^j \wedge \eta^k$ that
\begin{eqnarray}\nonumber
d \varphi_t & = & t (t^2+1) \frac{s}{24} \left( \eta_1 \wedge \eta_2 \wedge \omega_3 +  \eta_2 \wedge \eta_3 \wedge \omega_1 + \eta_3 \wedge \eta_1 \wedge \omega_2 \right) \\ \nonumber
& &  + 2 t \left(\frac{s}{48} \right)^2 \left( \omega_1 \wedge \omega_1 +  \omega_2 \wedge \omega_2 + \omega_3 \wedge \omega_3 \right).
\end{eqnarray}
Then, the equation $d \varphi_t = \lambda \psi_t$ becomes $12 t = \lambda$ and $t^2 + 1 = 2 \lambda t$, which has the solutions $t= 1/ \sqrt{5}$, $\lambda=12 / \sqrt{5}$ and $t= -1/ \sqrt{5}$, $\lambda=-12/ \sqrt{5}$. Note that we can scale $\lambda$ by scaling the metric and change the sign of $\lambda$ by changing the orientation. Conversley, it is possible to show that given a positive Einstein, anti-self-dual orbifold $(Z,g_Z)$ there is an $SO(3)$, or $SU(2)$ bundle $\pi: X^7 \rightarrow Z$ equipped with a tri-Sasakian structure, \cite{Boyer1999}, and so has a strictly nearly parallel $G_2$-structure as above. We further remark that this converse statement may however produce non-smooth $X^7$. We are now in position to give some examples of $G_2$-instantons, starting first with $SU(2)$-invariant instantons and then with $S^1$-invariant examples.

\begin{proposition}\label{prop:3_Parameter_Family}
For any $b_1,b_2,b_3 \in \mathbb{R}$ the $1$-form $\eta=b_1 \eta_1 + b_2 \eta_2 + b_3 \eta_3$ equips the trivial complex line bundle over $X^7$ with a $G_2$-instanton with respect to $\varphi_{\sqrt{s/48}}$.\\ Moreover, if $L$ is a complex line bundle over $X^7$ admitting a $G_2$-instanton with respect to $\varphi_{\sqrt{s/48}}$, then $L$ actually has a real $3$-parameter family of $G_2$-instantons.
\end{proposition}
\begin{proof}
The connection $\eta=b_1 \eta_1 + b_2 \eta_2 + b_3 \eta_3$ is not only $S^1$-invariant but also $SU(2)$-invariant. Its curvature is $d \eta$ and to show that $d \eta \wedge \psi_{s/48}=0$ it is enough to show that $d \eta_1 \wedge \psi_{s/48}=0,$ the $d\eta_2$ and $d \eta_3$ equations are dealt with similarly. So we compute
\begin{eqnarray}\nonumber
d \eta_1 \wedge \psi_{t} & = & (\omega_1 - 2 \eta_{23}) \wedge \left( \frac{1}{6} \left(\frac{s}{48}\right)^2 \omega_i \wedge \omega_i + t^2 \frac{s}{48} \left( \eta_{23} \wedge \omega_1 +  \ldots \right) \right) \\ \nonumber
& = & \left( - \left( \frac{s}{48} \right)^2 + t^2  \frac{s}{48}  \right) \eta_{23} \wedge \omega_1 \wedge \omega_1,
\end{eqnarray}
which vanishes if and only if $t= \sqrt{s/48}$.\\
The second part of the theorem follows immediately from the fact that the $G_2$-instanton equation is linear in the Abelian case. 
\end{proof}

\begin{proposition}\label{prop:G2InstFromQK}
Let $A$ be an anti-self-dual connection on a bundle over a positive, anti-self-dual, Einstein, orbifold $(Z, g_Z)$. Then, for all $t>0$ the $G_2$-structure $\varphi_t$ is coclosed and
\begin{itemize}
\item $\pi^* A$ is a $G_2$-instanton on $X^7$ with respect to $\varphi_t$. In particular, $\pi^* A$ is a $G_2$-instanton for the strictly nearly parallel $G_2$-structure $\varphi_{1/\sqrt{5}}$.
\item $\pi^* A$ is Yang-Mills with respect to $\varphi_t$.
\end{itemize}
\end{proposition}
\begin{proof}
The fact that the $G_2$-structure $\varphi_t$ is coclosed for any $t>0$ follows from computing that $d \psi_t=0$. This follows easily from the fact that $\eta_1 \wedge \eta_2 \wedge \eta_3$ is closed (in fact exact) and that each $\omega_i \wedge \omega_i$ is closed as well, since $d \omega_i= 2 \epsilon_{ijk} \omega_j \wedge \eta_k$ and $\omega_i \wedge \omega_j=0$ for $i \neq j$. This shows that $\varphi_t$ is coclosed.\\
We start by proving the first bullet in the statement, i.e. that $A$ pulls back to a $G_2$-instanton, let $F_A$ denote the curvature of $A$, which is anti-self-dual by hypothesis. Hence, as $\pi$ is a Riemannian submersion with respect to all $g_{\varphi_t}$, $\pi^* F_A \wedge \omega_i=0$ for $i=1,2,3$, it is then easy to check that $\pi^* F_A \wedge \psi_t=0$.\\
Now we prove the second bullet in the proposition. To ease notation denote by $A$ the pullback of such a anti-self-dual connection. Then $F_A$ takes values in $\Lambda^2_- \otimes \mathfrak{g}_P$ and we compute
\begin{eqnarray}\label{eq:SelfDual_to_YM}
d_A \ast_{g_{\varphi_t}} F_A = - d_A (F_A \wedge \varphi_t) = - F_A \wedge d \varphi_t  .
\end{eqnarray}
However, $d \varphi_t = t^3 d (\eta_{123}) + t d\eta_i \wedge \omega_i + t \eta_i \wedge d \omega_i$ and it is easy to check that $d (\eta_{123})= \omega_1 \wedge \eta_{23}+ c. p. $ and $\eta_i \wedge d \omega_i = 2(\eta_{13} \wedge \omega_2 - \eta_{12} \wedge \omega_3) + c.p.$, where $c.p.$ denotes cyclic permutations. Putting all these together we have
$$d \varphi_t = t \omega_i \wedge \omega_i + (t^2-6t) \left( \omega_1 \wedge \eta_{23} + \omega_2 \wedge \eta_{31} + \omega_3 \wedge \eta_{12} \right). $$
As $F_A$ is anti-self-dual $F_A \wedge \omega_i=0$, hence inserting $d \varphi_t$ into the equation \ref{eq:SelfDual_to_YM} we conclude that $d_A \ast F_A =0$ and $A$ is Yang-Mills.
\end{proof}

\begin{remark}\label{rem:TriSasakian_G2}
\begin{enumerate}
\item Recall that the map from $G_2$-structures to the metrics associated with them is highly nonlinear and certainly not injective. For instance, the tri-Sasakian metric $g_7$, besides being associated with $\varphi_1$, is also associated with the $G_2$-structure
$$\varphi^{ts}= - \eta_{123} + \frac{s}{48} \left( \eta_1 \wedge \omega_1+  \eta_2 \wedge \omega_2+  \eta_3 \wedge \omega_3 \right) ,$$
whose associated $4$-form $\psi^{ts}= \ast_{g_7} \varphi^{ts}$ is given by
$$ \psi^{ts}= \frac{1}{2} \left( \frac{s}{48} \right)^2 \omega_1 \wedge \omega_1 - \frac{s}{48} \left( \omega_1 \wedge \eta_{23}+ \omega_2 \wedge \eta_{31}- \omega_3 \wedge \eta_{12} \right)$$
and satisfies $d \varphi^{ts}=4 \psi^{ts}$. This can be related with the torsion free $Sp(2) \subset Spin(7)$ structure $(\overline{\omega}_1, \overline{\omega}_2, \overline{\omega}_3)$ on the cone, introduced in remark \ref{rem:TriSasakian_Sp(2)}. The $4$-form 
$$\Theta= \frac{1}{2} \left( \overline{\omega}_1^2 + \overline{\omega}_2^2 - \overline{\omega}_3^2 \right),$$
on the cone, has stabilizer $Spin(7)$, is closed and can be written as $\Theta=r^3 dr \wedge \varphi^{ts} + r^4 \psi^{ts}$.

\item The same proof as that of proposition \ref{prop:G2InstFromQK} also yields $G_2$-instantons with respect to $\varphi^{ts}$ by pulling back anti-self-dual connections on $(Z,g_Z)$.

\item One may also consider the $G_2$-structures obtained by scaling differently each of the $\eta_i$, while keeping them orthonormal, i.e.
$$\varphi_{a,b,c}= abc \eta_1 \wedge \eta_2 \wedge \eta_3 + a \eta_1 \wedge \omega_1 + b \eta_2 \wedge \omega_2 + c \eta_3 \wedge \omega_3.$$
It is easy to check that any such $G_2$-structure is coclosed if and only if $a=b=c$.
\end{enumerate}
\end{remark}

We now change the point of view on $(X^7, g_7)$ equipped with its tri-Sasakian structure, and regard it as a Sasakian manifold with respect to any of the Reeb vector fields $\xi_q = q_1 \xi_1 + q_2 \xi_2 + q_3 \xi_3$, for a unit quaternion $q=q_1i + q_2j + q_3k \in \Im (\mathbb{H})$. In fact, the resulting Sasakian manifold is always quasi-regular and does not depend on $q$. Take $\xi=\xi_1$ for example, i.e. $(X^7, \xi_1, g_7)$, then the leaf space $(Y^6, \omega_{KE}= \frac{d \eta_1}{2})$ is a K\"ahler-Einstein Fano orbifold. In fact $Y^6$ is the twistor space associated with the quaternionic K\"ahler structure on $Z$. Moreover, $Y$ is smooth if and only if $Z$ is. In fact, the twistor space also comes equipped with a nearly K\"ahler structure, see \cite{Nagy2002}. The next result relates this nearly K\"ahler structure with the $G_2$-structure $\varphi_{1/ \sqrt{2}}$ on $X$. We came across it after a conversation with Mark Haskins, so it may be known to experts. However, we were unable to locate a reference

\begin{proposition}
Let $(X^7,g^7)$ be a tri-Sasakian manifold, then $(\iota_{\xi_1/t} \varphi_t, -\iota_{\xi_1/t} \psi_t)$ are basic with respect to $\xi_1$ and and equip the twistor space with a nearly-K\"ahler structure if and only if $t=\pm 1/\sqrt{2}$.
\end{proposition}
\begin{proof}
The forms $\omega= \iota_{\xi_1/t} \varphi_t$, $\Omega_1 = -\iota_{\xi_1/t} \psi_t$ and $\Omega_2 = \varphi_t - t \eta_1 \wedge \iota_{\xi_1/t} \varphi_t$ are all basic with respect to $\xi_1$ and so they are the pullback of forms on the twistor $Y$. We denote these by $(\omega, \Omega_1, \Omega_2)$ respectively, and we must check these equip $Y^6$ with a nearly K\"ahler structure. Back in $X^7$ these can be written as
$$\omega= t^2 \eta_{23} - \frac{s}{48} \omega_1  , \ \Omega_1 =  \frac{st}{48} \left( \eta_2 \wedge \omega_3 - \eta_3 \wedge \omega_2 \right) , \ \Omega_2 = -  \frac{st}{48} \left( \eta_2 \wedge \omega_2 + \eta_3 \wedge \omega_3 \right).$$
Then, we compute that $d\omega=-3 \lambda \Omega_1$ and $d \Omega_2= 2 \lambda \omega_1^2$ for some $\lambda$ if and only if $t= \pm 1/ \sqrt{2}$, in which case $\lambda= \mp \sqrt{2}$ and so $(\omega, \Omega_1)$ does equip $Y^6$ with a nearly K\"ahler structure.
\end{proof}

\begin{remark}\label{rem:G2fromNK}
In particular, using the notation introduced in the proof of the previous proposition, we can recover the $G_2$-structure $\varphi_t$ by
$$\varphi_t = t\eta_1 \wedge \omega + \Omega_2 , \ \ \psi_t = - t\eta_1 \wedge \Omega_1 + \frac{\omega^2}{2}.$$

\end{remark}

As a consequence, we have

\begin{proposition}\label{prop:G2InstFromNK}
Let $A$ be a pseudo Hermitian-Yang-Mills (pHYM) connection for the nearly K\"ahler structure $(\omega, \Omega_1)$ on $Y^6$. Then, its pullback is a $G_2$-instanton with respect to $\varphi_{1/ \sqrt{2}}$.
\end{proposition}
\begin{proof}
If $A$ is pHYM, its curvature $F$ satisfies $F \wedge \omega^2 =0= F \wedge \Omega_1$. Then, writing $\varphi_1$ in terms of $(\omega, \Omega_1)$ as in remark \ref{rem:G2fromNK} we have $F \wedge \psi_{1/ \sqrt{2}}=0$ and so $A$ is a $G_2$-instanton with respect to $\varphi_{1/ \sqrt{2}}$.
\end{proof}

\begin{remark}
\begin{enumerate}
\item Every nearly parallel $G_2$-manifold carries a metric compatible connection $A$, in the tangent bundle whose holonomy is in $G_2$. Therefore, by the Ambrose-Singer theorem $F_A$ takes values in $\Lambda^2 \otimes \mathfrak{g}_2$. This connection is metric compatible and has anti-symmetric torsion, and then one can show that $F_A$ takes values in $S^2(\Lambda^2)$, see proposition 3.1 in \cite{Harland2011} for example. Putting all this together we see that actually $F_A$ takes values in $S^2(\Lambda^2_{14})$, as $\mathfrak{g}_2 \cong \Lambda^2_{14}$, and so is a $G_2$-instanton.

\item A similar statement to proposition \ref{prop:G2InstFromNK} holds for the pullback of a HYM on $Y^6$ with respect to its K\"ahler Einstein structure $\omega_{KE}=d\eta_1/2$. Namely, the pullback of such a HYM connection yields a $G_2$-instanton for $\varphi^{ts}$.
\end{enumerate}
\end{remark}

\subsection{Deformation theory revisited}

In this subsection we shall restrict to case where $(X^7, \varphi)$ is a strictly nearly parallel $G_2$-manifold and prove some rigidity results regarding $G_2$-instantons on them. Then, in \ref{sec:NPEnergy}, we prove that on nearly parallel manifolds there are $G_2$-instantons which are not locally energy minimizing. Recall, from formula \ref{eq:action}, that the analogous statement for torsion free $G_2$-structures is always false. 

\subsubsection{Rigidity}
The fact that strictly nearly parallel manifolds are Einstein with positive Einstein constant gives some hope of obtaining higher regularity for the moduli space of $G_2$-instantons than on torsion-free $G_2$-manifolds. In this direction we have

\begin{proposition}\label{prop:Rigid}
Let $(X^7, \varphi)$ be a strictly nearly parallel $G_2$-manifold, and $A$ be a $G_2$-instanton with the property that all the eigenvalues of $b \mapsto - 2 \ast   [ \ast F_A^{14}   \wedge b]$ are smaller than $6$. Then, $A$ is rigid as a $G_2$-instanton and $(A, 0)$ unobstructed as a monopole. Moreover, if $A$ is irreducible, then $(A, 0)$ is also rigid as a monopole.
\end{proposition}
\begin{proof}
Let $A$ be a connection as in the statement. Then we shall consider the operator $d_1$, $d_2$ from the complex \ref{eq:ComplexMonopole}, associated with $(A,0)$. As $\varphi$ is coclosed these can be written as
$$d_2(a, \phi)= \ast (d_{A} a \wedge \psi)-d_A \phi , \ \ d_2^* b = (\ast d_A b \wedge \psi , - d_A^* b),$$
while $d_1(\psi)= (-d_A \psi, 0)$, $d_1^*(a,\phi)=-d_A^* a$. Then, the operator $d_1^* \oplus d_2$ which controls the deformation theory of the $G_2$-instanton equation is 
$$(d_1^* \oplus d_2)(a, \phi) = (\ast (d_A a \wedge \psi)-d_A \phi , -d_A^* a),$$
which is self-adjoint. In order to study its infinitesimal deformations we must therefore study its kernel. So let $A$ be as in the statement and $(a, \phi) \in \ker (d_1^* \oplus d_2)$. Then, $\ast (d_A a \wedge \psi) = d_A \phi$ and $d_A^* a=0$, moreover as $\varphi$ is coclosed we have that
\begin{eqnarray}\nonumber
0 & = & (d_1^* \oplus d_2)^2(a, \phi) \\ \nonumber
& = & (\Delta_A \phi + \ast([F_A \wedge a] \wedge \psi) , \ast d_A( \ast (d_A a \wedge \psi) \wedge \psi) + d_A d_A^* a ).
\end{eqnarray}
Then, if $A$ is an irreducible $G_2$-instanton the first entry gives $\Delta_A \phi=0$. Hence, taking the inner product with $\phi$ and integrating by parts we get $d_A \phi=0$. From the second entry above and using that $d \varphi = \lambda \psi$ we compute 
\begin{eqnarray}\nonumber
0 & = & 3\ast d_A \ast d_A^7 a + d_A d_A^* a \\ \nonumber
& = & \ast d_A \ast (d_A a + \ast(d_A a \wedge \varphi)) + d_A d_A^* a \\ \nonumber
& = & \Delta_A a + \ast ([F_A \wedge a] \wedge \varphi) + \lambda \ast (d_A a \wedge \psi) \\
& = & \Delta_A a + \ast ([F_A \wedge a] \wedge \varphi),
\end{eqnarray}
where in the last equality we used that $\ast (d_A a \wedge \psi) = d_A \phi=0$. Putting this together with the Weitzenb\"ock formula $\Delta_A a = \nabla_A^* \nabla_A a + \ast  [\ast F_A \wedge a]+ \Ric (a)=0$, we obtain
$$\nabla_A^* \nabla_A a + \ast  ( [ (\ast F_A  - F_A \wedge \varphi )\wedge a]+ \Ric (a).$$
As $F_A \wedge \varphi = 2 \ast F_A^7 - \ast F_A^{14}$, and $g_{\varphi}$ is Einstein with positive Einstein constant $6$, i.e. $\Ric= 6 \id$, we have
$$\nabla_A^* \nabla_A a + \ast  ( [ \ast ( 2F_A^{14}  - F_A^7 ) \wedge a]+ 6 a =0.$$
If $A$ is as in the hypothesis of the statement, then taking the inner product with $b$, the sum of the last two terms is positive and so we have
$$\Vert \nabla_A a \Vert^2_{L^2}+ \mu \Vert a \Vert^2_{L^2} \leq 0,$$
for some $\mu>0$. We conclude that $a$ must vanish identically and as we have already seen $d_A \phi=0$. Hence, any infinitesimal monopole deformation of $(A, 0)$ is of the form $(0, \phi)$ for some $\phi$ satisfying $d_A \phi=0$. These can obviously be integrated as the path $\lbrace (A, t \phi) \rbrace_{t \in \mathbb{R}}$ and so is a purely monopole deformation which keeps the connection $A$ the same $G_2$-instanton.\\
Exactly the same proof shows that $d_2$ is surjective (by showing that $\ker(d_2^*)=0$) proving that $(A,0)$ is unobstructed as a monopole. Moreover, if $A$ is irreducible, then $d_A \phi=0$ implies that $\phi$ must vanish and so $(A, 0)$ is also rigid as a monopole.
\end{proof}

\begin{corollary}\label{cor:Rigid}
Let $(X^7, \varphi)$ be a strictly nearly parallel $G_2$-manifold. Then, 
\begin{enumerate}
\item Abelian $G_2$-instantons are rigid.
\item Flat connections are rigid as $G_2$-instantons.
\end{enumerate}
\end{corollary}

One may wonder if the rigidity of Abelian $G_2$-instantons extends from strictly nearly parallel $G_2$-structures to a more general class, say coclosed ones. We will see a counterexample to this in the second bullet of the first item in theorem \ref{thm:X11}, see also the second item in remark \ref{rem:X11}.\\
We shall now comment on the relation of proposition \ref{prop:Rigid} to the $G_2$-instantons we constructed earlier in this section. 

\begin{remark}
\begin{enumerate}
\item Through corollary \ref{cor:G2InstLineBundle} we know that there is a unique $G_2$-instanton on every complex line bundle $L$ over a strictly nearly parallel $G_2$-manifold. This actually supersedes corollary \ref{cor:Rigid}.
\item A similar result to corollary \ref{cor:Rigid} holds for nearly K\"ahler manifolds, see theorem 1 in \cite{Charbonneau2016}. In fact, also in that case any complex line bundle admits a unique pseudo-Hermitian-Yang-Mills connection. See theorem 3.23 and remark 3.25 in \cite{Foscolo2016}.
\end{enumerate}
\end{remark}

It is also possible to find examples of $G_2$-instantons on strictly nearly parallel $G_2$-manifolds for which proposition \ref{prop:Rigid} does not apply

\begin{example}
Consider an anti-self-dual, Einstein $4$-orbifold $(Z, g_Z)$, with positive Einstein constant admitting a family of anti-self-dual connections (e. g. $\mathbb{S}^4$). Then, by proposition \ref{prop:G2InstFromQK}, these connections lift to a family of $G_2$-instantons for a strictly nearly parallel $G_2$-structure constructed on the principal $SO(3)$-bundle associated with $\Lambda^2_+ Z$. Therefore, in this case $G_2$-instantons have nontrivial moduli and so the hypothesis in proposition \ref{prop:Rigid} must fail.
\end{example}

\subsubsection{Yang-Mills unstable $G_2$-instantons}\label{sec:NPEnergy}

Let $A$ be a $G_2$-instanton for a nearly parallel $G_2$-structure $\varphi$ such that $d \varphi = \lambda \psi$. We have seen, in proposition \ref{prop:NearlyG2YM}, that such $G_2$-instantons are actually Yang-Mills connections. Moreover, equation \ref{eq:action} and the discussion below it show that in the torsion free case a $G_2$-instanton minimizes the Yang-Mills energy. That need not be the case for strictly nearly parallel $G_2$-structures as we now show with a counterexample.

\begin{example}\label{ex:UnstableS7}
Equip the $7$-dmensional sphere, $\mathbb{S}^7$ with the nearly parallel $G_2$-structure $\varphi^{ts}$ induced from the tri-Sasakian one, as in remark \ref{rem:TriSasakian_G2}. Then $g_{\varphi^{ts}}$ is the round metric. Now consider the Hopf bundle $\pi_H : \mathbb{S}^7 \rightarrow \mathbb{S}^4$. A verbatim of the proof of proposition \ref{prop:G2InstFromQK} shows that the pullback, via $\pi_H$, of a self-dual connection on $\mathbb{S}^4$ is also a $G_2$-instanton with respect to $\varphi^{ts}$. Hence, if $A$ is the pullback of a charge $1$ self dual connection on $\mathbb{S}^4$, it is a $G_2$-instanton for $\varphi^{ts}$. As $d \varphi^{ts}= 4 \psi^{ts}$, we have that $A$ is also a (nonflat) Yang-Mills connection. However, it is shown in \cite{Simons1979} that any nonflat Yang-Mills connection on $S^n$, $n>4$, is Yang-Mills unstable. 
\end{example}

\begin{remark}
We have also proved in proposition \ref{prop:G2InstFromQK} that the pull-back of a Yang-Mills connection on a quaternion-K\"ahler manifold is both a $G_2$-instanton and a Yang-Mills connection, with respect to any of the $G_2$-structures $\varphi_t$, for $t>0$. Hence, the example above also works also for any $\varphi_t$ with $t$ in a neighborhood of $1$.
\end{remark}

\section{Aloff-Wallach spaces}\label{sec:Aloff-Wallach}

We start this section in \ref{ss:AWgeometry} by summarizing some facts about the geometry of homogeneous, coclosed $G_2$-structures on  Aloff-Wallach spaces. Then in subsection \ref{ss:InvariantConnections} we determine all the invariant connections on homogeneous $SO(3)$-bundles over the Aloff-Wallach spaces and use them in \ref{sec:AWG2Instantons} and \ref{sec:X11} to classify invariant $G_2$-instantons on the Aloff-Wallach spaces. As a consequence, we discover that $G_2$-instantons can distinguish between different strictly nearly parallel $G_2$-structures on the same Aloff-Walach space. We also produce examples of some interesting phenomena, for instance: irreducible $G_2$-instantons that merge into the same reducible $G_2$-instanton as the $G_2$-structure varies. This particular phenomenon was expected to occur, but these are the first examples. In \ref{sss:X1-1} we shall also give examples of $G_2$-instantons for a nearly parallel $G_2$-structure in $X_{1,-1}$. Some of these are then shown to not be locally energy minimizing. In fact, they are saddles of the invariant Yang-Mills functional. Further, in \ref{sss:X11} we show that the existence of $G_2$-instantons distinguishes between a tri-Sasakian and a strictly nearly parallel $G_2$-structure on $X_{1,1}$.

\subsection{Geometry of coclosed $G_2$-structures}\label{ss:AWgeometry}

Let $k, l \in \mathbb{Z}$ and $U(1)_{k,l}$ be a circle subgroup of $SU(3)$ consisting of elements of the form
$$\left( \begin{array}{ccc}
e^{ik\theta} & 0 & 0 \\
0 & e^{il\theta} & 0 \\
0 & 0 & e^{im\theta} \end{array} \right),$$
where $k+l+m=0$. The Aloff-Wallach space $X_{k,l}=SU(3) / U(1)_{k,l}$ is the quotient of $SU(3)$ by this circle subgroup. We shall now recall some basic facts about the geometry and topology of the Aloff-Wallach spaces. Aloff-Wallach spaces inherited their name from \cite{Aloff1975}, where they were shown to admit homogeneous metrics with positive curvature, for $klm \neq 0$ (see also the survey paper \cite{Ziller2004} page 18). Later, Wang showed in \cite{Wang1982} that Aloff-Wallach spaces admit homogeneous Einstein metrics with positive scalar curvature, not all of which are the ones considered by Aloff and Wallach. In \cite{Baum1990}, page 116, the authors show that each $X_{k,l}$ admits at least two homogeneous Einstein metrics. The authors further show, that for $X_{k,k}$ (and those related to it through the action of the Weyl group of $SU(3)$, see remark \ref{rem:Weyl}) one of these is tri-Sasakian and the other strictly nearly parallel, while on the other $X_{k,l}$ they are both strictly nearly parallel. As a side remark, we mention that there are examples of different pairs $(k,l)$ such that the corresponding Aloff-Wallach spaces are homeomorphic, but not diffeomorphic, \cite{Kreck1998}.\\
Regarding coclosed $G_2$-structures, Aloff-Wallach spaces were shown to admit a real 4-dimensional family of homogeneous, coclosed $G_2$-structures as described in \cite{Cabrera1996}.\\
To describe homogeneous coclosed $G_2$-structures on $X_{k,l}$ we follow \cite{Cabrera1996} and identify $TX_{k,l}$ with the imaginary octonions $\text{Im} \, \mathbb{O}$ as follows. Let $$s=\frac{\sqrt{k^2+l^2+m^2}}{\sqrt{6}},$$ and fix the following basis for $\mathfrak{su}(3):$

\begin{align*}
&\ e_1= \frac{1}{\sqrt{2}} \left( \begin{array}{ccc}
0 & 1 & 0 \\
-1 & 0 & 0 \\
0 & 0 & 0 \end{array}\right), \: e_5=\frac{i}{\sqrt{2}} \left( \begin{array}{ccc}
0 & 1 & 0 \\
1 & 0 & 0 \\
0 & 0 & 0 \end{array} \right), \\
&\ e_2= \frac{1}{\sqrt{2}} \left( \begin{array}{ccc}
0 & 0 & 0 \\
0 & 0 & 1 \\
0 & -1 & 0 \end{array}\right), e_6= \frac{i}{\sqrt{2}} \left( \begin{array}{ccc}
0 & 0 & 0 \\
0 & 0 & 1 \\
0 & 1 & 0 \end{array}\right), \\
&\ e_3= \frac{1}{\sqrt{2}} \left( \begin{array}{ccc}
0 & 0 & -1 \\
0 & 0 & 0 \\
1 & 0 & 0 \end{array}\right), e_7= \frac{i}{\sqrt{2}} \left( \begin{array}{ccc}
0 & 0 & 1 \\
0 & 0 & 0 \\
1 & 0 & 0 \end{array}\right), \\
&\ e_4= \frac{i}{3\sqrt{2}s} \left( \begin{array}{ccc}
l-m & 0 & 0 \\
0 & m-k & 0 \\
0 & 0 & k-l \end{array}\right), \\
&\  H= \frac{i}{\sqrt{6} s} \left( \begin{array}{ccc}
k & 0 & 0 \\
0 & l & 0 \\
0 & 0 & m \end{array}\right),
\end{align*}

where $\sqrt{6}sH$ is the infinitesimal generator of the $\mathfrak{u}(1)_{k,l}$ action, while $\lbrace e_i \rbrace_{i=1}^7$ is a basis of $\mathfrak{u}(1)_{k,l}^{\perp}$.

A basis for $\text{Im} \, \mathbb{O}$ is given by $\{ i,j,k,e,ie,je,ke\}$. Let $A, B, C,$, $D$ be nonzero constants. We identify $TX_{k,l}$ with $\text{Im} \, \mathbb{O}$ as follows: identify $Ae_1$ with $i$ and $Ae_5$ with $ie$; $Be_2$ with $j$ and $Be_6$ with $je$; $Ce_3$ with $k$ and $Ce_7$ with $ke$; and finally $De_4$ with $e$. Let $\omega_1,...,\omega_7$ be the left-invariant coframe of $X_{k,l}$ obtained from the $e_i$ by dualising using the $SU(3)$-invariant metric $\langle A, B \rangle = -\tr(AB)$ on $\mathfrak{su}(3)$. With our identifications, the $G_2$ structure is
\begin{equation}\label{eq:G2str}
\varphi=ABC(\omega_{123}-\omega_{167}+\omega_{257} -\omega_{356})-D\omega_4\wedge(A^2\omega_{15}+B^2\omega_{26}+C^2\omega_{37}),
\end{equation}
The metric $g_\varphi$ and the $4$-form $\psi= \ast_{\varphi} \varphi$ associated to the $G_2$-structure are:
\begin{eqnarray}\nonumber
\psi & = & ABCD(\omega_{4567}-\omega_{2345}+\omega_{1346}-\omega_{1247})+B^2C^2\omega_{2367} + A^2C^2\omega_{1357} \\ \nonumber & & +A^2B^2\omega_{1256}. \\ \nonumber
g_{\varphi} & = & A^2\left(\omega_1^2 +\omega_5^2\right)+B^2\left(\omega_2^2 +\omega_6^2\right)+C^2\left(\omega_3^2 +\omega_7^2\right)+D^2\omega_4^2.
\end{eqnarray}

where we have fixed the orientation induced by the volume form $\text{vol}_\varphi=7A^2B^2C^2D \omega_{1234567}$. Also, notice that this family of $G_2$-structures is up to scaling only $3$-dimensional. We now calculate the exterior derivatives of $\varphi$ and $\psi$, to get information about the torsion of these $G_2$-structures. We find

\begin{align*}
\sqrt{2} \: d\varphi=& D(A^2+B^2+C^2)(\omega_{4567}-\omega_{2345}+\omega_{1346}-\omega_{1247}) \\& + (4ABCs-B^2Dl-C^2Dk)\omega_{2367}+(4ABCs-C^2Dm-A^2Dl)\omega_{1357}\\&+(4ABCs-A^2Dk-B^2Dm)\omega_{1256},
\end{align*}
$$d\psi=0,$$

where $s=\frac{\sqrt{k^2+l^2+m^2}}{\sqrt{6}}$. From these we can extract the torsion component $\tau_0$:

\begin{align*}
&\frac{7}{\sqrt{2}} \tau_0=4\left( \frac{A}{BC}+\frac{B}{AC}+\frac{C}{AB}\right)-\frac{D}{s}\left(\frac{l}{C^2}+\frac{k}{B^2}+\frac{m}{A^2}\right)
\end{align*}

\begin{definition}\label{def:C}
Let $\mathcal{C}$ denote the spaces of $G_2$-structures of the form \ref{eq:G2str}.
\end{definition}

\begin{lemma}\label{lem:Space_Of_G2_Structures}
Let $k \neq \pm l$, $l \neq \pm m$, $m \neq k$, then the space of homogeneous coclosed $G_2$-structures $\mathcal{C}$ may be identified with $(\mathbb{R}^+)^2 \times (\mathbb{R} \backslash \lbrace 0 \rbrace )^2$. Moreover, given $(A,B,C,D) \in \mathcal{C}$ the corresponding $G_2$-structure can be written as in equation \ref{eq:G2str}.
\end{lemma}
\begin{proof}
It follows from the analysis in \cite{Cabrera1996} that for $k \neq \pm l$, $l \neq \pm m$, $m \neq k$ any homogeneous, coclosed $G_2$-structure is one of those considered above. These are precisely those with $s'=0$, in that reference. Now notice that the $G_2$-structures \ref{eq:G2str} are parametrized by $(A,B,C,D) \in (\mathbb{R} \backslash \lbrace 0 \rbrace )^4 $ minus the coordinates hyperplanes. Moreover, \ref{eq:G2str} stays invariant by any of the following maps $(A,B) \mapsto (-A,-B)$, $(B,C) \mapsto (-B, -C)$ and $(A,C) \mapsto (-A,-C)$. These discrete symmetries give rise to a $\mathbb{Z}_2 \times \mathbb{Z}_2$-action on $(\mathbb{R} \backslash \lbrace 0 \rbrace )^4$, generated by the first two symmetries. Hence, the $G_2$-structures in equation \ref{eq:G2str} are parametrized by $(\mathbb{R} \backslash \lbrace 0 \rbrace )^4 / \mathbb{Z}_2 \times \mathbb{Z}_2$. Taking a fundamental domain for the $\mathbb{Z}_2 \times \mathbb{Z}_2$ action we may equally well regard the space of $G_2$-structures as in \ref{eq:G2str} as $\mathbb{R}^+_A \times \mathbb{R}^+_B \times (\mathbb{R}_C \backslash \lbrace 0 \rbrace ) \times (\mathbb{R}_D \backslash \lbrace 0 \rbrace )$.
\end{proof}

\begin{remark}
\begin{enumerate}
\item Up to a cover, and the action of the Weyl group (see remark \ref{rem:Weyl} below), the restrictions in the lemma above can be simply written as $(k,l) \notin \lbrace (1,1), (1,-1) \rbrace$.

\item In the case when $(k,l) \in \lbrace (1,1), (1,-1) \rbrace$ we will continue to use $\mathcal{C}$ to denote the $G_2$-structures as in \ref{eq:G2str}. However, in that case there are homogeneous coclosed $G_2$-structures that can not be written as in \ref{eq:G2str} and so are not in $\mathcal{C}$.

\item We know that $\tau_1=\tau_2=0$ because the $G_2$-structure is coclosed, and we can compute $\tau_3$ by $\tau_3= \ast \left(d\varphi-\tau_0\psi\right)$.
\end{enumerate}
\end{remark} 

A $G_2$-structure of the form \ref{eq:G2str} is nearly parallel, i.e. $d\varphi = \lambda \psi,$ when $(A,B,C,D)$ satisfy
\begin{eqnarray}\label{eq:nearlyparallelG2eqns}
  A^2 + B^2 + C^2 -\sqrt{2} \lambda ABC & = &  0  \\ 
D \left(kA^2+mB^2 \right) -4sABC + \sqrt{2}\lambda s A^2B^2 & = & 0 \\ 
D \left(lB^2+kC^2 \right) -4sABC + \sqrt{2}\lambda s B^2C^2 & = & 0 \\
D \left(lA^2+mC^2 \right) -4sABC + \sqrt{2}\lambda s A^2C^2 & = & 0.
\end{eqnarray}

By fixing an orientation we can suppose that $\lambda >0$. Then, in \cite{Cabrera1996} it is shown that for $k \neq \pm l$, $l \neq \pm m$, $m \neq \pm l$, the system \ref{eq:nearlyparallelG2eqns} admits precisely eight solutions. Moreover, up to the action of $\mathbb{Z}_2 \times \mathbb{Z}_2$ alluded to in the proof of lemma \ref{lem:Space_Of_G2_Structures}, these eight solutions give only two nonequivalent solutions $\varphi \in \mathcal{C}$, which are in fact strictly nearly parallel. The following result completely determines the connected component in $\mathcal{C}$ in which each of these structure lives.

\begin{lemma}
Let $k \neq \pm l$, $l \neq \pm m$, $m \neq \pm l$ and $\varphi^{np_1}, \varphi^{np_2} \in \mathcal{C}$ denote the two strictly nearly parallel $G_2$ structures. Then, $C(\varphi^{np_1})$, $ C(\varphi^{np_2})$ have the same sign, while that of $D(\varphi^{np_1})$ and $D(\varphi^{np_2})$ is opposite. Moreover, $\sign(C)$ is constrained by $\lambda C>0$ and determines the orientation.
\end{lemma}
\begin{proof}
Fix an orientation and suppose that $\lambda >0$. Then, the first equation in \ref{eq:nearlyparallelG2eqns} above implies that $ABC$ must be positive for any such $\varphi$. On the other hand, it follows from the analysis in the bottom of page 413 in \cite{Cabrera1996} that the two solutions have different signs of $ABCD$ and so they must in fact have different signs of $D$. Choosing $\varphi \in \mathcal{C}$, we have and as $A>0$ and $B>0$, so we must also have $C>0$ (as $ABC>0$), which then implies each of the solutions has a different sign of $D$.
\end{proof}

\begin{remark}\label{rem:Weyl}
\begin{enumerate}
\item Notice that the Weyl group of $SU(3)$ moves the $U(1)_{k,l}$ subgroup inducing an action in the set of Aloff-Wallach spaces. In fact, this action is generated by $X_{k,l} \mapsto X_{l,k}$ and $X_{k,l} \mapsto X_{k,m}$, which can be combined to generate the order $3$-element $\sigma: X_{k,l} \rightarrow X_{l,m}$, i.e. cyclic permutations of $(k,l,m)$. Hence, up to coverings and this action, there is no loss in supposing that $k$ and $l$ are coprime and that $k \geq 0$, $-l \leq k \leq 2 l$.

\item Consider the $U(2)$-subgroup of $SU(3)$ generated by the image of the homomorphism $SU(2) \times U(1) \rightarrow SU(3)$ given by
$$(A, e^{i \theta}) \mapsto \diag(A e^{i \theta}, \det( Ae^{i\theta})^{-1} ).$$
As $\mathbb{C}^2 \cong SU(3)/U(2)$, we obtain a canonical fibration
$$ \pi_1: X_{k,l} \rightarrow \mathbb{CP}^2,$$
whose fibers one can check to be the lens spaces $U(2)/U(1)_{k,l} \cong S^3/\mathbb{Z}_{\vert k+l \vert}$, if $k+l \neq 0$, or $S^1 \times S^2$, if $k+l=0$. In fact, using the order $3$ element $\sigma$, we may obtain two more fibrations $\pi_2 = \pi_1 \circ \sigma$ and $\pi_3= \pi_1 \circ \sigma^2$, of $X_{k,l}$ over $\mathbb{CP}^2$. At least two of which have fibres $S^3/ \mathbb{Z}_p$ for a nonzero $p \in \lbrace \vert k \vert , \vert l \vert , \vert m \vert \rbrace$.
\end{enumerate}
\end{remark}

\subsection{Invariant Connections}\label{ss:InvariantConnections}

Given a Lie group $G$, a principal $G$-bundle $P$ over $X_{k,l}=SU(3) / U(1)_{k,l}$ is said to be $SU(3)$-homogeneous (or just homogeneous) if there is a lift of the $SU(3)$ action on $X_{k,l}$ to the total space, which commutes with the right $G$-action on $P$. In general, homogeneous $SO(3)$-principal bundles over $X_{k,l}$ are determined by their isotropy homomorphisms $\lambda_n : U(1) \to SO(3)$, and are constructed via 
$$P_{n} = SU(3) \times_{(U(1)_{k,l},\lambda_n)} SO(3),$$
where the possible group homomorphisms $\lambda_n$ are parametrised by $n \in \mathbb{Z}$. Explicitly we can think of $SO(3)$ as $SU(2)/ \mathbb{Z}_2$, where $\mathbb{Z}_2$ acts via multiplication by minus the identity matrix $-\textbf{1}$, then $\lambda_n$ is given by $$\lambda_n(\theta) =  \begin{pmatrix}
  e^{i\frac{n}{2}\theta} & 0  \\
  0 & e^{-i\frac{n}{2}\theta}
 \end{pmatrix} \mod - \textbf{1}.  $$

\begin{definition}
Let $T_1, T_2, T_3$ be a basis for $\mathfrak{su}(2)$ such that $[T_i,T_j]=2\epsilon_{ijk}T_k$. Then the canonical invariant connection on $P_{n}$ is $$A^n_c = - \frac{n}{2} \frac{h}{\sqrt{6} s} \otimes T_1.$$ 
\end{definition}

Using the Maurer-Cartan equations, the curvature of the canonical invariant connection $A^n_c$ is found to be $$F^c_c= - \frac{n}{s\sqrt{6}} \left( \left(k-l\right)\omega_{15} +(l-m)\omega_{26} +(m-l)\omega_{37} \right).$$

Wang's theorem \cite{Wang1958} classifies invariant connections on homogeneous bundles. In our situation, Wang's theorem says that $SU(3)$-invariant connections on $P_{n}$ are in bijection with morphisms of $U(1)$-representations $$\Lambda : (\mathfrak{m},\text{Ad}) \to (\mathfrak{so}(3),\text{Ad}\circ \lambda_n),$$ where $\mathfrak{m}$ is the $U(1)_{k,l}$-Ad complement to $\langle H \rangle$ in $\mathfrak{su(3)}$. If $(1,1)$ is not in the Weyl orbit of $(1,1)$ and $n \neq 0$, these split into the irreducible real representations as
\begin{align*}
&\ \mathfrak{m}= \langle X_1, X_5 \rangle_{k-l} \oplus \langle X_2, X_6 \rangle_{l-m} \oplus\langle X_3, X_7 \rangle_{m-k} \oplus \langle X_4 \rangle, \\
&\ \mathfrak{so}(3) =\langle T_1 \rangle \oplus \langle T_2, T_3 \rangle_n,
\end{align*}
where the weight of each 2-dimensional irreducible representation is indicated by a subscript. It will be useful to use the notation $V_1 = \langle X_1, X_5 \rangle, V_2 = \langle X_2, X_6 \rangle, V_3 = \langle X_3, X_7 \rangle$ (these are simply the real root spaces of $\mathfrak{su}(3)$). Applying Schur's lemma and Wang's theorem \cite{Wang1958} we have

\begin{lemma}\label{lem:ConnectionForms}($(k,l) \neq (1,1)$)
Let $A^n \in \Omega^1(SU(3), \mathfrak{so}(3))$ be the connection 1-form of an invariant connection on $P_{n}$ over $X_{k,l}$, for $(k,l)$ not in the Weyl orbit of $(1,1)$. Then it is left-invariant and can be written as $A^n=A^n_c + (A^n-A^n_c),$ where $(A-A^n_c) \in \mathfrak{m}^* \otimes \mathfrak{so}(3)$, extended as a left-invariant 1-form with values in $\mathfrak{so}(3),$ is given by
$$A-A^n_c = a_1 \psi_1 + a_2 \psi_2 + a_3 \psi_3 + b \omega_4 \otimes T_1.$$ Here the $\psi_i$ denote isomorphisms  $\psi_i : V_i \xrightarrow{\sim} \langle T_2, T_3 \rangle$ with $\vert \psi \vert \in \lbrace 0, 1 \rbrace$ with respect to the fixed basis, and the $a_i, b \in \mathbb{R}$ are constants. Moreover, each $a_i$ must vanish if the weight of $V_i$ is not equal to $n$, i.e.
\begin{eqnarray}\nonumber
a_1 & = & 0 , \ \text{if $n \neq k-l$} \\ \nonumber
a_2 & = & 0, \ \text{if $n \neq l-m$} \\ \nonumber
a_3 & = & 0 , \ \text{if $n \neq m-k$.}
\end{eqnarray}
\end{lemma}

\begin{remark}
\begin{enumerate}
\item Notice that the order $3$ element of the Weyl group $W$ permuts the different roots and so the different root spaces. In particular, there is no loss in considering the Aloff-Wallach spaces up to the action of $W$. Hence, in the previous lemma when we consider the case $k \neq l$, it is implicit that also $l \neq m$ or $m \neq k$.

\item Since it is not possible to have $k-l=l-m=m-k=n$ without forcing $k=l=m=n=0$, we must have $a_1a_2a_3=0$. This splits us into seven cases to be analyzed below.
\end{enumerate}
\end{remark}

\begin{lemma}\label{lem:ConnectionForms_Case_(1,1)}($(k,l)=(1,1)$)
Let $A^n \in \Omega^1(SU(3), \mathfrak{so}(3))$ be the connection 1-form of an invariant connection on $P_{n}$ over $X_{1,1}$. Then it is left-invariant and can be written as $A^n=A^n_c + (A^n-A^n_c),$ where $(A-A^n_c) \in \mathfrak{m}^* \otimes \mathfrak{so}(3)$, extended as a left-invariant 1-form with values in $\mathfrak{so}(3),$ is given by
$$A-A^n_c = a_1 \chi + a_2 \psi_2 + a_3 \psi_3 .$$ Here the $\psi_i$ denote isomorphisms $\psi_i : V_i \xrightarrow{\sim} \langle T_2, T_3 \rangle$ with $\vert \psi \vert=1$ with respect to the fixed basis, and $\chi: \langle X_1,X_5,X_4 \rangle \rightarrow \mathfrak{so}(3)$ denotes a linear map, which in the case $n \neq 0$ must take values in $\langle T_1 \rangle \subset \mathfrak{so}(3)$. 
\end{lemma}
\begin{proof}
The proof in this case is similar and we simply give the main steps. As before the proof amounts to using Wang's theorem, \cite{Wang1958}, to find the invariant connections. One must split the corresponding representations into irreducibles as
\begin{align*}
&\ \mathfrak{m}= \langle X_1 \rangle \oplus \langle X_5 \rangle  \oplus \langle X_4 \rangle \oplus \langle X_2, X_6 \rangle_{3} \oplus\langle X_3, X_7 \rangle_{-3}, \\
&\ \mathfrak{so}(3)=\langle T_1 \rangle \oplus \langle T_2, T_3 \rangle_n, \  \text{if $n \neq 0$}
 \\
&\ \mathfrak{so}(3)=\langle T_1 \rangle \oplus \langle T_2 \rangle \oplus \langle T_3 \rangle, \  \text{if $n = 0$}.
\end{align*}
Then the conclusion follows from a similar application of Schur's lemma.
\end{proof}

\subsubsection{Case splitting, for $k \neq l$}\label{sss:Case_Splitting_kl}

We shall now consider the case when $X_{k,l}$ is such that $(k,l)$ is not in the Weyl orbit of $(1,1)$, which will be investigated separately. Here we use lemma \ref{lem:ConnectionForms} in order to write down all the possible connection $1$-forms, up to invariant gauge transformations. We shall analyse the different cases corresponding to the different values of $n$.\\

\noindent \textbf{Case 0:} $n \neq k-l, l-m, m-k$\\

In this case $a_1=a_2=a_3=0$ and so every connection is reducible, with $$A^n=\left( - \frac{n}{2} \frac{h}{\sqrt{6} s} + b\omega_4 \right) \otimes T_1$$

\noindent \textbf{Case 1:} $n=k-l$\\

In this case $a_2 = a_3 =0$ and we may use our gauge freedom to write the isomorphism $\psi_1 : V_1 \xrightarrow{\sim} \langle T_2, T_3 \rangle$ as $\psi_1 = \omega_1 \otimes T_2 + \omega_5 \otimes T_3$. Then we have $$A^{k-l}=\left( - \frac{k-l}{2} \frac{h}{\sqrt{6} s} + b\omega_4 \right) \otimes T_1 + a_1 \left(\omega_1 \otimes T_2 + \omega_5 \otimes T_3 \right).$$

\noindent \textbf{Case 2:} $n=l-m$\\

Now we must have $a_1=a_3=0$ and as in case 1 we may use our gauge freedom to fix the form of $\psi_2$. We can write the connection form as $$A^{l-m}=\left( - \frac{l-m}{2} \frac{h}{\sqrt{6} s} + b\omega_4 \right) \otimes T_1 + a_2 \left(\omega_2 \otimes T_2 + \omega_6 \otimes T_3 \right).$$

\noindent \textbf{Case 3:} $n=m-k$\\

Similarly, in this case $a_1=a_2=0$ and we can write the connection form as $$A^{m-k}=\left( - \frac{m-k}{2} \frac{h}{\sqrt{6} s} + b\omega_4 \right) \otimes T_1 + a_3 \left(\omega_3 \otimes T_2 + \omega_7 \otimes T_3 \right).$$

\noindent \textbf{Case 4:} $n=m-k=l-m,$ i.e. $n=l=-k$\\

In this case $a_1=0$ and we exhaust our gauge freedom in fixing $\psi_2=\omega_2 \otimes T_2 + \omega_6 \otimes T_3,$ so that $$\psi_3=\omega_3 \otimes (\cos(\beta) T_2 + \sin(\beta)T_3) + \omega_7 \otimes (-\sin(\beta)T_2+\cos(\beta)T_3)$$ is dependent on an angle parameter $\beta.$ The connection form is
\begin{align*}
 A^{l}  =  &\ \left( - \frac{l}{2} \frac{h}{\sqrt{6} s} + b\omega_4 \right) \otimes T_1 + a_2 \left(\omega_2 \otimes T_2 + \omega_6 \otimes T_3 \right) + \\
&\ a_3 \left(\omega_3 \otimes (\cos(\beta) T_2 + \sin(\beta)T_3) + \omega_7 \otimes (-\sin(\beta)T_2+\cos(\beta)T_3) \right).
\end{align*}

\noindent \textbf{Case 5:} $n=l-m=k-l,$ i.e. $n=k=-m$\\

This is similar to case 4, but with $a_2=0$. The connection form is
\begin{align*}
 A^{k}  =  &\ \left( - \frac{k}{2} \frac{h}{\sqrt{6} s} + b\omega_4 \right) \otimes T_1 + a_1 \left(\omega_1 \otimes T_2 + \omega_5 \otimes T_3 \right) + \\
&\ a_3 \left(\omega_3 \otimes (\cos(\beta) T_2 + \sin(\beta)T_3) + \omega_7 \otimes (-\sin(\beta)T_2+\cos(\beta)T_3) \right).
\end{align*}

\noindent \textbf{Case 6:} $a_3=0, n=k-l=m-k,$ so that $n=m=-l$\\

This is similar to cases 4 and 5, except that we use $\alpha$ for the angle parameter. The connection form is
\begin{align*}
 A^{m}  =  &\ \left( - \frac{m}{2} \frac{h}{\sqrt{6} s} + b\omega_4 \right) \otimes T_1 + a_1 \left(\omega_1 \otimes T_2 + \omega_5 \otimes T_3 \right) + \\
&\ a_2 \left(\omega_2 \otimes (\cos(\alpha) T_2 + \sin(\alpha)T_3) + \omega_6 \otimes (-\sin(\alpha)T_2+\cos(\alpha)T_3) \right).
\end{align*}

\subsubsection{Case splitting, for $k=l=1$}\label{sss:Case_Splitting_11}

Now we use lemma \ref{lem:ConnectionForms_Case_(1,1)} to write down the possible connection 1-forms for an invariant connection on $P_{n}$ over  $X_{1,1}$, splitting into cases depending on the value of $n$.\\

\noindent \textbf{Case 0: }$n \neq 3, -3, 0.$ \\
$$A^n=\left( -\frac{n}{2}\frac{h}{\sqrt{6}}+b\omega_4+a_1\omega_1+a_5\omega_5 \right) \otimes T_1,$$
where $a_1, a_5, b \in \mathbb{R}.$\\

\noindent \textbf{Case 1:} $n=0.$ \\
$$A^0 = \omega_1 \otimes c_1 + \omega_4 \otimes c_4 + \omega_5 \otimes c_5,$$
where $c_1, c_4, c_5 \in \mathfrak{so}(3)$.\\

\noindent \textbf{Case 2:} $n=3.$ \\
$$A^3=\left( -\frac{3}{2}\frac{h}{\sqrt{6}}+b\omega_4+a_1\omega_1+a_5\omega_5 \right) \otimes T_1 + a_2 \left( \omega_2 \otimes T_2 + \omega_6 \otimes T_3 \right),$$
where $a_1, a_2, a_5, b \in \mathbb{R}.$\\

\noindent \textbf{Case 3:} $n=-3.$ \\
$$A^{-3}=\left( \frac{3}{2}\frac{h}{\sqrt{6}}+b\omega_4+a_1\omega_1+a_5\omega_5 \right) \otimes T_1 + a_3 \left( \omega_3 \otimes T_2 + \omega_7 \otimes T_3 \right),$$
where $a_1, a_3, a_5, b \in \mathbb{R}.$

\subsubsection{Topology of the homogenous bundles $P_{n}$}

Recall from the beginning of this subsection \ref{ss:InvariantConnections}, that given a group homomorphism $\lambda_n : U(1) \rightarrow SO(3)$ we may construct the homogeneous bundle
$$P_{n}= SU(3) \times_{(U(1)_{k,l}, \lambda_n)} SO(3)$$
over $X_{k,l}$. In this section we compute the first Pontryagin and second Stiefel-Whitney classes of the associated vector bundle $E_n$ with respect to standard action of $SO(3)$ on $\mathbb{R}^3$. To compute its characteristic classes it will be convenient to use a lift of $E_n$ to a $Spin^c(2)=U(2)$ bundle $W_n$. Then, the adjoint bundle $\mathfrak{g}_{W_n}$ of $W_n$ splits as $\mathfrak{g}_{W_n} \cong \underline{\mathbb{R}} \oplus E_n$, where $ \underline{\mathbb{R}}$ denotes the trivial bundle. We can then compute the characteristics of $E_n$ via the Chern classes of $W_n$ as
$$w_2(E_n) = c_1(W_n)  \mod 2 , \ \  p_1(E_n) = c_1(W_n)^2 - 4 c_2(W_n). $$
To state the result we recall some facts about the cohomology ring of $X_{k,l}$, \cite{Kreck1998}. Namely that $H^2(X_{k,l}, \mathbb{Z}) \cong \mathbb{Z}$, and that the square of its generator is the generator of $H^4(X_{k,l}, \mathbb{Z}) \cong \mathbb{Z}_{k^2+l^2+kl}$. We now state and prove

\begin{lemma}\label{lem:Characteristic_Classes}
The associated homogeneous $SO(3)$-bundle $E_n$ has
$$w_2(E_n) = n \mod 2 , \ p_1(E_n)= n^2 \mod k^2+kl+l^2.$$
\end{lemma}
\begin{proof}
The first step towards the computation is to notice that, for any $n \in \mathbb{Z}$, there is actually a homogeneous lift of $P_{n}$ to a $Spin^c(2)=U(2)$ bundle. To see this take we identify $SU(2) \times U(1) / \mathbb{Z}_2 \cong U(2)$ by the isomorphism $[(A, e^{i\theta})] \mapsto \diag(e^{i\theta}, e^{i \theta}) A$, and it is easy to see that there is a group homomorphism $\tau:U(2) \rightarrow SO(3)$ which is simply $\tau([A, e^{i \theta}])=A \in SU(2) / \mathbb{Z}_2 \cong SO(3)$.
\begin{remark}
One other way to describe this is by considering the adjoint action of $U(2)$ on its Lie algebra. This decomposes as $\mathfrak{u}(2)= \mathbb{R} \oplus \mathfrak{so}(3)$, and $U(2)$ acts on $\mathfrak{so}(3) \cong \mathbb{R}^3$ via $SO(3)$.
\end{remark}
Then, the bundle $P_{n}$ can be homogeneously lifted to a $U(2)$ bundle if and only if there is a group homomorphism $\mu_n : U(1) \rightarrow U(2)$ such that $\lambda_n = \tau \circ \mu_n$. That is indeed the case, as we can simply check that
$$\mu_n (e^{i\theta})= \left[  \begin{pmatrix}
e^{i n \theta/2} & 0 \\
0 & e^{i n \theta/2}
\end{pmatrix} , e^{i n \theta/2}    \right] \in SU(2) \times U(1) / \mathbb{Z}_2$$
does the job. Then, the canonical invariant connection on $W_n = SU(3) \times_{(U(1)_{k,l} , \mu_n)} U(2)$ is $A^n_c = - n \frac{h}{\sqrt{6} s} \otimes \diag (i,0)$ and its curvature $F^n_c= - n \frac{dh}{\sqrt{6} s} \otimes \diag (i,0)$. Then $c_1(W_n)= [i \tr (F_n^c) ]= n [dh]/ \sqrt{6}s$ with $ [dh]/ \sqrt{6}s $ being the generator of $H^2(X_{k,l}, \mathbb{Z})$, and so $w_2(E_n) = n \mod 2$. We now turn to the computation of $p_1(E_n)$ which besides $c_1(W_n)$ also requires $c_2(W_n)$ which we can check to be zero using the formula $1/2 [\tr (F_n^c \wedge F_n^c)- \tr (F_n^c)^2 ]$. Therefore, we conclude that $p_1(E_n) = n^2 \in \mathbb{Z}_{k^2+l^2+kl}$, finishing the proof of lemma \ref{lem:Characteristic_Classes}.
\end{proof}

A short computation also yields

\begin{corollary}\label{cor:charclass}
Let $n_1=k-l$, $n_2=l-m$, $n_3=m-k$, then
\begin{eqnarray}\nonumber
w_2(E_{n_1}) & = & k-l \mod 2 \\ \nonumber
w_2(E_{n_2}) & = & k \mod 2 \\ \nonumber
w_2(E_{n_3}) & = & l \mod 2 \\ \nonumber
p_1(E_{n_1}) & = & -3kl \mod k^2+kl+l^2 \\ \nonumber
p_1(E_{n_2}) & = & -3k^2 \mod k^2+kl+l^2 \\ \nonumber
p_1(E_{n_3}) & = & -3l^2 \mod k^2+kl+l^2 
\end{eqnarray}
\end{corollary}

\section{Gauge theory on $X_{k,l}$, with $(k,l) \neq (1,1)$}\label{sec:AWG2Instantons}

This section is concerned with stating and proving the main results of our paper, namely theorem \ref{thm:AbelianInstantons} and theorem \ref{thm:IrreducibleInstantons}, which classify all invariant $G_2$-instantons with gauge groups $U(1)$ and $SO(3)$, for any $G_2$-structure $\varphi \in \mathcal{C}$ as in definition \ref{def:C}. Recall that, as proved in \cite{Cabrera1996}, for $k \neq \pm l$, $l \neq \pm m$, $m \neq \pm k$, these are in fact all the homogeneous coclosed $G_2$-structures on $X_{k,l}$. Then, in theorem \ref{thm:G2_Instanton_Obstructed} we use the classification to show that in any Aloff-Wallach space as above, there are irreducible $G_2$-instantons, with gauge group $SO(3)$, which as the $G_2$-structure varies merge into the same reducible and obstructed $G_2$-instanton. This phenomenon was expected to be possible and theorem \ref{thm:G2_Instanton_Obstructed} gives plenty of explicit examples, see for instance examples \ref{ex:(k,l)=(1,-1)} and \ref{ex:(k,l)=(1,-5)}, together with their accompanying figures \ref{fig:(k,l)=(1,-1)} and \ref{fig:(k,l)=(1,-5)}, representing the merge of the $G_2$-instantons. As a consequence of theorem \ref{thm:IrreducibleInstantons} we give in section \ref{ss:distinguish} examples of Aloff-Wallach spaces where $G_2$-instantons can be used to distinguish between the two inequivalent strictly nearly parallel $G_2$-structures. More precisely, we show that in these examples there always exist invariant and irreducible $G_2$-instantons, however these live on topologically different $SO(3)$-bundles.\\
In \ref{sss:X1-1}, we fix $(k,l)=(1,-1)$ and a nearly parallel $G_2$-structure on $X_{1,-1}$. After finding the corresponding invariant $G_2$-instantons we show that any irreducible such $G_2$-instanton is not a local minimum of the Yang-Mills functional. In fact, they are saddles of the invariant Yang-Mills functional.

\subsection{$G_2$-instantons}

Before stating the main results we introduce some quantities which will simplify the notation later on
\begin{eqnarray}\nonumber
\Gamma & = & A^2B^2(m-k)+A^2C^2(l-m)+B^2C^2(k-l) \\ \nonumber
\Delta & = & A^2B^2l+A^2C^2k+B^2C^2m. 
\end{eqnarray}
Note that for a given Aloff-Wallach space $X_{k,l}$ each of these quantities only depends on the $G_2$-structure \ref{eq:G2str} and varies continuously with it.

\subsection{Abelian case}\label{ss:Abelian_X_(k,l)}

We start below by stating the result classifying $G_2$-instantons with gauge group $U(1)$. In this case, the possible homogeneous bundles are parametrised by $n \in \mathbb{Z}$, which denotes the degree of the homomorphism $\lambda_n : U(1)_{k,l} \rightarrow U(1)$ used to constructed the bundle $Q_n=SU(3) \times_{(U(1)_{k,l}, \lambda_n)} U(1)$.

\begin{theorem}\label{thm:AbelianInstantons}(Abelian case)
Let $(k,l) \neq (1,1)$ and $A$ be a $G_2$-instanton on a line bundle over $X_{k,l}$ equipped with the $G_2$-structure \ref{eq:G2str}. Then, either:
\begin{enumerate}
\item $\Delta \neq 0$, in which case there there is a unique $G_2$-instanton in any homogeneous line bundle. For instance, if $A$ lives on the bundle associated with $\lambda_n$, its connection $1$-form is
$$A= - \frac{n}{2} \left( \frac{1}{\sqrt{6} s} h + \frac{\Gamma}{3\sqrt{2} s\Delta}\omega_4 \right).$$

\item $\Delta=0$, but $\Gamma \neq 0$ in which case $A$ lives in the trivial homogenous bundle (i.e. that associated with $\lambda_0$), and $A$ is simply one of the $1$-forms $b \omega_4$, for some $b \in \mathbb{R}$.

\item $\Delta=0$ and $\Gamma=0$, in which case there is a real $1$-parameter family of such instantons on any homogeneous line bundle.
\end{enumerate}
\end{theorem}
\begin{proof}
Any Abelian $G_2$-instanton can also be interpreted as a reducible $SU(2)$-instanton. Hence, we can use the formula for the connection in the previous section. More precisely, for the instanton to be reducible we must have $a_1=a_2=a_3=0$, so
$$A^n= -\frac{n}{2 \sqrt{6} s} h + b\omega_4 .$$
Its curvature is
\begin{align*}
F^n= & F_c^n + b d \omega_4,
\end{align*}
where 
\begin{eqnarray}\nonumber
F^c_n & = & \- \frac{1}{12s^2} \left( n(k-l) \omega_{15} + n(l-m) \omega_{26}   + n(m-k)  \omega_{37} \right), \\ \nonumber
d \omega_4 & = & -\frac{1}{\sqrt{2} s} ( m \omega_{15} + k \omega_{26} + l \omega_{37} ).
\end{eqnarray}
Then we write $\psi= - D \omega_4 \wedge \Omega_2 + \frac{1}{2} \omega^2$, with $\Omega_2, \omega^2$ the pullback of differential forms on the flag manifold $\mathbb{F}_2 = SU(3)/T^2$, and determined by this relation. As in section \ref{ss:S1_Invariant_G2_Instantons}, more precisely equation \ref{eq:U(1)_Invariant_G2_Instanton}, we compute that the $G_2$-instanton equation reduces to the equations
$$(F^c_n + b \ d \omega_4) \wedge \Omega_2 =0 , \ \ (F^c_n +n b \ d \omega_4) \wedge \omega^2=0.$$
It is easy to check that $F^c_n \wedge \omega^2 =0 = d \omega_4 \wedge \omega^2$ always. We are therefore, reduced to the first equation, which turns into
$$ - n\Gamma-6\sqrt{2}s \Delta b =0,$$
with $\Gamma$, $\Delta$ are as in the beginning of this section. In particular we see that $F^c_n \wedge \Omega_2 =0$ if and only if $\Gamma =0$ and $d \omega_4 \wedge \Omega_2 =0$ if and only if $\Delta =0$. Therefore, if $\Delta \neq 0$ there is exactly one $SU(3)$-invariant instanton, whose connection form is
$$A^n= - \frac{n}{2} \left( \frac{1}{\sqrt{6} s} h + \frac{\Gamma}{3\sqrt{2} s\Delta}\omega_4 \right) \otimes T_1.$$
However, if $\Delta=0$ there are no instantons unless $n\Gamma=0$ as well, in which case there is a $1$-parameter family of instantons as we can chose $b$ arbitrarily.
\end{proof}

A few remarks are in order, related to how the existence of invariant Abelian $G_2$-instantons varies with the $G_2$-structure.

\begin{remark}\label{rem:Abelian_G2_Instantons_(k,l)}
\begin{enumerate}
\item Note that for a fixed Aloff-Wallach space $X_{k,l}$ both $\Delta$ and $\Gamma$ vary smoothly with the $G_2$-structure, and generically $\Delta \neq 0$. Note that $\Delta =0$ defines a hypersurface in the space of coclosed homogeneous $G_2$-structures.
\item Suppose that we vary the $G_2$-structure always keeping $\Gamma \neq 0$, but crossing the hypesurface defined by $\Delta =0$. We see that the instantons on the bundles $Q_n$, for $n \neq 0$, `disappear'  when $\Delta=0$ and `reappear' on the other side of the hypersurface.
\item For any $(k,l)$ it is easy to find examples where the situation $\Delta=0=\Gamma$ occurs. These equations, i.e. $\Delta= 0$ and $\Gamma=0$, can also be written as
\begin{eqnarray}\nonumber
A^2 (B^2 - C^2) l & = & B^2 (C^2 -A^2)k \\ \nonumber
C^2 (A^2-B^2) (l-k) & = & A^2 (B^2 - C^2) (k+l).
\end{eqnarray}
For example, it is easy to see that any $G_2$-structure having $A^2=B^2=C^2$ satisfies these equations.
\item The conditions $\Delta=0$ and $\Gamma=0$ are independent of scaling the metric as expected.
\item Both $\Gamma$ and $\Delta$ are independent of $D$. This can be understood directly from the proof, as follows. Recall that $(\omega, \Omega_2)$ induce an $SU(3)$-structure on the flag $\mathbb{F}_2 = SU(3)/T^2$. Then, it follows from the proof of theorem \ref{thm:AbG2instYM} that $F^c_n \wedge \Omega_2$ and $d \omega_4 \wedge \Omega_2$ are proportional to $\Gamma$ and $\Delta$ respectively. Notice that both $F^c_n$ and $d\omega_4$ are the pullback of $2$-forms from $\mathbb{F}_2$. Hence $\Gamma$, $\Delta$ measure the components of these $2$-forms in $\Lambda^{2,0}$ with respect to the complex structure on $\mathbb{F}_2$ induced by $\Omega_2$. In particular, the cannonical connection $A^c_n$, which is induced from a connection on $\mathbb{F}_2$, is pHYM with respect to $(\omega, \Omega_2)$ if and only if $\Gamma=0$.  
\item Any Abelian connection can be written as a direct sum of connections with gauge group $U(1)$, so there is no loss of generality in working with gauge group $U(1)$ when investigating Abelian connections. 
\end{enumerate}
\end{remark}

\subsection{Non-Abelian case}\label{ss:Non_Abelian_X_(k,l)}

In this section we prove theorem \ref{thm:IrreducibleInstantons}, this classifies invariant and irreducible $G_2$-instantons on $SO(3)$-bundles, with respect to the $G_2$-structures $\varphi \in \mathcal{C}$ on the $X_{k,l}$, for $k \neq \pm l$, $l \neq \pm m$, $m \neq \pm k$. Recall that in these cases, the $G_2$-structures in $\mathcal{C}$ are in fact all the homogeneous coclosed $G_2$-structures on $X_{k,l}$. Then we prove theorem \ref{thm:G2_Instanton_Obstructed} which yields examples of irreducible $G_2$-instantons that, as the $G_2$-structure varies, merge into the same reducible and obstructed $G_2$-instanton (see also examples \ref{ex:(k,l)=(1,-1)}, \ref{ex:(k,l)=(1,-5)}).\\
The reason for focusing our attention on irreducible $G_2$-instantons is that any reducible one is already taken into consideration by theorem \ref{thm:AbelianInstantons}. Recall, from the previous section, that the homogenous $SO(3)$ bundles are also parametrized by an integer $n \in \mathbb{Z}$ and we denote them by $P_n$. 

\begin{theorem}(Non-Abelian case)\label{thm:IrreducibleInstantons}
Let $(k,l) \neq (1,1)$ and $X_{k,l}$ be an Aloff-Wallach space equipped with one of the $G_2$-structures $\varphi$ in equation \ref{eq:G2str} and $n \in \mathbb{Z}$. Then, irreducible and invariant $G_2$-instantons on $P_n$ exist if and only if 
\begin{enumerate}
\item $n=k-l$ and  $\sigma_1(\varphi)= 3 \left( \frac{m}{2} - s \frac{AD}{BC} \right) \Delta + \frac{k-l}{2} \Gamma > 0$; in which case the instantons have $a_2=a_3=0$ and 
\begin{align*} 
&a_1^2 = \frac{1}{12 B^2C^2s^2} \left(  3 \left( \frac{m}{2} - s \frac{AD}{BC} \right) \Delta + \frac{k-l}{2} \Gamma  \right) \\&
b = \frac{1}{\sqrt{2}} \left( \frac{m}{2s} - \frac{AD}{BC} \right).
\end{align*}

\item $n=l-m$ and $\sigma_2(\varphi)= 3 \left( \frac{k}{2} - s \frac{BD}{AC} \right) \Delta + \frac{l-m}{2} \Gamma > 0$; in which case the instantons have $a_1 =a_3=0$ and
\begin{align*} 
&a_2^2 = \frac{1}{12A^2C^2s^2} \left( 3 \left( \frac{k}{2} - s \frac{BD}{AC} \right) \Delta + \frac{l-m}{2}\Gamma  \right), \\&
b = \frac{1}{\sqrt{2}} \left( \frac{k}{2s} - \frac{BD}{AC} \right).
\end{align*}

\item $n=m-k$ and $ \sigma_3(\varphi) = 3 \left( \frac{l}{2} - s \frac{CD}{AB} \right) \Delta + \frac{m-k}{2} \Gamma > 0 $; in which case the instantons have $a_1=a_2=0$,
\begin{align*} 
&a_3^2 = \frac{1}{12B^2A^2s^2} \left(  3 \left( \frac{l}{2} - s \frac{CD}{AB} \right) \Delta + \frac{m-k}{2} \Gamma  \right) \\&
b = \frac{1}{\sqrt{2}} \left( \frac{l}{2s} - \frac{CD}{AB} \right).
\end{align*}
\end{enumerate}
\end{theorem}
\begin{proof}
Let $A^n$ be an irreducible, invariant $G_2$-instanton on $P_n$ over $X_{k,l}$. In order to compute the instanton equations we must compute its curvature $F^n$ first. This may be found by the formula $$F^n=F^n_c + d_{A^n_c}(A^n-A^n_c)+\frac{1}{2}[A^n-A^n_c,A^n-A^n_c],$$
and the Maurer-Cartan equations. Our strategy for finding instantons will be simply to solve the equations $F^n \wedge \psi =0$ for the $a_i$ and $b$ in each of the cases listed above.\\

\textbf{Case 0} $n \neq k-l, l-m, m-k$: Here  $a_1=a_2=a_3=0$, so $A^n$, is always reducible and we immediately deduce the last bullet in the statement. We also remark that the $G_2$-instantons arising from this case are precisely those from theorem \ref{thm:AbelianInstantons}.\\

\textbf{Case 1} $n=k-l$: Here $a_2=a_3=0$, and
$$A^{k-l}=\left(  -\frac{k-l}{2 \sqrt{6} s} h + b\omega_4 \right) \otimes T_1 + a_1 \left(\omega_1 \otimes T_2 + \omega_5 \otimes T_3 \right),$$
whose curvature is
\begin{align*}
F^{k-l}= &\ - \frac{1}{12 s^2}\bigg( \left( (k-l)^2 + 6 \sqrt{2}  s m b - 24 s^2 a^2_1 \right) \omega_{15} +\left( (k-l)(l-m) + 2\sqrt{6} s k b \right) \omega_{26}  \\
&\ - \left( (k-l)(m-k)+ 2 \sqrt{6} s l b \right) \omega_{37} \bigg) \otimes T_1 \\
&\ + \frac{a_1}{\sqrt{2}} \left(  \omega_{67}  -\omega_{23} + \left( \frac{m}{s} - 2\sqrt{2}b \right) \omega_{45} \right) \otimes T_2 \\
&\ + \frac{a_1}{\sqrt{2}} \left( \omega_{27} - \omega_{36} + \left( \frac{m}{s} - 2\sqrt{2}b \right) \omega_{14}  \right) \otimes T_3.
\end{align*}
The equations resulting from $F^{k-l} \wedge \psi =0$ are:
\begin{align*}
& 6 \sqrt{2} s \Delta b - 24 B^2C^2s^2 a^2_1 +  \left( k-l \right)\Gamma = 0, \\
&\ a_1 BC \left( 2ADs+BC(2\sqrt{2}sb-m) \right) = 0.
\end{align*}
Hence, if $a_1=0$ we obtain the same reducible instanton as in case 0 and theorem \ref{thm:AbelianInstantons}, while if $a_1 \neq 0$, the solutions satisfy
\begin{align*} 
&a_1^2 = \frac{1}{12 B^2C^2s^2} \left( 3 \left( \frac{m}{2} - s \frac{AD}{BC} \right) \Delta + \frac{k-l}{2} \Gamma  \right) \\&
b = \frac{1}{\sqrt{2}} \left( \frac{m}{2s} - \frac{AD}{BC} \right).
\end{align*}
Therefore, in this case the existence of $SU(3)$-invariant irreducible instantons depends on the the sign of $\sigma_1= 3 \left( \frac{m}{2} - s \frac{AD}{BC} \right) \Delta + \frac{k-l}{2} \Gamma $.\\

\textbf{Case 2:} $n=l-m.$ As this case is very similar to case 1, we will omit the details. We must have $a_1 =a_3=0$ and if $a_2 \neq 0$, solutions to $F^{l-m} \wedge \psi =0$ must satisfy
\begin{align*} 
&a_2^2 = \frac{1}{12A^2C^2s^2} \left( 3 \left( \frac{k}{2} - s \frac{BD}{AC} \right) \Delta + \frac{l-m}{2}\Gamma  \right), \\&
b = \frac{1}{\sqrt{2}} \left( \frac{k}{2s} - \frac{BD}{AC} \right).
\end{align*}
The existence of solutions depends on the sign of $\sigma_2=3 \left( \frac{k}{2} - s \frac{BD}{AC} \right) \Delta + \frac{l-m}{2}\Gamma $.\\

\textbf{Case 3:} $n=m-k.$ Again, we will omit the details. Now $a_1=a_2=0$ and if $a_3 \neq 0,$ the equation $F^{l-m} \wedge \psi =0$ gives
\begin{align*} 
&a_3^2 = \frac{1}{12B^2A^2s^2} \left(  3 \left( \frac{l}{2} - s \frac{CD}{AB} \right) \Delta + \frac{m-k}{2} \Gamma  \right) \\&
b = \frac{1}{\sqrt{2}} \left( \frac{l}{2s} - \frac{CD}{AB} \right).
\end{align*}
The existence of solutions depends on the sign of $ \sigma_3 =  3 \left( \frac{l}{2} - s \frac{CD}{AB} \right) \Delta + \frac{m-k}{2} \Gamma $.\\

\textbf{Case 4:} $n=m-k=l-m,$ and so $n=l=-k$. Recall that in this case we have an angle parameter $\beta$. Then, the equation $F^l \wedge \psi = 0$ becomes
\begin{align*} 
& 6\sqrt{2} s \Delta b- 24A^2s^2(B^2a^2_3+C^2a^2_2) + l \Gamma = 0 \\&
a_2   \left( 2BDs+AC(2\sqrt{2}sb+l) \right) = 0 \\&
a_3 \sin(\beta) \left( 2CDs+AB(2\sqrt{2}sb-l) \right) = 0 \\&
a_3 \cos(\beta) \left( 2CDs+AB(2\sqrt{2}sb-l) \right) = 0.
\end{align*}
squaring and summing the last two equations we are left with
$$a_3 ( 2CDs+AB(2\sqrt{2}sb-l) )=0.$$
This together with the second equation then implies that either $a_3=0$ or $a_2=0$, in which case we can then use an invariant gauge transformation to set $\beta=0$. We have then reduced this case to the cases 2 and case 3 above. In particular, the existence of $G_2$-instantons is determined by the signs of $\sigma_3$ and $\sigma_2$ (note that here we have $l=-k$). \\

\textbf{Cases 5 and 6:} These cases exhibit the same phenomena as in the last one and so reduce to the cases 1,2 and 4 above.
\end{proof}

\begin{remark}
Fix $X_{k,l}$ and the bundle $P_{k-l}$, then the first bullet in theorem \ref{thm:IrreducibleInstantons} shows that for a $G_2$-structure $\varphi$ so that $\sigma_1(\varphi)>0$ there are two irreducible $G_2$-instantons. In addition, we also have a reducible $G_2$-instanton given by theorem \ref{thm:AbelianInstantons} (with $n=k-l$). Varying $\varphi$ so that $\sigma_1(\varphi) \searrow 0$ the two irreducible, invariant $G_2$-instantons existent when $\sigma_1>0$ merge with the reducible Abelian $G_2$-instanton from theorem \ref{thm:AbelianInstantons}. Indeed, it is easy to check that when $\sigma_1=0$ (and $\Delta \neq 0$) then $a_1=0$ and $b=-\frac{n \Gamma}{6 \sqrt{2} s \Delta}$. We shall see below that the resulting $G_2$-instanton is obstructed. From the second and third bullet in the statement, a similar phenomena holds on the bundles $P_{l-m}$ and $P_{m-k}$.
\end{remark}

\begin{theorem}\label{thm:G2_Instanton_Obstructed}
Let $n =k-l$, and suppose $\lbrace \varphi(s) \rbrace_{s \in \mathbb{R}}$ is a continuous family of homogeneous, coclosed $G_2$-structures such that $\sigma_1(\varphi(s)) >0$, for $s<0$ and $\sigma_1(\varphi(s))<0$, for $s>0$. Then, as $s \nearrow 0$ the two irreducible $G_2$-instantons on $P_{n}$ from theorem \ref{thm:IrreducibleInstantons} merge and become the same reducible and obstructed $G_2$-instanton when $s \geq 0$.
\end{theorem}
\begin{proof}
Recall that an invariant connection on $P_{k-l}$ can be written as $A=A^n_c + b \omega_4 \otimes T_1 + a_1(\omega_1 \otimes T_2 + \omega_5 \otimes T_3 )$. Similarly, an invariant $1$-form with values in the adjoint bundle can be written as $a=f \omega_4 \otimes T_1 + g(\omega_1 \otimes T_2 + \omega_5 \otimes T_3 ) $, for some $f,g \in \mathbb{R}$. Using these it is easy to compute 
\begin{eqnarray}\nonumber
d_A a & = & \left( x d \omega_4 + 4 a_1 y \omega_1 \wedge \omega_5 \right) \otimes T_1 \\ \nonumber
& & + (y  d \omega_1 + 2(b y + x a_1) \omega_5 \wedge \omega_4) \otimes T_2 \\ \nonumber
& & + (y  d \omega_5 - 2(b y + x a_1) \omega_1 \wedge \omega_4) \otimes T_3 .
\end{eqnarray}
We are now ready to find the invariant Lie algebra valued $1$-forms $a$ which lie in the cokernel of the deformation operator of the $G_2$-instanton equation $L( \cdot)= \ast (d_A \cdot \wedge \psi)$. As the $G_2$-structure is coclosed $L$ is self-adjoint and we can identify the cokernel with its own kernel. Hence $a \in \ker(L)$ if and only if $d_A a \wedge \psi=0$, which we compute to be equivalent to
\begin{eqnarray}
\sqrt{2} \Delta x -8B^2 C^2 s a_1 y & = & 0 \\
4 B C s a_1 x + \left( \sqrt{2} \left( 2 \frac{AD}{BC} s - m \right) + 4 s b \right) B C y & = & 0.
\end{eqnarray}
Hence, there is a nonzero solution $(x,y)$ if and only if the linear operator in the left hand side is not invertible, i.e. its determinant vanishes
\begin{equation}
\det=32 B^3 C^3 s^2 a^2 + (4AD s - 2 BC m + 4 \sqrt{2} BC bs) \Delta =0.
\end{equation}
Inserting into the equation above the formulas in theorem \ref{thm:IrreducibleInstantons} for the reducible instantons when $n=k-l$ we obtain
$$\det = 8BC \sigma_1/3,$$
which vanishes if and only if $\sigma_1=0$.
We have thus proved that as the instantons from theorem \ref{thm:IrreducibleInstantons} on $P_{k-l}$ merge, when $\sigma_1=0$ they become reducible and obstructed before disappearing.
\end{proof}

\begin{remark}
A similar statement to theorem \ref{thm:G2_Instanton_Obstructed} holds for $n=l-m$ and $n=m-k$, with $\sigma_1$ replaced by $\sigma_2$ and $\sigma_3$ respectively.
\end{remark}

Here are two examples of this phenomenon.

\begin{example}\label{ex:(k,l)=(1,-1)}
On the Aloff-Wallach space $X_{1,-1}$ consider the $G_2$-structures given by $B=1$, $C=1$, $D=1$ with $A$ allowed to vary freely in order to make $\sigma_1$ change sign. Then, as $A$ varies the condition for irreducible $G_2$-instantons on $P_{2}$ to exist is that $\sigma_1(\varphi)= 2(1-A^2)$ be positive, this happens if and only if $A^2 <1$. See figure \ref{fig:(k,l)=(1,-1)} for a plot of $a_1$ (the ``irreducible part'' of the connections) as $A$ varies. There one can clearly see that the irreducible $G_2$-instantons merge into the same reducible and obstructed (by theorem \ref{thm:G2_Instanton_Obstructed}) $G_2$-instanton.
\begin{figure}[]
\centering
\includegraphics[scale=0.4]{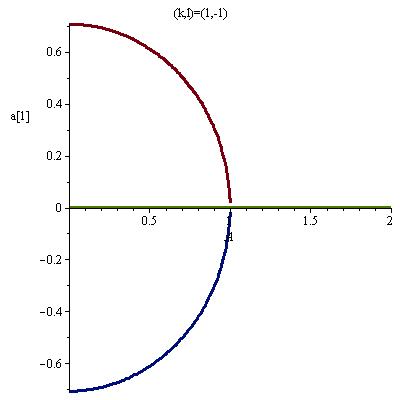}
\caption{\label{fig:(k,l)=(1,-1)} Instantons on $P_2$ over $X_{1,-1}$}
\end{figure}
\end{example}

\begin{example}\label{ex:(k,l)=(1,-5)}
Similarly we consider $G_2$-intantons on $P_{3}$ over $X_{1,-5}$, equipped with the $G_2$-structures having $B=C=D=1$. In this case the existence of irreducible $G_2$-instantons is controlled by the positivity of $\sigma_1(\varphi)= (A^2-1)(12 \sqrt{7} A-42)$, which is positive if and only if $A^2<1$ or $A>\sqrt{7}/2$. Figure \ref{fig:(k,l)=(1,-5)} below illustrates the two irreducible $G_2$-instantons merging into the same reducible and obstructed $G_2$-instanton.
\begin{figure}[]
\centering
\includegraphics[scale=0.4]{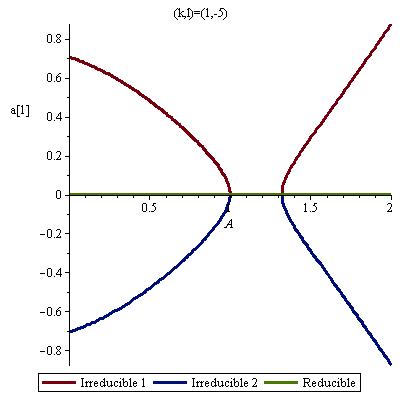}
\caption{\label{fig:(k,l)=(1,-5)} Instantons on $P_6$ over $X_{1,-5}$}
\end{figure}
\end{example}

\begin{remark}
The phenomenon described above can be interpreted as the $G_2$ analogue of a family of stable holomorphic bundles in a K\"ahler manifold, that become polystable as either the K\"ahler metric or the complex structure varies, see for example \cite{Anderson2011} and \cite{Anderson2010}. We thank Mark Stern for these references.
\end{remark}


\subsection{Distinguishing strictly nearly parallel structures}\label{ss:distinguish}

Suppose that $k \neq \pm l, l \neq \pm m, m \neq \pm k$. As remarked in section \ref{sec:Aloff-Wallach}, it is shown in \cite{Cabrera1996} that the system \ref{eq:nearlyparallelG2eqns} yields two inequivalent solutions $\varphi^{np_1}, \varphi^{np_2} \in \mathcal{C}$, which are strictly nearly parallel. In this section we will give examples of $X_{k,l}$ where the $G_2$-instantons can be used to distinguish between $\varphi^{np_1}$ and $\varphi^{np_2}$. More precisely, we shall prove that in many examples of $k,l$ the structures $\varphi^{np_1}$ and $\varphi^{np_2}$ do admit invariant and irreducible $G_2$-instantons with gauge group $SO(3)$. However, the $G_2$-instantons live on topologically different $SO(3)$-bundles.\\
To fix notation, let $\varphi^{+}$ denote the solution of \ref{eq:nearlyparallelG2eqns} that satisfies $C(\varphi^{+}) >0, D(\varphi^{+}) >0$, and let $\varphi^{-}$ denote the solution satisfying $C(\varphi^{-})>0,D(\varphi^{-})<0.$ While it is possible to solve the equations \ref{eq:nearlyparallelG2eqns} symbolically, the resulting formulae are extremely unwieldy, so we will instead just write decimal approximations.

\begin{example}
($k=1, l=2$) On $X_{1,2}$ we find that $\varphi^{+}$ is determined by
$$A=2.82249, \ B=2.29632, \ C=1.79654, \ D=2.49609,$$
so that
$$\sigma_1(\varphi^{+})=-694.91837, \ \sigma_2(\varphi^{+})=-357.13002, \ \sigma_3(\varphi^{+})=102.96860.$$
On the other hand, $\varphi^{-}$ satisfies
$$A=1.69915, \ B=2.63936, \ C=2.72083, \ D=-1.72713,$$
so that
$$\sigma_1(\varphi^{-})=257.21323, \ \sigma_2(\varphi^{-})=-623.28938, \ \sigma_3(\varphi^{-})=-676.14197.$$
Hence, theorem \ref{thm:IrreducibleInstantons} implies that for $\varphi^{+}$ irreducible, invariant $G_2$-instantons exist only on the bundle $P_{-4},$ while for $\varphi^{-}$ irreducible, invariant $G_2$-instantons exist only on the bundle $P_{-1}$. These bundles are topologically distinct, indeed using the formulae from corollary \ref{cor:charclass} we find that $w_2(E_{-4}) \equiv 0 \;(\bmod\; 2),\; p_1(E_{-4}) \equiv 2 \;(\bmod\; 7)$, while $w_2(E_{-1}) \equiv 1 \;(\bmod\; 2),\; p_1(E_{-1}) \equiv 1 \;(\bmod\; 7)$.
\end{example}

\begin{example}
($k=1, l=3$) On $X_{1,3}$ we find that $\varphi^{+}$ is determined by
$$A=2.81314, \ B=2.38489, \ C=1.76003, \ D=2.30416,$$
so that
$$\sigma_1(\varphi^{+})=-1304.73725, \ \sigma_2(\varphi^{+})= -794.17740, \ \sigma_3(\varphi^{+})= 286.31370.$$
On the other hand, $\varphi^{-}$ satisfies
$$A=1.70181, \ B=2.61482, \ C=2.73734, \ D=-1.76385,$$
so that
$$\sigma_1(\varphi^{-})=468.21163, \ \sigma_2(\varphi^{-})=-1124.80823, \ \sigma_3(\varphi^{-})=-1272.28946.$$
Hence, theorem \ref{thm:IrreducibleInstantons} implies that for $\varphi^{+}$ irreducible, invariant $G_2$-instantons exist only on the bundle $P_{-5},$ while for $\varphi^{-}$ irreducible, invariant $G_2$-instantons exist only on the bundle $P_{-2}$. Also in this case we can use the formulae from corollary \ref{cor:charclass} to find $w_2(E_{-5}) \equiv 1 \;(\bmod\; 2),\; p_1(E_{-5}) \equiv 12 \;(\bmod\; 13)$, and $w_2(E_{-2}) \equiv 0 \;(\bmod\; 2),\; p_1(E_{-2}) \equiv 4 \;(\bmod\; 13)$, so the bundles are topologically distinct.
\end{example}

\begin{example}($k=1, l=4$) On $X_{1,4}$ we find that $\varphi^{+}$ satisfies
$$A=2.80647, B=2.42496, C=1.74612, D=2.20834,$$
so that
$$\sigma_1(\varphi^{+})=-2113.76099, \sigma_2(\varphi^{+})=-1378.20704, \sigma_3(\varphi^{+})=526.44201.$$
On the other hand, $\varphi^{-}$ satisfies
$$A=1.01066,  B=2.42496, C=1.74612,  D=-1.79228,$$
so that
$$\sigma_1(\varphi^{-})=349.25330, \: \sigma_2(\varphi^{-})= -1593.71394, \: \sigma_3(\varphi^{-})= -823.16662.$$
Hence, theorem \ref{thm:IrreducibleInstantons} implies that for $\varphi^{+}$ irreducible, invariant $G_2$-instantons exist only on the bundle $P_{-6},$ while for $\varphi^{-}$ irreducible, invariant $G_2$-instantons exist only on the bundle $P_{-3}$. Using the formulae from corollary \ref{cor:charclass} we find $w_2(E_{-6}) \equiv 0 \;(\bmod\; 2),\; p_1(E_{-6}) \equiv 15 \;(\bmod\; 21)$, and $w_2(E_{-3}) \equiv 1 \;(\bmod\; 2),\; p_1(E_{-3}) \equiv 9 \;(\bmod\; 21)$, so the bundles are topologically distinct.
\end{example}

\begin{example}
($k=2, l=3$) On $X_{2,3}$ we find that $\varphi^{+}$ satisfies
$$A=2.82707, B=2.19724, C=1.84821, D=2.66829,$$
so that
$$\sigma_1(\varphi^{+})=-1857.93578, \sigma_2(\varphi^{+})=-753.70309, \sigma_3(\varphi^{+})=107.33579.$$
On the other hand, $\varphi^{-}$ satisfies
$$A=1.69781, B=2.65772, C=2.70655, D=-1.70795,$$
so that
$$\sigma_1(\varphi^{-})=705.20889, \: \sigma_2(\varphi^{-})= -1726.54024, \: \sigma_3(\varphi^{-})=-1812.54120.$$
Hence, theorem \ref{thm:IrreducibleInstantons} implies that for $\varphi^{+}$ irreducible, invariant $G_2$-instantons exist only on the bundle $P_{-7},$ while for $\varphi^{-}$ irreducible, invariant $G_2$-instantons exist only on the bundle $P_{-1}$. Using the formulae from corollary \ref{cor:charclass} we find $w_2(E_{-7}) \equiv 1 \;(\bmod\; 2),\; p_1(E_{-7}) \equiv 11 \;(\bmod\; 19)$, and $w_2(E_{-1}) \equiv 1 \;(\bmod\; 2),\; p_1(E_{-1}) \equiv 1 \;(\bmod\; 19)$, so the bundles are topologically distinct.
\end{example}

\begin{example}
$k=2, l=11.$

On $X_{2,11}$ we find that $\varphi^{+}$ satisfies
$$A=2.80000, B = 2.45576, C = 1.73649, D = 2.13220,$$
so that
$$\sigma_1(\varphi^{+})=-14809.57254, \sigma_2(\varphi^{+})=-10158.19056, \sigma_3(\varphi^{+})=4009.81206.$$
On the other hand, $\varphi^{-}$ satisfies
$$A = 1.70630, B = 2.58424, C = 2.75458, E = -1.82250,$$
so that
$$\sigma_1(\varphi^{-})=5116.36820, \: \sigma_2(\varphi^{-})= -12243.99444, \: \sigma_3(\varphi^{-})= -14559.71627.$$

Hence, theorem \ref{thm:IrreducibleInstantons} implies that for $\varphi^{+}$ irreducible, invariant $G_2$-instantons exist only on the bundle $P_{-15},$ while for $\varphi^{-}$ irreducible, invariant $G_2$-instantons exist only on the bundle $P_{-9}$.  Using the formulae from corollary \ref{cor:charclass} we find $w_2(E_{-15}) \equiv 1 \;(\bmod\; 2),\; p_1(E_{-15}) \equiv 78 \;(\bmod\; 147)$, and $w_2(E_{-9}) \equiv 1 \;(\bmod\; 2),\; p_1(E_{-9}) \equiv 81 \;(\bmod\; 147)$, so the bundles are topologically distinct.
\end{example}

\begin{remark}
We did not find any Aloff-Wallach space for which one of the strictly nearly parallel $G_2$-structures does not admit irreducible, invariant $G_2$-instantons with gauge group $SO(3)$.
\end{remark}

\subsection{Yang-Mills connections}\label{sss:AW_YM}

It is interesting to consider the question: what conditions on a $G_2$-structure ensure that a $G_2$-instanton is a Yang-Mills connection? Proposition \ref{prop:NearlyG2YM} says that this is the case for parallel and nearly parallel $G_2$-structures. In this section we shall characterise the homogeneous coclosed $G_2$-structures $\varphi \in \mathcal{C}$ for which an Abelian $G_2$-instanton is a critical point for the Yang-Mills energy.

\begin{proposition}\label{thm:AbG2instYM}
Equip $X_{k,l}$ with a $G_2$-structure \ref{eq:G2str} such that $\Delta \neq 0.$ Let $A^n$ be the unique $G_2$-instanton on the line bundle associated with $\lambda_n$. Then $A$ is a critical point for the Yang-Mills energy if and only if the $G_2$-structure satisfies
\begin{equation}\label{eq:ABG2YMcondition}
A^2B^2(A^2-B^2)l+A^2C^2(C^2-A^2)k+B^2C^2(B^2-C^2)m=0
\end{equation}
\end{proposition}

\begin{proof}
From the proof of theorem \ref{thm:AbelianInstantons} we have
$$A^n= - \frac{n}{2} \left( \frac{1}{\sqrt{6} s} h + \frac{\Gamma}{3\sqrt{2} s\Delta}\omega_4 \right) \otimes T_1.$$

The Yang-Mills energy for an invariant Abelian connection
$A^n=\left( -\frac{n}{2 \sqrt{6} s} h + b\omega_4 \right) \otimes T_1.$
is
\begin{align*}
E(b)= & \frac{1}{144s^2} \Bigl( \frac{1}{A^4} \bigl( 6\sqrt{2}bms+n(k-l)\bigr)^2 + \frac{1}{B^4} \bigl( 6\sqrt{2}bms+n(l-m)\bigr)^2 \\
&\ + \frac{1}{C^4} \bigl( 6\sqrt{2}bms+n(m-k)\bigr)^2 \Bigr).
\end{align*}
Then, we require that at the $G_2$-instanton, i.e. $b=-\frac{n \Gamma}{6 \sqrt{2} s \Delta}$ be a critical point of $E(b)$, which immediately yields \ref{eq:ABG2YMcondition}.\\
For completeness we also remark that, in general, the critical points of $E$ have
$$b=\frac{n}{6\sqrt{2}s}\frac{A^4B^4l(k-m)+A^4C^4k(m-l)+B^4C^4m(l-k)}{A^4B^4l^2+A^4C^4k^2+B^4C^4m^2}.$$
\end{proof}

\begin{remark}
\begin{enumerate}
\item If $\Delta = 0$ then only one of the $G_2$-instantons in the 1-parameter family described in theorem \ref{thm:AbelianInstantons} is a critical point for the Yang-Mills energy.

\item For a given $X_{k,l}$ the condition \ref{eq:ABG2YMcondition} describes a hypersurface in the space of homogeneous coclosed $G_2$-structures, containing the nearly parallel $G_2$-structures.

\item One can carry out a similar analysis to determine conditions on the $G_2$-structure so that the irreducible $G_2$-instantons decribed in theorem \ref{thm:IrreducibleInstantons} are Yang-Mills. The space of such $G_2$-structures is cut out in $\mathcal{C}$ by two real algebraic equations.
\end{enumerate}
\end{remark}

\subsection{For a nearly parallel structure on $X_{1,-1}$}\label{sss:X1-1}

\subsubsection{$G_2$-instantons}\label{sss:X1-1_G2_Instantons}

We shall now see an example of a nearly parallel $G_2$-structure on a Aloff-Wallach space, namely $X_{1,-1}$, for which such instantons do exist. The precise statement we shall prove is

\begin{theorem}\label{thm:G2_Instantons_X_(1,-1)_Nearly_Parallel}
Let $\varphi$ be the nearly parallel $G_2$-structure on $X_{1,-1}$.
\begin{enumerate}
\item For each $n$, there is a unique, invariant, $G_2$-instanton on the line bundle $L_n= SU(3) \times_{ U(1)_{1,-1} , \rho_n} \mathbb{C}$.

\item Let $A$ be an irreducible and invariant $G_2$-instanton, with gauge group $SO(3)$ on $X_{1,-1}$. Then, $A$ lives on the bundle $P_{-1}$. Moreover, such instantons do exist.
\end{enumerate}
\end{theorem}

The rest of this section is dedicated to proving this result. First we must obtain the strictly nearly parallel $G_2$-structure on $X_{1,-1}$. This is of the form \ref{eq:G2str}, with
$$A= -4 \sqrt{\frac{2}{5}}, \ B=\frac{4}{15} \sqrt{75+15 \sqrt{5}}, \ C=-\frac{4}{15} \sqrt{75-15 \sqrt{5}}, \ D= - \frac{16}{45} \sqrt{30},$$
as a straightforward computation shows. We shall now compute $G_2$-instantons for this structure. Starting with Abelian ones on the bundles $L_n = SU(3) \times_{\lambda_n} \mathbb{C}$. The invariant connections are of the form $\frac{n}{2}h + a_4 \omega_4$ and the $G_2$-instanton equation $(\frac{n}{2} dh + a_4 d \omega_4 ) \wedge \psi=0$ gives
$$\left( - \frac{256 n}{45} + \frac{512 a_4}{75} \sqrt{30}   \right) \omega_{1234567}=0.$$
Hence, we must have $a_4 = n \frac{\sqrt{30}}{36}$ and the resulting $G_2$-instanton has curvature
$$F=\frac{n}{2} \omega_{15} - \frac{n}{4} (\omega_{26} + \omega_{37}) + \frac{n \sqrt{5}}{12}  ( \omega_{37} - \omega_{26} ).$$
We turn now to non-Abelian $G_2$-instantons, namely those with gauge group $SO(3)$ that we constructed before. We start by considering the case $n=k-l=2$, i.e. instantons on the bundle on $P_2 = SU(3) \times_{\lambda_2} SO(3)$. Inserting the $A,B,C,D$, associated with the nearly parallel $G_2$-structure, into our general formula one can check that the quantity inside the square root is negative and so there are no invariant, irreducible, $G_2$-instantons on $P_{2}$. In fact, to be a little more explicit we shall explain all the steps underlying that computation in this case. First, the more general invariant connection on $P_{2}$ has $a_2=a_3=0$ and so is of the form $A= (\frac{1}{\sqrt{2}} h + a_4 \omega_4 ) \otimes T_1 + a_1 (\omega_1 \otimes T_2 + \omega_5 \otimes T_3)$. We compute its curvature $F_A$ as before and equate $F_A \wedge \psi=0$, which yields the following equations
\begin{eqnarray}
1-\frac{3}{5} \sqrt{30} a_4-4a_1^2 & = & 0 \\
a_1 (2 \sqrt{2} - \sqrt{15}a_4) & = & 0.
\end{eqnarray}
From the second of this we see that either $a_1=0,$ in which case the connection is reducible, or $a_4= 2 \sqrt{ \frac{ 2}{15}}$. Inserting this into the first equation we then have to solve $-7/5-4a_1^2=0$, which has no real solutions. Alternatively we could have just evaluated $\sigma_1=-14336/225$, which being negative shows that there no irreducible instantons on $P_{2}$.\\
We analyze now the case when $n=l-m$ or $n=m-k$ as in both these cases we have $n=-1$. In this case an invariant connection must have $a_1=0$, while $a_2$ and $a_3$ can be nonzero. However, as we have seen in our analysis of the general case, the $G_2$-instanton equations imply that at least one of these vanish. In fact, after inserting the values of $A,B,C,D$ into the formulae of theorem \ref{thm:IrreducibleInstantons}, we can check that $\sigma_2 <0$ and $\sigma_3>0$. Hence, there are irreducible $G_2$-instantons and any such has $a_2=0$ and
$$a_3 = \pm \frac{\sqrt{15}}{60} \sqrt{5-\sqrt{5}} \sqrt{13 \sqrt{5}-25} , \ \ a_4 =- \frac{\sqrt{6}}{36} \frac{45-7 \sqrt{5}}{5 + \sqrt{5}}.$$
The quantities appearing inside the square root are positive and so these solutions do correspond to genuine $G_2$-instantons for the nearly parallel $G_2$-structure on $X_{1,-1}$. For completeness we write the curvature of such an instanton in the usual way with
\begin{eqnarray}\nonumber
F_1 & = & -\frac{1}{2} \  \omega_{15} - \frac{1}{3 (5 + \sqrt{5})} \left( (15- \sqrt{5}) \omega_{26} + (20 - 9 \sqrt{5}) \omega_{37} \right) \\ \nonumber
F_2 & = & \mp \frac{\sqrt{2}}{72} \frac{\sqrt{5 - \sqrt{5}} \sqrt{13 \sqrt{5} - 25}}{5 + \sqrt{5}} \left( 3 \sqrt{3} (1+ \sqrt{5}) (\omega_{12}-\omega_{56} ) + 16 \omega_{47} \right) \\ \nonumber
F_3 & = & \pm \frac{\sqrt{2}}{72} \frac{\sqrt{5 - \sqrt{5}} \sqrt{13 \sqrt{5} - 25}}{5 + \sqrt{5}} \left( 3 \sqrt{3} (1+ \sqrt{5}) (\omega_{16}-\omega_{25} ) - 16 \omega_{34} \right).
\end{eqnarray}

\subsubsection{Yang-Mills unstable $G_2$-instantons}\label{sss:YMX1-1}

Let $A$ be a $G_2$-instanton for a be a nearly parallel $G_2$-structure $\varphi$, i.e. such that $d \varphi = \lambda \psi$. We have seen, in proposition \ref{prop:NearlyG2YM}, that such $G_2$-instantons are actually Yang-Mills connections. Moreover, equation \ref{eq:action} and the discussion below it show that in the torsion free case a $G_2$-instanton minimizes the Yang-Mills energy, and so is Yang-Mills stable. That need not be the case for strictly nearly parallel $G_2$-structures as we now show with a counterexample on the nearly parallel $X_{1,-1}$.

\begin{proposition}\label{prop:UnstableX1-1}
The irreducible $G_2$-instantons constructed in the second item of theorem \ref{thm:G2_Instantons_X_(1,-1)_Nearly_Parallel}, over the nearly parallel $X_{1,-1}$, are unstable as Yang-Mills connections.
\end{proposition}

\begin{proof}
In order to demonstrate instability, it will be sufficient to consider the Yang-Mills energy only for invariant connections with $a_2=0$. We will denote $a_3$ simply by $a$. The Yang-Mills energy for the connection
$$A^{-1}=\left( \frac{h}{2\sqrt{2}} + b\omega_4 \right) \otimes T_1 + a \left(\omega_3 \otimes T_2 + \omega_7 \otimes T_3 \right)$$
on $P_{-1}$ is
\begin{align*}
E(a,b)&\ =\frac{25}{4096}+\frac{50625}{4096}\frac{(2\sqrt{6}b+1)^2}{(75+15\sqrt{5})^2}+\frac{50625}{4096}\frac{(2\sqrt{6}b+8a^2-1)^2}{(75-15\sqrt{5})^2} \\ &\ +\frac{1125}{256}\frac{a^2}{75+15\sqrt{5}}+\frac{91125}{147456}\frac{a^2(4\sqrt{3}b+3\sqrt{2})^2}{75-15\sqrt{5}}.
\end{align*}

A routine calculation shows that, as expected from Proposition \ref{prop:NearlyG2YM}, the $G_2$-instantons 
$$a=\pm\frac{1}{6}\sqrt{12\sqrt{5}-21}, b=-\frac{1}{36}\frac{\sqrt{6}(202\sqrt{5}-345)}{14\sqrt{5}-5}$$
are critical points for this energy. For both of these $G_2$-instantons the determinant and trace of the Hessian of $E(a,b)$ are
\begin{align*}
& \det(\text{Hess}(E))=\frac{1265625}{81920000}\frac{250875-126967\sqrt{5}}{4580-1364\sqrt{5}} < 0 \\
&\ \tr(\text{Hess}(E))=\frac{1874}{512000}\frac{54305\sqrt{5}-28931}{402-56\sqrt{5}} > 0,
\end{align*}
and so they are critical points of index one, hence unstable as Yang-Mills connections.
\end{proof}
\begin{figure}[h]
\centering
\includegraphics[scale=0.4]{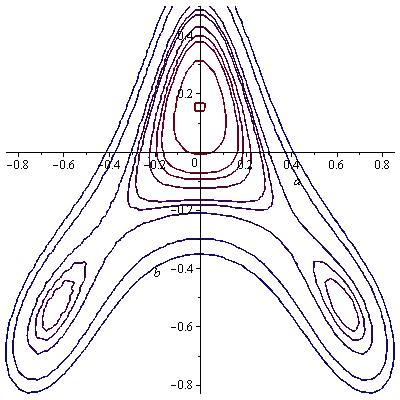}
\caption{\label{fig:YMfunctional}  Level sets of the invariant Yang-Mills functional with $a_2=0$. One can see three local minimuns, the global minimum is on top and is a reducible $G_2$-instanton. There are also two saddles which lie on straight lines from the reducible $G_2$-instanton to the other local minimuns. Those saddle points correspond to the irreducible $G_2$-instantons.}
\end{figure}
\begin{remark}
\begin{enumerate}
\item It is not difficult to check that the reducible $G_2$-instanton is the global minimum for the Yang-Mills energy among all invariant connections on the bundle $P_{-1}$ (i.e. even when $a_2 \neq 0.$)
\item When restricting to the $a_2=0$ case there are three local minimums of the Yang-Mills energy: the reducible $G_2$-instanton, and a pair of Yang-Mills connections that are not $G_2$-instantons, see figure \ref{fig:YMfunctional}. The two irreducible $G_2$-instantons are the two saddles in that figure.
\end{enumerate}
\end{remark}

\section{Gauge theory on $X_{1,1}$}\label{sec:X11}

In this section we study $G_2$-instantons on $X_{1,1}$, with respect to the $G_2$-structures \ref{eq:G2str}. This case was excluded from the previous section, since here the existence result for invariant connections, lemma \ref{lem:ConnectionForms_Case_(1,1)}, requires a separate analysis. We further remark that in the case $(k,l)=(1,1)$ the form \ref{eq:G2str} for the $G_2$-structure does not yield the most general homogeneous coclosed $G_2$-structure. We start by proving theorems \ref{thm:abelianinstantonsX11} and \ref{thm:irreducibleinstantonsX11}, which are the analogues of theorems \ref{thm:AbelianInstantons} and \ref{thm:IrreducibleInstantons}, classifying Abelian and non-Abelian $G_2$-instantons on $X_{1,1}$. Then, in theorem \ref{thm:G2_Instanton_Obstructed_(1,1)} we prove that the same phenomenon as in theorem \ref{thm:G2_Instanton_Obstructed} occurs in the case of $X_{1,1}$. Namely, we prove that on $X_{1,1}$ there are irreducible invariant $G_2$-instantons, with gauge group $SO(3)$, that as the $G_2$-structure varies merge into the same reducible and obstructed one.\\
Then, in section \ref{sss:X11} we specialize to a certain subfamily of $G_2$-structures in $\mathcal{C}$ and write down the explicit formulas for the $G_2$-instantons in this subfamily. The main results here are theorems \ref{thm:X11} and \ref{thm:X11Irreducible}. In particular, this last one proves that there are two bundles (one of which is the trivial one) carrying irreducible $G_2$-instantons, with gauge group $SO(3)$, for a continuous family of $G_2$-structures. Also, we prove in theorem \ref{thm:Limit} that as the fibres of a projection $\pi: X_{1,1} \rightarrow \overline{\mathbb{CP}}^2$ collapse, the irreducible $G_2$-instantons in the trivial bundle converge to the pullback of a connection from $\overline{\mathbb{CP}}^2$. We also show this cannot be true for the $G_2$-instantons in the other bundle. Finally, in corollary \ref{cor:NoInstantons} we prove that: while there are no invariant irreducible $G_2$-instantons with gauge group $SO(3)$ for the tri-Sasakian structure on $X_{1,1}$, these do exist for the strictly nearly parallel one.

\subsection{Abelian case}\label{ss:Abelian_X_(1,1)}

The following theorem is the analog of theorem \ref{thm:AbelianInstantons}, classifying invariant $G_2$-instantons on $X_{1,1}$ with gauge group $U(1)$. Note that for $(k,l)=(1,1),$
$$\Gamma=3A^2(C^2-B^2), \: \: \Delta=A^2B^2+A^2C^2-2B^2C^2.$$

\begin{theorem}\label{thm:abelianinstantonsX11}
Equip $X_{1,1}$ with the $G_2$-structure \ref{eq:G2str}. Let $A^n$ be an invariant $G_2$-instanton on the line bundle $Q_n$ over $X_{1,1}$. Then,
\begin{enumerate}
\item if $AD+BC \neq 0$, either
\begin{enumerate}
\item $\Delta \neq 0$, in which case $A^n$ is the unique $G_2$-instanton on $Q_n$. Its connection $1$-form is
$$A^n= - \frac{n}{2} \left( \frac{1}{\sqrt{6}} h + \frac{\Gamma}{3\sqrt{2}\Delta}\omega_4 \right).$$ 

\item $\Delta=0$, but $\Gamma \neq 0$ in which case $n=0$ and so $A$ lives in the trivial homogenous bundle (i.e. that associated with $\lambda_0$), and $A^n$ is simply one of the $1$-forms $b \omega_4$, for some $b \in \mathbb{R}$.

\item $\Delta=0$ and $\Gamma=0$, in which case there is a real $1$-parameter family of such instantons.
\end{enumerate}

\item while if $AD+BC=0$, then either
\begin{enumerate}
\item $\Delta \neq 0$, in which case there is a real $2$-parameter family of such $G_2$-instantons on $Q_n$, and $A^n$ is given by
$$A^n= - \frac{n}{2} \left( \frac{1}{\sqrt{6}} h + \frac{\Gamma}{3\sqrt{2}\Delta}\omega_4 + a_1 \omega_1 + a_5 \omega_5 \right),$$
for some $a_1, a_5 \in \mathbb{R}.$

\item $\Delta=0$, but $\Gamma \neq 0$ in which case $n=0$ and so $A$ lives in the trivial homogenous bundle (i.e. that associated with $\lambda_0$), and $A$ is simply one of the $1$-forms $b \omega_4+a_1 \omega_1 + a_5\omega_5$, for some $a_1, a_5, b \in \mathbb{R}$.

\item $\Delta=0$ and $\Gamma=0$, in which case there is a real $3$-parameter family of such instantons.
\end{enumerate}
\end{enumerate}
\end{theorem}

\begin{proof}
Any Abelian $G_2$-instanton can be interpreted as a reducible $SO(3)$ instanton. Hence, we can use the formulae \ref{sss:Case_Splitting_11} for the connection form
$$A^n= -\frac{n}{2}\frac{h}{\sqrt{6}}+b\omega_4+a_1\omega_1+a_5\omega_5 .$$

For this connection the $6$-form $F^n \wedge \psi$ becomes
\begin{align*}
& \sqrt{2}BC(AD+BC) \left( a_1\omega_{234567} + a_5 \omega_{123467} \right) \\
&\ + \left(\frac{1}{\sqrt{2}} \Delta a_4 + \frac{1}{12}n\Gamma \right) \omega_{123567}.
\end{align*}
Equating this to zero the result follows from splitting into the various possible cases and simple algebraic manipulations.
\end{proof}

\begin{remark}
\begin{itemize}
\item The condition that $\Delta=0=\Gamma$ and $AD+BC=0$ can occur. Take for example a $G_2$-structure with $A=B=C$ and $D=-A$. In this case there is a $3$-parameter family of invariant $G_2$-instantons on any complex line bundle over $X_{1,1}$.

\item These existence of this real $3$-parameter family for these $G_2$-structures can be understood in light of proposition \ref{prop:3_Parameter_Family}.
\end{itemize}
\end{remark}

\subsection{Non-Abelian case}\label{ss:Non_Abelian_X_(1,1)}

Next we have the analog of theorem \ref{thm:IrreducibleInstantons}, classifying invariant, irreducible $G_2$-instantons over $X_{1,1}$ with gauge group $SU(2)$.

\begin{theorem}\label{thm:irreducibleinstantonsX11}
Equip $X_{1,1}$ with the $G_2$-structure \ref{eq:G2str}. Then, invariant, irreducible $G_2$-instantons exist on the bundle $P_{\lambda_n}$ if and only if
\begin{enumerate}
\item $n=0$ and $-\Delta \left( 1 + \frac{AD}{BC} \right) > 0$, in which case the $G_2$-instanton has connection 1-form
$$A^0=a_4 \omega_4 \otimes T_1 + a_1 \left( \omega_1 \otimes T_2 +  \omega_5 \otimes T_3 \right),$$
where the $a_i$ satisfy
\begin{align*}
&a_1^2= \frac{-\Delta}{4B^2C^2}\left(  1 + \frac{AD}{BC}  \right) \\
&\ a_4= - \frac{1}{\sqrt{2}} \left(1 + \frac{AD}{BC} \right).
\end{align*}

\item $n=3$ and $\sigma_2(\varphi)= 3 \left( \frac{1}{2} - \frac{BD}{AC} \right) \Delta + \frac{3}{2}\Gamma > 0$, in which case $a_1=a_5=0,$ and
\begin{align*}
& a_2^2 = \frac{1}{12A^2C^2} \left( 3 \left( \frac{1}{2} - \frac{BD}{AC} \right) \Delta + \frac{3}{2}\Gamma \right) \\
&\ b=\frac{1}{2}\left( \frac{1}{2} - \frac{BD}{AC} \right).
\end{align*}

\item $n=-3$ and $\sigma_3(\varphi)=  3 \left( \frac{1}{2} - \frac{CD}{AB} \right) \Delta - \frac{3}{2} \Gamma > 0$, in which case $a_1=a_5=0,$ and
\begin{align*}
& a_3^2 = \frac{1}{12A^2B^2} \left(  3 \left( \frac{1}{2} - \frac{CD}{AB} \right) \Delta - \frac{3}{2} \Gamma \right) \\
&\ b=\frac{1}{2}\left( \frac{1}{2} - \frac{CD}{AB} \right).
\end{align*}
\end{enumerate}
\end{theorem}
\begin{proof}
We follow the same strategy as in the proof of theorem \ref{thm:IrreducibleInstantons}, splitting into the cases described above. \\

\textbf{Case 0:} $n \neq 0, 3, -3$. Here $A^n$ is always reducible so there cannot be an invariant, irreducible instanton. We note that the reducible $G_2$-instantons arising from this case are exactly those appearing in theorem \ref{thm:abelianinstantonsX11}. \\

\textbf{Cases 2 and 3:} $n=3, -3$. These cases can be handled in the same way as the second and third items in theorem \ref{thm:IrreducibleInstantons}, so we omit the details. \\

\textbf{Case 1:} $n=0$. Here any invariant connection is simply a left, and $\Ad(U(1)_{1,1})$ invariant 1-form, with values in $\mathfrak{so}(3)$. We write it as
$$A^0 =   \omega_1 \otimes c_1 + \omega_4 \otimes c_4 + \omega_5 \otimes c_5,$$
where $c_1, c_4, c_5 \in \mathfrak{so}(3)$. We compute the curvature of this connection using the formula $F^0=dA^0+\frac{1}{2}[A^0 \wedge A^0]$. This gives
\begin{align*}
F^0&=d\omega_1 \otimes c_1 + d\omega_4 \otimes c_4 + d\omega_5 \otimes c_5 \\
&\ +\omega_{14} \otimes [c_1,c_4] + \omega_{15} \otimes [c_1,c_5] + \omega_{45} \otimes [c_4,c_5].
\end{align*}
The equation $F^0 \wedge \psi=0$, after a small amount of simplification, yields
\begin{align*}
&  BC[c_1,c_4]=\sqrt{2} \left( AD+BC \right) c_5 \\
&\ BC[c_4,c_5]=\sqrt{2} \left( AD+BC \right) c_1 \\
&\sqrt{2}\ B^2C^2[c_1,c_5]= \Delta c_4.
\end{align*}

Bracketing the third equation with $c_4$ gives us $ \left[ \left[ c_1,c_5 \right], c_4 \right] =0$. We first assume that $[c_1,c_5]\neq 0$, which by the third equation implies $c_4 \neq 0$. This being the case, we may change gauge to require that
$$c_1 = r_1 T _2, c_4=r_4 T_1, c_5=r_5 T_3$$ for some nonzero real constants $r_1, r_4, r_5.$ With this choice, the system becomes
\begin{align*}
&  -2BCr_1r_4=\sqrt{2} \left( AD+BC \right) r_5 \\
&\ -2BCr_4r_5=\sqrt{2} \left( AD+BC \right) r_1 \\
&\ 2\sqrt{2}B^2C^2r_1r_5= \Delta r_4.
\end{align*}
since we have assumed that the $r_i$ are nonzero, we must have $\Delta \neq 0$ and $AD+BC \neq 0$. The solutions to these equations are readily found to be
\begin{align*}
&r_1^2= \frac{-\Delta}{4B^2C^2}\left(  1 + \frac{AD}{BC}  \right) \\
&\ r_5 = \pm r_1 \\
&\ r_4= \mp \frac{1}{\sqrt{2}} \left(1 + \frac{AD}{BC} \right),
\end{align*}
This seems to yield four solutions, provided $$-\frac{\Delta}{2\sqrt{2}} \left( 1 + \frac{AD}{BC} \right) > 0,$$
however the solutions differing only by the $\pm$ sign are gauge equivalent: we can change gauge to send $T_1$ to $-T_1,$ and $T_3$ to $-T_3.$ At this point we set $a_1=r_1$ and $a_4=r_4$ yielding the result in the statement.

If $[c_1,c_5]=0$ then we may by change of gauge fix $c_1=\lambda_1 T_1, c_5=\lambda_5 T_1$ for some (possibly zero) constants $\lambda_1, \lambda_5.$ Then, considering the first equation $BC[c_1,c_4]=\sqrt{2} \left( AD+BC \right) c_5,$ we must have $[c_1,c_4]=0.$ Therefore the connection is reducible, and the solutions will correspond to Abelian $G_2$-instantons already described in theorem \ref{thm:AbelianInstantons}.
\end{proof}

With exactly the same method as in theorem \ref{thm:G2_Instanton_Obstructed_(1,1)} we can prove that when the $G_2$-instantons merge they become reducible and obstructed.

\begin{theorem}\label{thm:G2_Instanton_Obstructed_(1,1)}
Let $\lbrace \varphi(s) \rbrace_{s \in \mathbb{R}}$ be a continuous family of $G_2$-structures as in \ref{eq:G2str} such that $\sigma_1(\varphi(s)) >0$, for $s<0$ and $\sigma_1(\varphi(s))<0$, for $s>0$. Then as $s \nearrow 0$ the two irreducible $G_2$-instantons on $P_{\lambda_0}$ from theorem \ref{thm:IrreducibleInstantons} merge and become the same reducible and obstructed $G_2$-instanton when they disappear for $s\leq 0$.
\end{theorem}

\begin{remark}
A similar statement holds for the $G_2$-instantons on $P_{\lambda_{\pm}}$ with $\sigma_1$ replaced by $\sigma_2$ and $\sigma_3$ respectively.
\end{remark}

\subsection{An example of merging $G_2$-instantons on $X_{1,1}$}\label{sss:X11}

We may think of $\pi_1 :X_{1,1} \rightarrow \overline{\mathbb{CP}}^2$ as in remark \ref{rem:Weyl}, i.e. as an $SO(3)$-bundle over $\overline{\mathbb{CP}}^2$, which is a quaternion-K\"ahler $4$-manifold (anti-self-dual, Einstein) with positive Einstein constant. The discussion before proposition \ref{prop:G2InstFromQK}, in section \ref{sec:NearlyParallel}, shows that $X_{1,1}$ carries two nearly parallel $G_2$-structures. One inducing a tri-Saskian metric and other inducing a strictly nearly parallel one. This last one will be contained in the family of $G_2$-structures we consider in this section. Proposition \ref{prop:G2InstFromQK} gives some examples of $G_2$-instantons on $X_{1,1}$ by pulling back anti-self-dual connections on $\overline{\mathbb{CP}}^2$. In fact, on any line bundle over $\overline{\mathbb{CP}}^2$ there is one such connection that is $SU(3)$-invariant, namely the canonical connection $-\frac{n}{2}\frac{1}{\sqrt{6}} h$ on the degree $n$-bundle. In what follows we shall confirm this fact and we will also obtain other examples of $G_2$-instantons that are not pulled back from $\overline{\mathbb{CP}}^2$.\\
In this subsection we will consider the $G_2$-structures in the family \ref{eq:G2str} that satisfy $C=B$ and $D=A$. This is, up to scaling, the $1$-parameter family in the hypothesis of proposition \ref{prop:G2InstFromQK} with $t$ proportional to $A/B$. For completeness we note that the $G_2$-structure in equation \ref{eq:G2str} gives
\begin{eqnarray}\nonumber
\psi & = & B^4 \left( \omega_{2367} - \frac{A^2}{B^2} \left( \omega_{51} \wedge \Omega_1 + \omega_{45} \wedge \Omega_2 - \omega_{14} \wedge \Omega_3 \right)   \right),
\end{eqnarray} 
where $\Omega_1 = \omega_{26}+\omega_{37}$, $\Omega_2=\omega_{23}+\omega_{76}$ and $\Omega_3=\omega_{27}+\omega_{63}$ form an orthonormal basis for the pullback of the space of self-dual $2$-forms on $\overline{\mathbb{CP}}^2$.
One can then chack that this family contains one of the homogeneous nearly parallel $G_2$-structure on $X_{1,1}$. In fact, one can check that $A = 2 \sqrt{2} / \lambda$ and $B=2/\lambda$ satisfy $d \varphi= \lambda \psi$.\\
For the structures we are considering,
\begin{align*}
&AD+BC=A^2 + B^2 \neq 0, \\
&\ \Delta = 2 B^2 (A^2-B^2) , \\
&\ \Gamma =0,
\end{align*}
and thus theorem \ref{thm:abelianinstantonsX11} tells us that for $\Delta \neq 0$, i.e. $A^2 \neq B^2$, there is a unique $G_2$-instanton on $Q_n$, this has $b=0$ and so is precisely the canonical invariant connection $-\frac{n}{2 \sqrt{6}} h$. Its curvature is
$$-\frac{n}{2\sqrt{6}} dh = -\frac{n}{4} (\omega_{26} - \omega_{37}),$$
and as remarked before, is actually the pullback from $\overline{\mathbb{CP}^2}$ of self dual 2-form. On the other hand, the first bullet in the same theorem shows that when $A^2 = B^2$ there is a one parameter family of $G_2$-instantons, namely any of the connections $-\frac{n}{2 \sqrt{6}} h + b \omega_4$, for $b \in \mathbb{R}$. We state these conclusions as
%

\begin{theorem}\label{thm:X11}
Let $A, B \in \mathbb{R}^+$ and equip $X_{1,1}= SU(3)/U(1)_{1,1}$ with the $G_2$-structure
$$\varphi_{A,B}=  A^3 \omega_{145} + AB^2(\omega_{123}-\omega_{167}+\omega_{257}-\omega_{356} - \omega_{426} - \omega_{437}).$$
If $L$ is a complex line bundle over $X_{1,1}$ with $c_1(L)=n \in \mathbb{Z} \cong H^2(X_{1,1}, \mathbb{Z})$, then:
\begin{itemize}
\item If $A^2 \neq B^2$ the canonical connection $-\frac{n}{2 \sqrt{6}} h$ is the unique invariant $G_2$-instanton on $L$.
\item If $A^2=B^2$, then the connections $-\frac{n}{2 \sqrt{6}} h + b \omega_4$ are $G_2$-instantons for any $b \in \mathbb{R}$, these being the unique invariant $G_2$-instantons on $L$.
\end{itemize} 
\end{theorem}

\begin{remark}\label{rem:X11}
\begin{enumerate}
\item The canonical connection $-\frac{n}{2\sqrt{6}}h$ is the pullback of an anti-self-dual connection on $\overline{\mathbb{CP}^2}$. Therefore, the fact that it is a $G_2$-instanton with respect to $\varphi_{A,B}$ also follows from proposition \ref{prop:G2InstFromQK}. Its uniqueness for the nearly parallel structure is also a consequence of corollary \ref{cor:G2InstLineBundle}, however uniqueness amongst invariant ones for other structures in the family $\lbrace \varphi_{A,B} \rbrace_{A \neq B}$ is not.

\item The Abelian instantons constructed for $A=B$ show that the uniqueness part of corollary \ref{cor:G2InstLineBundle} does not extend from nearly parallel to general coclosed $G_2$-structures. In fact, not even the rigidity stated in corollary \ref{cor:Rigid} holds.
\end{enumerate}
\end{remark}

We turn now to invariant, irreducible, non-Abelian $G_2$-instantons. We start with case $n=k-l=0$. Theorem \ref{thm:irreducibleinstantonsX11} tells us that $G_2$-instantons on $P_0$ exist if and only if
$$-2 B^2 (A^2-B^2) \left( 1 + \frac{A^2}{B^2} \right) > 0,$$
or in other words if and only if $B^2>A^2$. In this case we have 
$$A^0=a_4 \omega_4 \otimes T_1 + a_1 \left( \omega_1 \otimes T_2 + \omega_5 \otimes T_3 \right),$$ where the $a_i$ must satisfy
\begin{align*}
& a_1^2= \frac{B^4-A^4}{2B^4}  \\
&\ a_4= - \frac{A^2+B^2}{\sqrt{2}B^2}.
\end{align*}
The curvature  of these connections is 
\begin{equation}\nonumber
F=F_1 \otimes T_1 + F_2 \otimes T_2 + F_3 \otimes T_3,
\end{equation}
with
\begin{eqnarray}\label{eq:CurvatureX_11.1.1}
F_1 & = & - \left( \frac{A^2}{B^2} + 1 \right) \left( \frac{A^2}{B^2} \omega_{15} - \frac{1}{2} (\omega_{26} - \omega_{73}) \right) \\ \label{eq:CurvatureX_11.1.2}
F_2 & = & \pm \sqrt{1- \frac{A^4}{B^4}} \left( \frac{A^2}{B^2}  \omega_{45} - \frac{1}{2} (\omega_{23} - \omega_{67} ) \right) \\ \label{eq:CurvatureX_11.1.3}
F_3 & = & \pm \sqrt{1- \frac{A^4}{B^4}} \left( \frac{A^2}{B^2}  \omega_{14} - \frac{1}{2} ( \omega_{36} - \omega_{27}) \right). 
\end{eqnarray}
The other cases when there exist nontrivial invariant connections are when $n= \pm 3$. Notice that $P_3$ and $P_{-3}$ are interchanged by the automorphism of $SU(3)$ given by $g \mapsto g^{-1}$. This automorphism preserves $U(1)_{1,1}$ and so descends to a diffeomorphism of $X_{1,1}=X_{-1,-1}$. We shall therefore consider only the case $n=3$ where $a_1=a_3=0$. Also in this case, our work above gives that there are irreducible, invariant $G_2$-instantons on $P_3$ (resp. $P_{-3}$) if and only if 
$$\sigma_2 =\sigma_3 = 3B^2(B^2-A^2) \geq 0,$$
i.e. $B^2 > A^2$. In that case we have
$$a_2  = \pm \frac{1}{2} \sqrt{-1+ \frac{B^2}{A^2}} , \ a_4 = - \frac{1}{2 \sqrt{2}},$$
and their curvature is such that
\begin{eqnarray}\label{eq:CurvatureX_11.2.1}
F_1 & = &  -\frac{1}{2} \omega_{15} - \left(1-\frac{B^2}{2A^2} \right) \omega_{26} + \omega_{37}  \\ \label{eq:CurvatureX_11.2.2}
F_2 & = & \pm \frac{1}{ \sqrt{2}} \sqrt{-1+ \frac{B^2}{A^2}} \left(  \omega_{46} + \frac{1}{2} \left(  \omega_{13} - \omega_{57} \right) \right) \\ \label{eq:CurvatureX_11.2.3}
F_3 & = & \mp \frac{1}{ \sqrt{2}} \sqrt{-1 + \frac{B^2}{A^2}} \left( - \omega_{24} + \frac{1}{2} \left(  \omega_{17} - \omega_{35} \right) \right).  
\end{eqnarray}
As before these are clearly irreducible and not pulled back from $\overline{\mathbb{CP}^2}$ via $\pi$. We have thus proved

\begin{theorem}\label{thm:X11Irreducible}
For $A, B \in \mathbb{R}^+$ let $\varphi_{A,B}$ be the $G_2$-structure on $X_{1,1}= SU(3)/U(1)_{1,1}$ from theorem \ref{thm:X11}. Let $\nabla_A$ be an $SU(3)$-invariant, irreducible $G_2$-instanton for $\varphi_{A,B}$, with gauge group $SO(3)$. Then, either
\begin{enumerate}
\item $\nabla_A$ lives on $P_{0}$, the trivial $SO(3)$-bundle over $X_{1,1}$, in which case the following holds:
\begin{itemize}
\item If $A < B$, then $\nabla_A$ is one of two $G_2$-instanton on $P_{0}$, having curvature as in equations \ref{eq:CurvatureX_11.1.1}--\ref{eq:CurvatureX_11.1.3}.
\item If $A \geq B$, there is no invariant, irreducible $G_2$-instanton on $P_{0}$.
\end{itemize}
\item $\nabla_A$ lives on one of the bundles $P_{3}$ or $P_{-3}$, in which case the following holds:
\begin{itemize}
\item If $A < B$, then $\nabla_A$ is one of two invariant, irreducible $G_2$-instantons on $P_{\pm 3}$. In the case $\nabla_A$ lives on $P_{3}$, its curvature is as in equations \ref{eq:CurvatureX_11.2.1}--\ref{eq:CurvatureX_11.2.3}.
\item If $A \geq B$, there is no invariant, irreducible $G_2$-instanton on either $P_{\pm 3}$.
\end{itemize}
\end{enumerate}
\end{theorem}

\begin{remark}\label{rem:X11Irreducible}
\begin{itemize}
\item Both in $P_{0}$ and $P_{3}$, the $G_2$-instantons $(\nabla_A)_{A,B}$ constructed above become Abelian when $A=B$.

\item None of the irreducible $G_2$-instantons on $P_{0}$ and $P_{3}$ constructed for $A < B$ is pulled back from $\overline{\mathbb{CP}^2}$ and so do not follow from proposition \ref{prop:G2InstFromQK}.
\end{itemize}
\end{remark}

The instantons on $P_{0}$ and $P_{3}$ constructed above are quite different. In fact, looking at the expressions for the curvature of these, we see that by metrically collapsing the fibers of $\pi : X_{1,1} \rightarrow \overline{\mathbb{CP}^2}$ by sending $A$ to $0$, the instantons constructed on $P_{0}$ converge to the pullback of a connection on $\overline{\mathbb{CP}^2}$. However, this property does not hold for those constructed on $P_{3}$. More precisely, we have

\begin{theorem}\label{thm:Limit}
Let $(\nabla_A)_{A,B}$ be the $G_2$-instanton associated with $\varphi_{A,B}$ on $P_{0}$. Then, there is a $SO(3)$-connection $\nabla$ on $\overline{\mathbb{CP}^2}$ such that as $A \rightarrow 0$, $(\nabla_A)_{A,B}$ converges uniformly with all its derivatives to $\pi^* \nabla$.\\
Let $(\tilde{\nabla_A})_{A,B}$ be the $G_2$-instanton associated with $\varphi_{A,B}$ on $P_{3}$. There is no connection $\nabla$ on $\overline{\mathbb{CP}^2}$ such that $(\nabla_A)_{A,B} \rightarrow \pi^* \nabla$ uniformly with respect to $\varphi_{1,1}$, as $A \rightarrow 0$.
\end{theorem}
\begin{proof}
Let $P= SU(3) \times_{U(2), \lambda} SO(3)$ be the bundle constructed from $\lambda: SU(2) \times U(1) / \mathbb{Z}_2 \rightarrow SO(3)$ with
$$\lambda(g, e^{i\theta})=g \mod -1.$$
The canonical invariant connection $\nabla$ associated with this bundle is $-\frac{1}{\sqrt{2}}(\omega_4 \otimes T_1 + \omega_1 \otimes T_2 + \omega_5 \otimes T_3) \in \Omega^1(SU(3), \mathfrak{so}(3))$. Its curvature is $F=T_1 \otimes T_1 + F_2 \otimes T_2 + F_3 \otimes T_3$ is such that
\begin{eqnarray}\nonumber
F_1 =  \frac{1}{2} (\omega_{26} + \omega_{37})  , \ \ F_2 = \mp \frac{1}{2} ( \omega_{23} - \omega_{67}) , \  \ F_3 = \pm  \frac{1}{2} (\omega_{27} + \omega_{63}) , 
\end{eqnarray}
and so is a self-dual connection. In fact notice that the components $F_1, F_2, F_3$ of the curvature pullback to $\Omega_1, \Omega_2, \Omega_3$ on $X_{1,1}$ respectively. We now let $(\nabla_A)_{A,B}$ be our $G_2$-instanton on $\varphi_{A,B}$ which has connection $1$-form
$$-\frac{1}{\sqrt{2}} \left( \frac{A^2}{B^2}+1 \right) \omega_4 \otimes T_1 - \frac{1}{\sqrt{2}} \sqrt{ 1- \frac{A^4}{B^4} } (\omega_1 \otimes T_2 + \omega_5 \otimes T_3) ,$$
seen as an element of $\Omega^1(SU(3), \mathfrak{so}(3))$. Hence the difference of the two connections $a_{A,B}=(\nabla_A)_{A,B} - \pi^* \nabla$ is a $dh$-horizontal $1$-form in $SU(3)$ given by
$$a= \frac{A^2}{B^2} \omega_4 \otimes T_1 + \frac{1}{\sqrt{2}} \left( \sqrt{ 1- \frac{A^4}{B^4} } -1  \right)(\omega_1 \otimes T_2 + \omega_5 \otimes T_3) .$$
Using the fixed metric associated with the $G_2$-structure $\varphi_{1,1}$ to take norms we compute that for any $k \in \mathbb{Z}^+$
$$\Vert a_{A,B} \Vert_{C^k} \leq c_k \frac{A^2}{B^2},$$
for some positive constant $c_k$ independent of $A$, $B$. Taking $A$ to $0$ we see that $a_{A,B}$ converges uniformly to $0$ with all derivatives, proving the first assertion in the statement.\\
We turn now to the proof of the second assertion. Namely, that the same phenomena cannot happen for the instantons we constructed on $P_{\lambda_3}$. If such a statement was true then the curvatures $\tilde{F}_{A,B}$ of $(\tilde{\nabla_A})_{A,B}$ should converge to an  $\mathfrak{so}(3)$-valued $2$-form on $SU(3)$ that is basic with respect to the projection $SU(3) \rightarrow \overline{\mathbb{CP}^2}$. Any linear combination $V$ of the vector fields $e_1,e_4,e_5$ is vertical with respect to this projection. Taking $V= e_1$ we have
$$\iota_{e_1} \tilde{F}_{A,B} = -\frac{1}{2} \omega_{5} \otimes T_1 \pm \frac{1}{2 \sqrt{2}} \sqrt{-1 + \frac{B^2}{A^2}} (\omega_3 \otimes T_2 - \omega_7 \otimes T_3) $$
and clearly $\lim_{A \rightarrow 0} \Vert \iota_{e_1} \tilde{F}_{A,B} \Vert_{C^k} = + \infty$ for all $k \in \mathbb{N}_0$. Hence, $\tilde{F}_{A,B}$ cannot converge to a basic form.
\end{proof}

\begin{remark}
The $SO(3)$-connection $\nabla$ on $\overline{\mathbb{CP}^2}$ appearing in the previous theorem is in fact self-dual. However, we do not want to emphasize this fact too much, as it may be misleading. Indeed, we expect that in other similar situations the same phenomena can occur with the corresponding $\nabla$ not being self-dual. 
\end{remark}

There is one other homogeneous nearly parallel $G_2$-structure on $X_{1,1}$. In fact, the equations for homogeneous nearly parallel $G_2$-structures in the case $(k,l)=(1,1)$ yield $8$ solutions, which give rise to two different metrics. The solutions are completely determined by $C^2=B^2$, $D^2=A^2$ and
\begin{itemize}
\item $A^2= 2 B^2 $ and $ABCD>0$, in which case the corresponding metric is tri-Saskian.

\item $A^2=2B^2/5$ and $ABCD<0$, and so the $G_2$-structure is obtained through the squashing construction in section \ref{sec:NearlyParallel}. In this case the corresponding metric is a proper nearly parallel $G_2$-metric, see theorem 5.5 in \cite{Friedrich1997}.
\end{itemize}
Notice that theorem \ref{thm:X11Irreducible} does not yield any irreducible $G_2$-instanton for the nearly parallel $G_2$-structure contained in the family we are analysing, which is the one inducing the tri-Saskian structure. However, as we shall now show, the strictly nearly parallel does. 

\begin{corollary}\label{cor:NoInstantons}
\begin{itemize}
\item There is no irreducible, invariant $G_2$-instantons with gauge group $SO(3)$ for the nearly parallel $G_2$-structure on $X_{1,1}$ inducing the tri-Saskian metric. 
\item There are irreducible, invariant $G_2$-instantons with gauge group $SO(3)$ for the strictly nearly parallel $G_2$-structure on $X_{1,1}$.
\end{itemize}
\end{corollary}
\begin{proof}
Any homogeneous nearly parallel $G_2$-structure on $X_{1,1}$ satisfies $A^2=D^2$ and $B^2=C^2$. Moreover, they must be such that and either
\begin{itemize}
\item $A^2= 2 B^2 $ and $ABCD>0$. In fact, for $ABCD>0$ we compute 
$$ \sigma_1(\varphi)= 6 (B^4-A^4), \ \sigma_2 (\varphi) = \sigma_3(\varphi)= 3 B^2(B^2 -A^2) . $$
As the nearly parallel $G_2$-structure in this case has $A^2= 2 B^2 > B^2$ we see that all $\sigma_i$, for $i=1,2,3$, are negative and so there are no $G_2$-instantons.
\item $A^2=2B^2/5$ and $ABCD<0$. In this case we compute that for $ABCD<0$
$$ \sigma_1(\varphi)= 6 (A^2-B^2)^2, \ \sigma_2 (\varphi) = \sigma_3(\varphi)= 9 B^2(A^2 -B^2) . $$
The nearly parallel $G_2$-structure has $A^2=2B^2/5<B^2$, so both $\sigma_2$ and $\sigma_3$ are negative. On the other hand $\sigma_1$ is positive and so there, irreducible $G_2$-instanton on this nearly parallel $G_2$-structure do exist. Any such must live in the trivial bundle $P_{\lambda_0}$.
\end{itemize}
\end{proof}

\begin{remark}
\begin{itemize}
\item The previous result shows the $G_2$-structures inducing the tri-Sasakian and the strictly nearly parallel $G_2$-structures on $X_{1,1}$ can be distinuguished by the existence of an irreducible, invariant $G_2$-instantons with gauge group $SO(3)$.

\item We further remark that we are not analyzing the most general homogeneous and coclosed $G_2$-structures on $X_{1,1}$. In fact, for $(k,l)=(1,1)$ there is a larger dimensional family, containing in particular a nearly parallel $G_2$-structure whose associated metric is Sasaki-Einstein, see \cite{Boyer1994} and \cite{Cabrera1996}.

\item In the Sasaki case, a Hermitian connection is transverse Hermitian-Yang-Mills (tHYM) if its curvature is basic with respect to the Reeb-flow and satisfies
$$F^{0,2}=0 , \ \ \Lambda_T F=0,$$
where $F^{0,2}$ is the $(0,2)$-component of the curvature with respect to the transverse complex structure and $\Lambda_T$ is the dual of wedging with the transverse K\"ahler form. Equivalently, a tHYM connection is characterized by having basic curvature and its pullback to the metric cone being HYM with respect to the K\"ahler structure induced on the cone. See \cite{Calvo2016} for relations of $G_2$ gauge theory when a Sasakian structure is present.
\end{itemize}
\end{remark}

%
%
%

\section{Questions for further work}

The following are natural questions for further work
\begin{enumerate}
\item Similar methods can be used in many other cases where homogeneous $G_2$-structures exist. Of particular interest would be the cases admitting nearly parallel $G_2$-structures, see \cite{Friedrich1997} for the classification of homogeneous nearly-parallel $G_2$-manifolds.
\item Carry on a general analysis for which $(k,l)$ do theorems \ref{thm:AbelianInstantons} and \ref{thm:IrreducibleInstantons} provide irreducible $G_2$-instantons for the nearly parallel $G_2$-structures in $X_{k,l}$. We intend to address this question in the future.
\item Compute the Crowley-Nordstr\"om invariants, \cite{Crowley2015}, for the $G_2$-structures $\varphi \in \mathcal{C}$ and check if this distinguishes the two disconnected components in $\mathcal{C}$. If that is the case, then for $k \neq l$, $l \neq m$, $m \neq k$ these invariants can be used to distinguish the two strictly nearly parallel $G_2$-structures.
\item Given a $G_2$-instanton $A$ for a $G_2$-structure on $X_{k,l}$ such that $A$ is also Yang-Mills, in which cases is $A$ stable as a Yang-Mills connection? Here, it would be interesting to understand better how the answer to this question depends on the $G_2$-structure.
\end{enumerate}

%
%

\begin{bibdiv}
\begin{biblist}

\bib{Aloff1975}{article}{
   author={Aloff, Simon},
   author={Wallach, Nolan R.},
   title={An infinite family of distinct $7$-manifolds admitting positively
   curved Riemannian structures},
   journal={Bull. Amer. Math. Soc.},
   volume={81},
   date={1975},
   pages={93--97},
   issn={0002-9904},
   review={\MR{0370624}},
}
   
\bib{Anderson2011}{article}{
  author={L.~B.~Anderson}
  author ={ J.~Gray}
  author={ A.~Lukas}
  author={B.~Ovrut}
  title={Stabilizing the Complex Structure in Heterotic Calabi-Yau Vacua},
  journal={JHEP},
  volume= { \bf{1102}, 088} 
  year={2011}
}

\bib{Anderson2010}{article}{
  author={L.~B.~Anderson}
  author ={ J.~Gray}
  author={ A.~Lukas}
  author={B.~Ovrut}
  title={Yukawa Textures From Heterotic Stability Walls},
  journal={JHEP},
  volume= { {\bf 1005}, 086} 
  year={2010}
}

\bib{Simons1979}{article}{
   author={Bourguignon, Jean-Pierre},
   author={Lawson, H. Blaine},
   author={Simons, James},
   title={Stability and gap phenomena for Yang-Mills fields},
   journal={Proc. Nat. Acad. Sci. U.S.A.},
   volume={76},
   date={1979},
   number={4},
   pages={1550--1553},
   issn={0027-8424},
   review={\MR{526178}},
}

\bib{Baum1990}{book}{
   author={Baum, Helga},
   author={Friedrich, Thomas},
   author={Grunewald, Ralf},
   author={Kath, Ines},
   title={Twistor and Killing spinors on Riemannian manifolds},
   series={Seminarberichte [Seminar Reports]},
   volume={108},
   publisher={Humboldt Universit\"at, Sektion Mathematik, Berlin},
   date={1990},
   pages={179},
   review={\MR{1084369}},
}

\bib{Boyer1994}{article}{
    AUTHOR = {Charles Boyer and Krzysztof Galicki},
     TITLE = {The geometry and topology of 3-{S}asakian manifolds,},
 journal = {Journal f\"ur die reine und angewandte Mathematik},
  volume = {455},
     PAGES = {183--220},
 PUBLISHER = {Int. Press, Boston, MA},
      year = {1994},
       URL = {http://eudml.org/doc/153658},
}


\bib{Boyer1999}{article}{
    AUTHOR = {Charles Boyer and Krzysztof Galicki},
     TITLE = {3-{S}asakian manifolds},
 BOOKTITLE = {Surveys in differential geometry: essays on {E}instein
              manifolds},
    SERIES = {Surv. Differ. Geom., VI},
     PAGES = {123--184},
 PUBLISHER = {Int. Press, Boston, MA},
      YEAR = {1999},
   MRCLASS = {53C26 (53C25)},
  MRNUMBER = {1798609},
MRREVIEWER = {Roger Bielawski},
       DOI = {10.4310/SDG.2001.v6.n1.a6},
       URL = {http://dx.doi.org/10.4310/SDG.2001.v6.n1.a6},
}



\bib{Bryant2006}{article}{
	Author = {Bryant, R.~L.},
	Booktitle = {Proceedings of {G}\"okova {G}eometry-{T}opology {C}onference 2005},
	Mrclass = {53C10 (53C29)},
	Mrnumber = {2282011 (2007k:53019)},
	Mrreviewer = {Simon G. Chiossi},
	Pages = {75--109},
	Publisher = {G\"okova Geometry/Topology Conference (GGT), G\"okova},
	Title = {{Some remarks on $\rG_2$--structures}},
	Year = {2006}}
	
\bib{Cabrera1996}{article}{
   author={Cabrera, F. M.},
   author={Monar, M. D.},
   author={Swann, A. F.},
   title={Classification of $G_2$-structures},
   journal={J. London Math. Soc. (2)},
   volume={53},
   date={1996},
   number={2},
   pages={407--416},
   issn={0024-6107},
   review={\MR{1373070}},
   doi={10.1112/jlms/53.2.407},
}

\bib{Calvo2016}{article}{
   author = {O. {Calvo-Andrade} and L.~O. {Rodr{\'{\i}}guez D{\'{\i}}az} and H.~N.~S. {Earp}},
    title = {Gauge theory and $G_2$-geometry on Calabi-Yau links},
  journal = {ArXiv e-prints},
   eprint = {1606.09271},
 keywords = {Mathematics - Differential Geometry, Mathematical Physics},
     year = {2016},
    month = {jun},
}


\bib{Charbonneau2016}{article}{
author={ Benoit Charbonneau and Derek Harland},
title={Deformations of Nearly K{\"a}hler Instantons},
journal={Communications in Mathematical Physics},
year={2016},
volume={348},
number={3},
pages={959--990},
doi={10.1007/s00220-016-2675-y},
url={http://dx.doi.org/10.1007/s00220-016-2675-y}
}



%
%

\bib{Clarke14}{article}{
    Author = {Clarke, A.},
     Title = {Instantons on the exceptional holonomy manifolds of {B}ryant and {S}alamon},
   Journal = {J. Geom. Phys.},
  FJOURNAL = {Journal of Geometry and Physics},
    Volume = {82},
      Year = {2014},
     Pages = {84--97},
      ISSN = {0393-0440},
   MRCLASS = {Preliminary Data},
  MRNUMBER = {3206642},
       DOI = {10.1016/j.geomphys.2014.04.006}
}

\bib{Corrigan1983}{article}{
	Author = {E. Corrigan and C. Devchand and  D.~B. Fairlie and J. Nuyts},
	Coden = {NUPBBO},
	Doi = {10.1016/0550-3213(83)90244-4},
	Fjournal = {Nuclear Physics. B},
	Issn = {0550-3213},
	Journal = {Nuclear Phys. B},
	Mrclass = {81E10 (53C80)},
	Mrnumber = {698892 (84i:81058)},
	Mrreviewer = {Alfred Actor},
	Number = {3},
	Pages = {452--464},
	Title = {First-order equations for gauge fields in spaces of dimension greater than four},
	Url = {http://dx.doi.org/10.1016/0550-3213(83)90244-4},
	Volume = {214},
	Year = {1983},
	Bdsk-Url-1 = {http://dx.doi.org/10.1016/0550-3213(83)90244-4}}

\bib{Crowley2015}{article}{
	Author = {Diarmuid Crowley and Johannes Nordstr\"om},
	Title = {New invariants of $G_2$-structures},
	Journal = {Geometry \& Topology},
	Year = {2015},
	Pages = {2949--2992},
	Volume = {19},
	DOI = {10.2140/gt.2015.19.2949}}

\bib{Donaldson1998}{article}{
	Author = {S.~K. Donaldson and R.~P. Thomas},
	Booktitle = {The geometric universe ({O}xford, 1996)},
	Mrclass = {57R57 (14J32 32J18 53C07 57R58 58D27)},
	Mrnumber = {MR1634503 (2000a:57085)},
	Mrreviewer = {Krzysztof Galicki},
	Owner = {thomas},
	Pages = {31--47},
	Publisher = {Oxford Univ. Press},
	Timestamp = {2009.11.23},
	Title = {Gauge theory in higher dimensions},
	Year = {1998}}
		
\bib{Donaldson2009}{incollection}{
	Author = {S.~K. Donaldson{,} E.~P.Segal},
	Booktitle = {Surveys in differential geometry. Volume XVI. Geometry of special holonomy and related topics},
	Pages = {1--41},
	Publisher = {Int. Press, Somerville, MA},
	Series = {Surv. Differ. Geom.},
	Title = {Gauge theory in higher dimensions, {II}},
	Volume = {16},
	Year = {2011}
	}
	
\bib{Fernandez1982}{article}{
	Author = {M. Fern{\'a}ndez{,} A. Gray},
	Coden = {ANLMAE},
	Doi = {10.1007/BF01760975},
	Fjournal = {Annali di Matematica Pura ed Applicata. Serie Quarta},
	Issn = {0003-4622},
	Journal = {Ann. Mat. Pura Appl. (4)},
	Mrclass = {53C25},
	Mrnumber = {696037 (84e:53056)},
	Mrreviewer = {Hidekiyo Wakakuwa},
	Pages = {19--45 (1983)},
	Title = {Riemannian manifolds with structure group {$\rG_2$}},
	Volume = {132},
	Year = {1982},
	}
	

\bib{Foscolo2016}{article}{
   author = {{Foscolo}, L.},
    title = {Deformation theory of nearly K\"ahler manifolds},
  journal = {ArXiv e-prints},
   eprint = {1601.04400},
 keywords = {Mathematics - Differential Geometry, 53C10, 53C15, 58H15},
     year = {2016},
    month = {jan},
     adsurl = {http://adsabs.harvard.edu/abs/2016arXiv160104400F},
  adsnote = {Provided by the SAO/NASA Astrophysics Data System},
}

	
\bib{Friedrich1997}{article}{
   author={Friedrich, Th.},
   author={Kath, I.},
   author={Moroianu, A.},
   author={Semmelmann, U.},
   title={On nearly parallel $G_2$-structures},
   journal={J. Geom. Phys.},
   volume={23},
   date={1997},
   number={3-4},
   pages={259--286},
   issn={0393-0440},
   review={\MR{1484591}},
   doi={10.1016/S0393-0440(97)80004-6},
}

\bib{Harland2011}{article}{
    AUTHOR = {Derek Harland and Christoph N{\"o}lle},
     TITLE = {Instantons and {K}illing spinors},
   JOURNAL = {J. High Energy Phys.},
  FJOURNAL = {Journal of High Energy Physics},
      YEAR = {2012},
    NUMBER = {3},
     PAGES = {082, front matter+37},
      ISSN = {1126-6708},
   MRCLASS = {53Cxx},
  MRNUMBER = {2980180},
}

\bib{Kreck1998}{article}{
   author={Kreck, Matthias},
   author={Stolz, Stephan},
   title={A correction on: ``Some nondiffeomorphic homeomorphic homogeneous
   $7$-manifolds with positive sectional curvature'' [J.\ Differential
   Geom.\ {\bf 33} (1991), no. 2, 465--486; MR1094466 (92d:53043)]},
   journal={J. Differential Geom.},
   volume={49},
   date={1998},
   number={1},
   pages={203--204},
   issn={0022-040X},
   review={\MR{1642125}},
}
	
%
%

\bib{Lotay2016}{article}{
   author = {J.~D. {Lotay}{,} G. {Oliveira}},
    title = {$SU(2)^2$-invariant $G\_2$-instantons},
  journal = {ArXiv e-prints},
archivePrefix = {arXiv},
   eprint = {1608.07789},
 primaryClass = {math.DG},
 keywords = {Mathematics - Differential Geometry, 53C07, 53C25},
     year = {2016},
    month = {aug},
  adsnote = {Provided by the SAO/NASA Astrophysics Data System}
}

\bib{Nagy2002}{article}{
   author={Nagy, Paul-Andi},
   title={Nearly K\"ahler geometry and Riemannian foliations},
   journal={Asian J. Math.},
   volume={6},
   date={2002},
   number={3},
   pages={481--504},
   issn={1093-6106},
   review={\MR{1946344}},
   doi={10.4310/AJM.2002.v6.n3.a5},
}

\bib{Oliveira2014}{article}{
      author={Oliveira, Goncalo},
       title={Monopoles on the Bryant-Salamon manifolds},
        date={2014},
        ISSN={0393-0440},
     journal={Journal of Geometry and Physics},
      volume={86},
      number={0},
       pages={599 \ndash  632},
}

\bib{SaEarp2015}{article}{
  AUTHOR = {T. Walpuski, H. S{\'a} Earp},
     TITLE = {{$\rm {G}\sb 2$}-instantons over twisted connected sums},
   JOURNAL = {Geom. Topol.},
  FJOURNAL = {Geometry \& Topology},
    VOLUME = {19},
      YEAR = {2015},
    NUMBER = {3},
     PAGES = {1263--1285},
      ISSN = {1465-3060},
   MRCLASS = {53C07 (53C25 53C38)},
  MRNUMBER = {3352236},
MRREVIEWER = {Ana Cristina Ferreira},
       DOI = {10.2140/gt.2015.19.1263},
       URL = {http://dx.doi.org/10.2140/gt.2015.19.1263},
}



\bib{Wang1982}{article}{
   author={Wang, McKenzie Y.},
   title={Some examples of homogeneous Einstein manifolds in dimension
   seven},
   journal={Duke Math. J.},
   volume={49},
   date={1982},
   number={1},
   pages={23--28},
   issn={0012-7094},
   review={\MR{650366}},
}


\bib{Walpuski2011}{article}{
  AUTHOR = {Walpuski, T.},
     TITLE = {{$\rm G\sb 2$}-instantons on generalised {K}ummer
              constructions},
   JOURNAL = {Geom. Topol.},
  FJOURNAL = {Geometry \& Topology},
    VOLUME = {17},
      YEAR = {2013},
    NUMBER = {4},
     PAGES = {2345--2388},
      ISSN = {1465-3060},
   MRCLASS = {53B15 (53C38 58D19)},
  MRNUMBER = {3110581},
MRREVIEWER = {Selman U{\u{g}}uz},
       DOI = {10.2140/gt.2013.17.2345},
       URL = {http://dx.doi.org/10.2140/gt.2013.17.2345},
}

\bib{Walpuski2015}{article}{
   author = {{Walpuski}, T.},
    title = {${\rm G}_2$-instantons over twisted connected sums: an example},
  journal = {ArXiv e-prints},
   eprint = {1505.01080},
 keywords = {Mathematics - Differential Geometry, Mathematics - Algebraic Geometry, 53C07, 53C25, 53C38, 14D20, 14J28},
     year = {2015},
    month = {may},
  adsnote = {Provided by the SAO/NASA Astrophysics Data System}
}

\bib{Wang1958}{article}{
   author={Wang, Hsien-chung},
   title={On invariant connections over a principal fibre bundle},
   journal={Nagoya Math. J.},
   volume={13},
   date={1958},
   pages={1--19},
   issn={0027-7630},
   review={\MR{0107276}},
}

\bib{Ziller2004}{article}{
   author={Ziller, Wolfgang},
   title={Examples of Riemannian manifolds with non-negative sectional
   curvature},
   conference={
      title={Surveys in differential geometry. Vol. XI},
   },
   book={
      series={Surv. Differ. Geom.},
      volume={11},
      publisher={Int. Press, Somerville, MA},
   },
   date={2007},
   pages={63--102},
   review={\MR{2408264}},
   doi={10.4310/SDG.2006.v11.n1.a4},
}

\end{biblist}
\end{bibdiv}

\end{document}